\documentclass[10pt,a4paper]{article}
\usepackage[utf8]{inputenc}
\usepackage[english]{babel}
\usepackage[T1]{fontenc}
\usepackage{amsmath, amsfonts, amssymb, amsthm}
\usepackage{lmodern}
\usepackage{microtype}
\usepackage{enumerate}
\usepackage{mathtools} 
\usepackage[vcentermath]{youngtab}
\usepackage{fullpage}

\usepackage{tikz}
\usetikzlibrary{matrix, arrows}
\usepackage{array}
\newcolumntype{M}[1]{>{\centering\arraybackslash}m{#1}}

\usepackage[colorlinks=true, linkcolor=blue]{hyperref}

\newtheorem{theoreme}{Theorem}[section]
\newtheorem{lemme}[theoreme]{Lemma}
\newtheorem{proposition}[theoreme]{Proposition}
\newtheorem{corollaire}[theoreme]{Corollary}

\theoremstyle{definition}
\newtheorem{definition}[theoreme]{Definition}

\theoremstyle{remark}
\newtheorem{remarque}[theoreme]{Remark}

\numberwithin{equation}{section}

\author{Salim \textsc{Rostam}}
\title{Cyclotomic Yokonuma--Hecke algebras are cyclotomic quiver Hecke algebras}
\date{}

\newcommand*{\Y}{\mathrm{Y}}
\renewcommand*{\H}{\mathrm{H}}
\newcommand*{\K}{\mathcal{K}}
\DeclareMathOperator{\nrelbar}{\,\not\!\!\!\!\;\text{---}\,}

\newcommand*{\comp}{\models}
\newcommand*{\Mat}{\mathrm{Mat}}
\newcommand*{\tuple}[1]{\boldsymbol{#1}}

\newcounter{BKproof}
\renewcommand{\theBKproof}{\Roman{BKproof}}
\stepcounter{BKproof}

\begin{document}
\maketitle

\abstract{We prove that cyclotomic Yokonuma--Hecke algebras of type A are cyclotomic quiver Hecke algebras and we give an explicit isomorphism with its inverse, using a similar result of Brundan and Kleshchev on cyclotomic Hecke algebras. The quiver we use is given by disjoint copies of cyclic quivers. We  relate this work to an isomorphism of Lusztig.}

\section*{Introduction}

Iwahori--Hecke algebras appeared first in the context of finite Chevalley groups, as centralizer algebras of the induced representation from the trivial representation on a Borel subgroup.
Since then, both their structure and their representation theory have been intensively studied.
In particular, they have been defined independently as deformations of the group algebra of finite Coxeter groups. 
Further, connections with many other objects and theories have been established (this includes the theory of quantum groups, knot theory, etc.).
Many variations and generalisations of the ``classical'' Iwahori--Hecke algebras have yet been defined. Among these, the following ones will catch our interest in this paper: Ariki--Koike algebras, Yokonuma--Hecke algebras and finally cyclotomic quiver Hecke algebras.

In their seminal paper, Ariki and Koike  \cite{ArKo} introduced and studied generalisations of Iwahori--Hecke algebras of type A and B: the so called Ariki--Koike algebras. It turns out that these algebras can be seen as cyclotomic quotients of affine Hecke algebras of type A, and also as deformations of the group algebra of the complex reflection group $G(\ell, 1, n)$. Such deformations, in the general case of complex reflection groups, have been defined by Broué, Malle and Rouquier \cite{BMR}: in that sense, Ariki--Koike algebras are the Hecke algebras associated with $G(\ell, 1, n)$.

One of the most important results on the representation theory of Ariki--Koike algebras is \emph{Ariki's categorification theorem} \cite{Ar} (proving a conjecture of Lascoux, Leclerc and Thibon \cite{LLT}). This result implies that the decomposition matrices of such algebras can be computed using the canonical bases for quantum groups in affine type A. Partially motivated by this work, Khovanov and Lauda \cite{KhLau1, KhLau2} and Rouquier \cite{Rou} have independently defined the same algebra, known as \emph{quiver Hecke algebra} or \textit{KLR algebra}, in order to categorify quantum groups. In fact, they have shown that we have the following algebra isomorphism:
\[
U_{\mathbb{A}}^-(\mathfrak{g}) \simeq [\mathrm{Proj}(\H)] = \bigoplus_{\beta \in Q^+} [\mathrm{Proj}(\H_{\beta})]
\]
where $U_{\mathbb{A}}^-(\mathfrak{g})$ is the integral form of the negative half of the quantum group $U_q(\mathfrak{g})$ associated with a symmetrizable Cartan datum, with $\mathbb{A} = \mathbb{Z}[q, q^{-1}]$, the set $Q^+$ is the positive root lattice associated with the Cartan datum, the algebra $\H = \oplus_{\beta \in Q^+} \H_{\beta}$ is the quiver Hecke algebra corresponding to this Cartan datum and $[\mathrm{Proj}(\H)]$ is the Grothendieck group of the additive category of finitely generated graded projective $\H$-modules. A cyclotomic version of this theorem was conjectured in \cite{KhLau1}, the \emph{cyclotomic categorification conjecture}, which was later proved by Kang and Kashiwara \cite{KanKa}. More specifically, for each dominant weight $\tuple{\Lambda}$ the algebra $\H$ has a cyclotomic quotient $\H^{\tuple{\Lambda}}$ which categorifies the corresponding highest weight module $V(\tuple{\Lambda})$.

A big step towards understanding these cyclotomic quiver Hecke algebras was made by Brundan and Kleshchev \cite{BrKl} and independentely by Rouquier \cite{Rou}. The first two authors proved that cyclotomic Hecke algebras of type A are particular cases of cyclotomic quiver Hecke algebras; a similar result in the affine case has also been proved by Rouquier.
Brundan and Kleshchev also noticed that the cyclotomic Hecke algebra inherits the natural grading of the cyclotomic quiver Hecke algebra, whose grading allows in particular to study the graded representation theory of cyclotomic Hecke algebras (see for example \cite{BrKl2}). Moreover, they established a connection between the cyclotomic categorification theorem in type A for quiver Hecke algebras and Ariki's categorification theorem.

On the other hand, Yokonuma \cite{Yo} defined the Yokonuma--Hecke algebras in the study  of finite Chevalley groups: they arise once again as centralizer algebras of the induced representation from the trivial representation, but now on a maximal unipotent subgroup. Their natural presentation in type A has been transformed since (see \cite{Ju1, Ju2, JuKa, ChPA}), and the one we use here is given in \cite{ChPou}.
Similarly to Ariki--Koike algebras, Yokonuma--Hecke algebras of type A can be viewed as deformations of the group algebra of $G(d, 1, n)$. This deformation, unlike in the Ariki--Koike case, ``respects'' the wreath product structure $G(d, 1, n) \simeq (\mathbb{Z} / d\mathbb{Z})\wr \mathfrak{S}_n$.
The representation theory of Yokonuma--Hecke algebras has been first studied by Thiem \cite{Th1, Th2, Th3}, while a combinatorial approach to this representation theory in type A has been given in \cite{ChPA, ChPA2}. In this latter paper \cite{ChPA2}, Chlouveraki and Poulain d'Andecy introduced and studied generalisations of these algebras: the affine Yokonuma--Hecke algebras and their cyclotomic quotients, which generalise affine Hecke algebras of type A and Ariki--Koike algebras respectively. The interest in Yokonuma--Hecke algebras has grown recently: in \cite{CJKL}, the authors  defined a link invariant from Yokonuma--Hecke algebras which is stronger than the famous ones (such as the HOMFLYPT polynomial) obtained from classical Iwahori--Hecke algebras of type A and Ariki--Koike algebras.


The first aim of this paper is to show that cyclotomic Yokonuma--Hecke algebras are particular cases of cyclotomic quiver Hecke algebras, generalising thus the results of Brundan and Kleshchev \cite{BrKl}; our goal will be achieved in Section~\ref{section:isomorphism} with Theorem~\ref{theorem:main}. In fact, every known result on the cyclotomic quiver Hecke algebra can be applied to the cyclotomic Yokonuma--Hecke algebra: this includes the cyclotomic categorification theorem and the existence of a graded representation theory. In order to prove the main result Theorem~\ref{theorem:main}, our strategy is to define inverse algebra homomorphisms, by constructing the images of the defining generators of the corresponding algebras; we proceed as in \cite{BrKl}. In particular, we give an argument to avoid doing the calculations of \cite{BrKl} again, see for instance Remark~\ref{remark:brundan_kleshchev_ja=ja+1}. In Section~\ref{section:degenerate}, as in \cite{BrKl} we consider the degenerate case: we will define the degenerate cyclotomic Yokonuma--Hecke algebras and show that they are cyclotomic quiver Hecke algebras as well (Theorem~\ref{theorem:main_degenerate}). Finally, in Section~\ref{section:disjoint_guiver} we relate our results to an isomorphism obtained in \cite{Lu, JaPA, PA}. To that end, we first prove a general result on (cyclotomic) quiver Hecke algebras, when the quiver is the disjoint union of full subquivers. Although similar situations have already been studied in the literature (see for instance \cite[Theorem 3.15]{SVV} or \cite[Lemma 5.33]{RSVV}), the result we obtain in our context seems to be new and of independent interest.

We give now a brief overview of this article.
Given a base field $F$ and $d, n \in \mathbb{N}^*$, we first define in Section~\ref{section:setting} the cyclotomic Yokonuma--Hecke algebra $\widehat{\Y}_{d, n}^{\tuple{\Lambda}}(q)$ where $q \in F \setminus\{0, 1\}$ has order $e \in \mathbb{N}_{\geq 2} \cup \{\infty\}$ in $F^\times$ and $\tuple{\Lambda}$ is a finitely-supported $e$-tuple of non-negative integers. We also define the quiver Hecke algebra $\H_\alpha(Q)$ in full generality, where $\alpha$ is a composition of $n$ indexed by a set $\K$ and $Q = (Q_{k, k'})_{k, k' \in \K}$ is a matrix satisfying some properties. Considering particular cases for the matrix $Q = (Q_{k, k'})$, we define the cyclotomic quiver Hecke algebra $\H_{\alpha}^{\tuple{\Lambda}}(\Gamma)$ where $\tuple{\Lambda}$ is now a finitely-supported tuple indexed by $\K$ and $\Gamma$ is a loop-free quiver without any multiple edge; in particular, with the exception of Section~\ref{section:disjoint_guiver} we consider the case where $\tuple{\Lambda}$ is given by $d$ copies of the previous $e$-tuple $\tuple{\Lambda}$ and $\Gamma$ is the quiver given by $d$ disjoint copies of the (cyclic) quiver $\Gamma_e$ with $e$ vertices used in \cite{BrKl}.
We begin Section~\ref{section:quiver_hecke_generators} by considering in $\widehat{\Y}_{d, n}^{\tuple{\Lambda}}(q)$ a natural system $\{e(\alpha)\}_{\alpha \comp_{ed} n}$ of pairwise orthogonal central idempotents. Then, we define the ``quiver Hecke generators'' of $\widehat{\Y}_{\alpha}^{\tuple{\Lambda}}(q) \coloneqq e(\alpha) \Y_{d, n}^{\tuple{\Lambda}}(q)$ and we check that they verify the defining relations of $\H_{\alpha}^{\tuple{\Lambda}}(\Gamma)$. In Section~\ref{section:yokonuma_hecke_generators} we define the ``Yokonuma--Hecke generators'' of $\H_{\alpha}^{\tuple{\Lambda}}(\Gamma)$ and again check the corresponding defining relations. We conclude the proof of the main theorem in Section~\ref{section:isomorphism} by showing that we have defined inverse algebra homomorphisms. We justify in Section~\ref{section:degenerate} that the isomorphism of Theorem~\ref{theorem:main} remains true for the \emph{degenerate} cyclotomic Yokonuma--Hecke algebra $\widehat{\Y}_{d, n}^{\tuple{\Lambda}}(1)$ that we define in \textsection\ref{subsection:degenerate_cyclotomic_algebra}.
We end the section with Corollary~\ref{corollary:link_nondegenerate_degenerate}, which states that, under some conditions, the algebras $\widehat{\Y}_{d, n}^{\tuple{\Lambda}}(q)$ and $\widehat{\Y}_{d, n}^{\tuple{\Lambda}}(1)$ are isomorphic. Finally, we begin Section~\ref{section:disjoint_guiver} by some quick calculations about the minimal length representatives of the cosets of a Young subgroup in the symmetric group on $n$ letters $\mathfrak{S}_n$. The main results of the section are given in Theorems~\ref{theorem:disjoint_guiver} and \ref{theorem:disjoint_guiver_cyclotomic}, where we prove an isomorphism about (cyclotomic) ``disjoint quiver'' Hecke algebra. We end the paper with Theorem~\ref{theorem:commutative_diagram_JaPA}: we   show that we recover the isomorphism $\widehat{\Y}_{d, n}^{\tuple{\Lambda}}(q) \simeq \oplus_{\lambda\comp_d n} \Mat_{m_\lambda} \widehat{\H}_\lambda^{\tuple{\Lambda}}(q)$ of \cite{Lu, JaPA, PA}, where $m_\lambda \coloneqq \frac{n!}{\lambda_1 ! \cdots \lambda_d !}$ and the algebra $\widehat{\H}_\lambda^{\tuple{\Lambda}}(q)$ is a tensor product of cyclotomic Hecke algebras.

\paragraph*{Acknowledgements}
I am grateful to Maria Chlouveraki for many discussions about the proof and a careful reading of the paper, to Nicolas Jacon for his corrections and to Loïc Poulain d'Andecy for useful conversations around this work.

\section{Setting}
\label{section:setting}

Let $d, n \in \mathbb{N}^*$ and let $F$ be a field which contains a primitive $d$th root of unity $\xi$; in particular, the characteristic of $F$ does not divide $d$. We consider an element $q \in F^{\times}$ and we define $e \in \mathbb{N}^* \cup\{\infty\}$ as the smallest integer such that $1 + q + \dots + q^{e-1} = 0$. We will sometimes use the \emph{quantum characteristic} of $F$, given by:
\[
\mathrm{char}_q(F) \coloneqq
\begin{cases}
e & \text{if } e < \infty,
\\
0 & \text{if } e = \infty,
\end{cases}
\]
in particular $\mathrm{char}_1(F)$ is exactly the usual characteristic of $F$. Except in Section~\ref{section:degenerate}, the element $q$ will always be taken different from $1$.
We set $I \coloneqq \mathbb{Z}/\mathrm{char}_q(F)\mathbb{Z}, J \coloneqq \mathbb{Z} / d\mathbb{Z} \simeq \{1, \dots, d\}$; unless mentioned otherwise, we have $K \coloneqq I \times J$.
 
If $\K$ is a set, we will refer to a finitely-supported tuple of non-negative integers $\tuple{\Lambda} = (\Lambda_k)_{k \in \K} \in \mathbb{N}^{(\K)}$ as a \emph{weight}. We say that a finitely-supported tuple $\alpha = (\alpha_k)_{k \in \K} \in \mathbb{N}^{(\K)}$ is a $\K$-composition of $n$ and we write $\alpha \comp_{\K} n$ if ($\alpha_k \neq 0$ for finitely many $k \in \K$ and) $\sum_{k \in \K} \alpha_k = n$. If $\alpha \comp_{\K} n$, we denote by $\K^\alpha$ the subset of $\K^n$ formed by the elements $\tuple{k} = (k_1, \dots, k_n) \in \K^n$ such that:
\[
\forall k \in \K, \#\big\lbrace a \in \{1, \dots, n\} : k_a = k\big\rbrace = \alpha_k,
\]
that is, $(k_1, \dots, k_n) \in \K^\alpha$ if and only if for all $k \in \K$, there are exactly $\alpha_k$ integers $a \in \{1, \dots, n\}$ such that $k_a = k$.
The subsets $\K^\alpha$ are the orbits of $\K^n$ under the natural action of the symmetric group on $n$ letters $\mathfrak{S}_n$, in particular each $\K^\alpha$ is finite.
 
\subsection{Cyclotomic Yokonuma--Hecke algebras}

Let $\tuple{\Lambda} = (\Lambda_i)_{i \in I} \in \mathbb{N}^{(I)}$ be a weight; we assume that its \emph{level} $\ell(\tuple{\Lambda}) \coloneqq \sum_{i \in I} \Lambda_i$ verifies $\ell(\tuple{\Lambda}) > 0$.
The \emph{cyclotomic Yokonuma--Hecke algebra of type A}, denoted by $\widehat{\Y}_{d, n}^{\tuple{\Lambda}}(q)$, is the unitary associative $F$-algebra generated by the elements
\begin{equation}
\label{equation:generators_yokonumahecke}
g_1, \dots, g_{n-1}, t_1, \dots, t_n, X_1
\end{equation}

subject to the following relations:
\begin{align}
\label{relation:ordre_t_a}
t_a^d &= 1, \\
\label{relation:ta_tb}
t_a t_b &= t_b t_a, \\
\label{relation:t_b_g_a}
t_b g_a &= g_a t_{s_a(b)}, \\
\label{relation:ordre_ga}
g_a^2 &= q + (q-1) g_a e_a,
\\
\label{relation:ga_gb}
g_a g_b &= g_b g_a \qquad \forall |a-b| > 1,\\
\label{relation:tresse_ga}
g_{a+1} g_a g_{a+1} &= g_a g_{a+1} g_a, 
\end{align}
where $s_a$ is the transposition $(a, a+1) \in \mathfrak{S}_n$ and $e_a \coloneqq \frac{1}{d} \sum_{j \in J} t_a^j t_{a+1}^{-j}$, together with the following relations:
\begin{align}
\label{relation:X1_g1_X1_g1}
X_1 g_1 X_1 g_1 &= g_1 X_1 g_1 X_1, \\
\label{relation:X1_ga}
X_1 g_a &= g_a X_1  \qquad \forall a > 1, \\
\label{relation:X1_tb}
X_1 t_b &= t_b X_1,
\end{align}
and finally the cyclotomic one:
\begin{equation}
\label{relation:X1_cyclo}
\prod_{i \in I} (X_1 - q^i)^{\Lambda_i} = 0.
\end{equation}

Note that the presentation comes from $\cite{ChPA2}$, excepting the normalisation in \eqref{relation:ordre_ga} which was used in \cite{ChPou}. In particular, it comes from \eqref{relation:X1_cyclo} that $X_1$ is invertible in $\widehat{\Y}_{d, n}^{\tuple{\Lambda}}(q)$.
When $d = 1$, we recover the \emph{cyclotomic Hecke algebra} $\widehat{\H}_n^{\tuple{\Lambda}}(q)$ of \cite{BrKl}; it is the cyclotomic Yokonuma--Hecke algebra $\widehat{\Y}_{1, n}^{\tuple{\Lambda}}(q)$. In particular, the element $e_a$ becomes $1$. We write $g_a^{\H}$ (respectively $X_1^{\H}$) for the element $g_a$ (resp. $X_1$) when $d = 1$, that is, considered in $\widehat{\H}_n^{\tuple{\Lambda}}(q)$.

Following \cite{ChPA2}, we define inductively $X_{a+1}$ for $a \in \{1, \dots, n-1\}$ by
\begin{equation}
\label{equation:Xa+1_ga_Xa}
qX_{a+1} \coloneqq  g_a X_a g_a
\end{equation}
(note that the $q$ comes from our different normalisation in \eqref{relation:ordre_ga}). As for $X_1$, we introduce the notation $X_a^{\H}$ to denote $X_a$ in the case $d = 1$. The family $\{t_1, \dots, t_n, X_1, \dots, X_n\}$ is commutative and we have the following equalities:
\begin{align}
\label{relation:g_a_X_b}
g_a X_b &= X_b g_a \qquad \forall b \neq a, a+1, \\
\label{relation:g_a_X_a+1}
g_a X_{a+1} &= X_a g_a + (q-1) X_{a+1}e_a, \\
\label{relation:X_a+1_g_a}
X_{a+1} g_a &= g_a X_a + (q-1) X_{a+1} e_a.
\end{align}

The proof of the following result is the same as in \cite[Proposition 4.7]{ChPA2}, where we write $g_w \coloneqq g_{a_1} \cdots g_{a_r}$ for a reduced expression $w = s_{a_1} \cdots s_{a_r} \in \mathfrak{S}_n$; by Matsumoto's theorem (see for instance \cite[Theorem 1.2.2]{GePf}) the value of $g_w$ does not depend on the choice of the reduced expression, since the generators $g_a$ satisfy the same braid relations as the $s_a \in \mathfrak{S}_n$.
\begin{proposition}
\label{proposition:yokonuma-hecke_finite_dimensional}
The algebra $\widehat{\Y}_{d, n}^{\tuple{\Lambda}}(q)$ is a finite-dimensional $F$-vector space and a generating family is given by the elements $g_w X_1^{u_1}\cdots X_n^{u_n} t_1^{v_1} \cdots t_n^{v_n}$ for $w \in \mathfrak{S}_n$, $u_a \in \{0, \dots, \ell(\tuple{\Lambda}) - 1\}$ and  $v_a \in J$.
\end{proposition}

\begin{remarque}
\label{remark:basis_cyclotomic_YH}
The above family is even an $F$-basis of $\widehat{\Y}^{\tuple{\Lambda}}_{d, n}(q)$, see \cite[Theorem 4.15]{ChPA2}.
\end{remarque}

\subsection{Cyclotomic quiver Hecke algebras}
\label{subsection:quiver_hecke_algebras}

In their landmark paper \cite{KhLau1}, starting from a quiver without loop and multiple edges Khovanov and Lauda have constructed the so-called ``quiver Hecke algebra''. Independently, Rouquier \cite{Rou} made a similar construction, where the underlying object is a matrix $Q = (Q_{k, k'})_{k, k' \in \K}$ which contains the case of quivers. Here, we will first give this definition of \cite{Rou}, which generality will be used in Section~\ref{section:disjoint_guiver} only, and then specialise to the case of quivers.

\subsubsection{General definition}
\label{subsubsection:general_definition_quiver}
Let $\K$ be a set, $A$ a commutative ring, $u$ and $v$ two indeterminates and $Q = (Q_{k, k'})_{k, k' \in \K}$ a matrix satisfying the following conditions:
\begin{itemize}
\item the polynomials $Q_{k, k'} \in A[u, v]$ verify $Q_{k, k'}(u, v) = Q_{k', k}(v, u)$ for all $k, k' \in \K$;
\item we have $Q_{k, k} = 0$ for all $k \in \K$.
\end{itemize}

Let $\alpha \comp_{\K} n$. The \emph{quiver Hecke algebra} $\H_\alpha(Q)$ associated with $(Q_{k, k'})_{k, k' \in \K}$ at $\alpha$ is the unitary associative $A$-algebra with generating set
\begin{equation}
\label{equation:quiver_generators_alpha}
\{e(\tuple{k})\}_{\tuple{k} \in \K^\alpha} \cup \{y_1, \dots, y_n\} \cup \{\psi_1, \dots, \psi_{n-1}\}
\end{equation}
and the following relations:
\begin{align}
\label{relation:quiver_sum_alpha_e(i)}
\sum_{\tuple{k} \in \K^\alpha} e(\tuple{k}) &= 1,
\\
\label{relation:quiver_e(i)e(i')}
e(\tuple{k})e(\tuple{k}') &= \delta_{\tuple{k}, \tuple{k}'} e(\tuple{k}),
\\
\label{relation:quiver_y_ae(i)}
y_a e(\tuple{k}) &= e(\tuple{k}) y_a,
\\
\label{relation:quiver_psiae(i)}
\psi_a e(\tuple{k}) &= e(s_a \cdot \tuple{k}) \psi_a,
\\
\label{relation:quiver_ya_yb}
y_a y_b &= y_b y_a,
\\
\label{relation:quiver_psia_yb}
\psi_a y_b &= y_b \psi_a \quad \text{if } b \neq a, a+1,
\\
\label{relation:quiver_psia_psib}
\psi_a \psi_b &= \psi_b \psi_a \quad \text{if } |a-b| > 1,
\\
\label{relation:quiver_psia_ya+1}
\psi_a y_{a+1} e(\tuple{k}) &= \begin{cases}
(y_a \psi_a + 1)e(\tuple{k}) & \text{if } k_a = k_{a+1}, \\
y_a \psi_a e(\tuple{k}) & \text{if } k_a \neq k_{a+1},
\end{cases}
\\
\label{relation:quiver_ya+1_psia}
y_{a+1} \psi_a e(\tuple{k}) &= \begin{cases}
(\psi_a y_a + 1)e(\tuple{k}) & \text{if } k_a = k_{a+1}, \\
\psi_a y_a e(\tuple{k}) & \text{if } k_a \neq k_{a+1},
\end{cases}
\\
\label{relation:quiverQ_psia^2}
\psi_a^2 e(\tuple{k}) &= Q_{k_a, k_{a+1}}(y_a, y_{a+1}) e(\tuple{k}), 
\\
\label{relation:quiverQ_tresse}
\psi_{a+1}\psi_a \psi_{a+1}e(\tuple{k}) &= \begin{cases}
\psi_a \psi_{a+1} \psi_a e(\tuple{k}) + \frac{Q_{k_a, k_{a+1}}(y_a, y_{a+1}) - Q_{k_{a+2}, k_{a+1}}(y_{a+2}, y_{a+1})}{y_a - y_{a+2}} e(\tuple{k})
& \text{if } k_a = k_{a+2},
\\
\psi_a \psi_{a+1} \psi_a e(\tuple{k})
& \text{otherwise.}
\end{cases}
\end{align}

\begin{remarque}
Let $\tuple{k} \in \K^\alpha, a \in \{1, \dots, n - 2\}$ and let $P \coloneqq Q_{k_a, k_{a+1}}$; the relation \eqref{relation:quiverQ_tresse} for $k_a = k_{a+2}$ is:
\begin{equation}
\label{equation:relation_quiver_tresse_sans_fraction}
\psi_{a+1}\psi_a \psi_{a+1}e(\tuple{k}) = \psi_a \psi_{a+1} \psi_a e(\tuple{k}) + 
\frac{P(y_a, y_{a+1}) - P(y_{a+2}, y_{a+1})}{y_a - y_{a+2}} e(\tuple{k}).
\end{equation}
Writing $P(u, v) = \sum_{m \geq 0} u^m P_m(v)$, 
we get that the right side of \eqref{equation:relation_quiver_tresse_sans_fraction} is well-defined and is an element of $A[y_a, y_{a+1}, y_{a+2}]e(\tuple{k})$.
\end{remarque}

\begin{remarque}
The generators in \cite{Rou}  are given by $1_{\tuple{k}} \coloneqq e(\tuple{k})$, $x_{a, \tuple{k}} \coloneqq y_a e(\tuple{k})$ and $\tau_{a, \tuple{k}} \coloneqq \psi_a e(\tuple{k})$.
\end{remarque}

When the set $\K$ is finite, in a similar way we can define the quiver Hecke algebra $\H_n(Q)$ as the unitary associative $A$-algebra with generating set
\begin{equation}
\label{equation:quiver_generators_n}
\{e(\tuple{k})\}_{\tuple{k} \in \K^n} \cup \{y_1, \dots, y_n\} \cup \{\psi_1, \dots, \psi_{n-1}\}
\end{equation}
together with the same relations \eqref{relation:quiver_e(i)e(i')}--\eqref{relation:quiverQ_tresse}, where \eqref{relation:quiver_sum_alpha_e(i)} is replaced by:
\begin{equation}
\label{relation:quiver_sum_n_e(i)}
\sum_{\tuple{k} \in \K^n} e(\tuple{k}) = 1.
\end{equation}
Defining for $\alpha \comp_{\K} n$ the central idempotent $e(\alpha) \coloneqq \sum_{\tuple{k} \in \K^\alpha} e(\tuple{k}) \in \H_n(Q)$, we have:
\[
e(\alpha)\H_n(Q) \simeq \H_\alpha(Q),
\]
thus:
\begin{equation}
\label{equation:Hn(Q)_simeq_oplus_Halpha(Q)}
\H_n(Q) \simeq \bigoplus_{\alpha \comp_{\K} n} \H_\alpha(Q)
\end{equation}
(note that this equality can be seen as a definition of $\H_n(Q)$ is $\K$ is infinite).

For each $w \in \mathfrak{S}_n$, we now choose a reduced expression $w = s_{a_1} \cdots s_{a_r}$ and we set:
\begin{equation}
\label{equation:definition_psi_w}
\psi_w \coloneqq \psi_{a_1} \cdots \psi_{a_r} \in \H_n(Q).
\end{equation}
Although this reduced expression dependence differs from the usual case of Hecke algebras (for instance), we are still able to give a basis of $\H_n(Q)$. In fact, we have the following theorem (\cite[Theorem 3.7]{Rou}, \cite[Theorem 2.5]{KhLau1}).

\begin{theoreme}
\label{theorem:base_quiver}
The family $\{\psi_w y_1^{r_1} \cdots y_n^{r_n} e(\tuple{k}) : w \in \mathfrak{S}_n, r_a \in \mathbb{N}, \tuple{k} \in \K^\alpha\}$ is a basis of the free $A$-module $\H_\alpha(Q)$.
\end{theoreme}

We conclude this paragraph by introducing cyclotomic quotients of these quiver Hecke algebras; let $\alpha \comp_{\K} n$ and $\tuple{\Lambda} = (\Lambda_k)_{k \in \K} \in \mathbb{N}^{(\K)}$ be a weight. Following \cite[\textsection 4.1]{KanKa}, we define the \emph{cyclotomic} quiver Hecke algebra $\H^{\tuple{\Lambda}}_\alpha(Q)$ at $\alpha$ as the quotient of the quiver Hecke algebra $\H_\alpha(Q)$ by the following relations:
\begin{equation}
\label{relation:quiver_cyclo_KK}
\sum_{m = 0}^{\Lambda_{k_1}} c_m y_1^m e(\tuple{k}) = 0 \qquad \forall \tuple{k} \in \K^\alpha,
\end{equation}
where $c_m \in A$ with $c_{\Lambda_{k_1}} = 1$. Similarly, if $\K$ is finite we define the cyclotomic quiver Hecke algebra $\H_n^{\tuple{\Lambda}}(Q)$. 

\begin{theoreme}[\protect{\cite[Corollary 4.4]{KanKa}}]
\label{theorem:generating_family_cyclotomic_quiver}
The $A$-module $\H_\alpha^{\tuple{\Lambda}}(Q)$ is finitely generated.
\end{theoreme}


\subsubsection{Case of quivers}
\label{subsubsection:case_quivers}

With the exception of Section~\ref{section:disjoint_guiver}, throughout this paper our matrix $Q$ will always be associated with a loop-free quiver without multiple edges and with vertex set $\K$. If $\Gamma$ is such a quiver, following \cite[\textsection 3.2.4]{Rou} we associate the following matrix $(Q_{k, k'})_{k, k' \in \K}$:
\begin{equation}
\label{equation:Q_associated_quiver}
Q_{k, k'}(u, v) \coloneqq \begin{cases}
0 & \text{if } k = k', \\
1 & \text{if } k \nrelbar k', \\
v - u & \text{if } k \to k', \\
u - v & \text{if } k \leftarrow k', \\
-(u-v)^2 & \text{if } k \leftrightarrows k',
\end{cases}
\end{equation}
where:
\begin{itemize}
\item we write $k \nrelbar k'$ when $k \neq k'$ and neither $(k, k')$ or $(k', k)$ is an edge of $\Gamma$;
\item we write $k \to k'$ when $(k, k')$ is an edge of $\Gamma$ and $(k', k)$ is not;
\item we write $k \leftarrow k'$ when $(k', k)$ is an edge of $\Gamma$ and $(k, k')$ is not;
\item we write $k \leftrightarrows k'$ when both $(k, k')$ and $(k', k)$ are edges of $\Gamma$.
\end{itemize}
Moreover, we define:
\[
\H_\alpha(\Gamma) \coloneqq \H_\alpha(Q),
\]
and if $\K$ is finite we also set $\H_n(\Gamma) \coloneqq \H_n(Q)$. Note that, in the setting of \eqref{equation:Q_associated_quiver}, the defining relations \eqref{relation:quiverQ_psia^2} and \eqref{relation:quiverQ_tresse} become in $\H_\alpha(\Gamma)$:
\begin{align}
\label{relation:quiver_psia^2}
\psi_a^2 e(\tuple{k}) &= \begin{cases}
0 & \text{if } k_a = k_{a+1}, \\
e(\tuple{k}) & \text{if } k_a \nrelbar k_{a+1}, \\
(y_{a+1} - y_a)e(\tuple{k}) & \text{if } k_a \to k_{a+1}, \\
(y_a - y_{a+1})e(\tuple{k}) & \text{if } k_a \leftarrow k_{a+1}, \\
(y_{a+1} - y_a)(y_a - y_{a+1})e(\tuple{k}) & \text{if } k_a \leftrightarrows k_{a+1},
\end{cases}
\\
\label{relation:quiver_tresse}
\psi_{a+1}\psi_a \psi_{a+1}e(\tuple{k}) &= \begin{cases}
(\psi_a \psi_{a+1}\psi_a -1)e(\tuple{k}) & \text{if } k_{a+2} = k_a \to k_{a+1}, \\
(\psi_a \psi_{a+1}\psi_a +1)e(\tuple{k}) & \text{if } k_{a+2} = k_a \leftarrow k_{a+1}, \\
(\psi_a \psi_{a+1} \psi_a + 2y_{a+1} - y_a - y_{a+2})e(\tuple{k}) & \text{if } k_{a+2} = k_a \leftrightarrows k_{a+1}, \\
\psi_a \psi_{a+1} \psi_a e(\tuple{k}) & \text{otherwise.}
\end{cases}
\end{align}

\bigskip
We now give a remarkable fact about quiver Hecke algebras; its proof only requires a simple check of the different defining relations.

\begin{proposition}
\label{proposition:gradation_quiver_Hecke_algebra}
Let $\Gamma$ be a loop-free quiver without multiple edges with vertex set $\K$. The quiver Hecke algebra $\H_\alpha(\Gamma)$ is $\mathbb{Z}$-graded through:
\begin{gather*}
\deg e(\tuple{k}) = 0,
\\
\deg y_a e(\tuple{k}) = 2,
\\
\deg \psi_a e(\tuple{k}) = -c_{k_a, k_{a+1}},
\end{gather*}
where $C = (c_{k, k'})_{k, k' \in \K}$ is the \emph{Cartan matrix} of $\Gamma$, defined by:
\begin{equation}
\label{equation:Cartan_matrix}
c_{k, k'} \coloneqq \begin{cases}
2 & \text{if } k = k',
\\
0 & \text{if } k \nrelbar k',
\\
-1 & \text{if } k \to k' \text{ or } k \leftarrow k',
\\
-2 & \text{if } k \leftrightarrows k'.
\end{cases}
\end{equation}
\end{proposition}

\bigskip
We now define the quivers which will be particularly important to us; we recall that $I = \mathbb{Z}/\mathrm{char}_q(F)\mathbb{Z}$ and $J = \mathbb{Z}/d\mathbb{Z}$. We denote by $\Gamma_e$ the following quiver:
\begin{itemize}
\item the vertices are the elements of $I$;
\item for each $i \in I$ there is a directed edge from $i$ to $i + 1$.
\end{itemize}
In particular, for $i, i' \in I$:
\begin{itemize}
\item we have $i \to i'$ if and only if $i' = i + 1$ and $i \neq i' + 1$;
\item we have $i \leftarrow i'$ if and only if $i = i' + 1$ and $i' \neq i + 1$;
\item we have $i \leftrightarrows i'$ if and only if $i = i' + 1$ and $i' = i + 1$ (thus this only happens when $e = 2$);
\item we have $i \nrelbar i'$ if and only if $i \neq i', i'\pm 1$.
\end{itemize}

We give some examples in Figure~\ref{figure:gamma_e}.
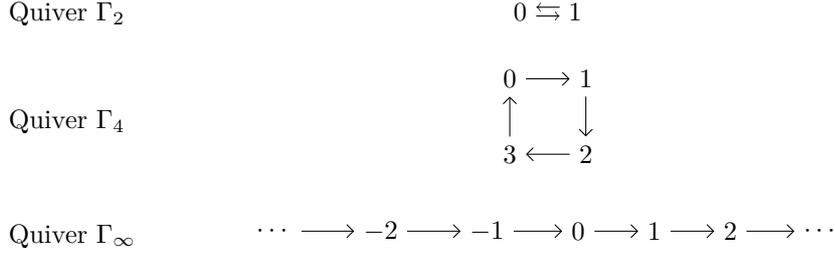
\begin{figure}[h]
\centering
\begin{tabular}{lM{10cm}}
Quiver $\Gamma_2$
&
$0 \leftrightarrows 1$
\\
\\
Quiver $\Gamma_4$
&
\begin{tikzpicture}[>=angle 90]
\node (0) at (0, 1) {$0$};
\node (1) at (1, 1) {$1$};
\node (2) at (1, 0) {$2$};
\node (3) at (0, 0) {$3$};

\draw[->] (0) -- (1);
\draw[->] (1) -- (2);
\draw[->] (2) -- (3);
\draw[->] (3) -- (0);
\end{tikzpicture}
\\
\\
Quiver $\Gamma_{\infty}$
&
\begin{tikzpicture}[>=angle 90]
\node (-3) at (-4, 0) {$\cdots$};
\node (-2) at (-2.6, 0) {$-2$};
\node (-1) at (-1.2, 0) {$-1$};
\foreach \i in {0,1,2}
	\node (\i) at (\i, 0) {$\i$};
\node (3) at (3.2, 0) {$\cdots$};
\foreach \i [count=\j from -2] in {-3,...,2}
	\draw[->] (\i) -- (\j);
\end{tikzpicture}
\end{tabular}
\caption{Three examples of quivers $\Gamma_e$}
\label{figure:gamma_e}
\end{figure}
We now define the quiver
\begin{equation}
\label{equation:gamma=union_gammae}
\Gamma \coloneqq \coprod_{j \in J} \Gamma_e
\end{equation}
given by $d$ disjoint copies of $\Gamma_e$. Hence, our quiver $\Gamma$ is described in the following way:
\begin{itemize}
\item the vertices are the elements of $\K \coloneqq K = I \times J$;
\item for each $(i, j) \in K$ there is a directed edge from $(i, j)$ to $(i+1, j)$.
\end{itemize}
In particular, there is an arrow between $(i, j)$ and $(i', j')$ in $\Gamma$ if and only if there is an arrow between $i$ and $i'$ in $\Gamma_e$ and $j = j'$. Moreover, the set $K$ is finite if and only if $e$ is finite.

We consider the diagonal action of $\mathfrak{S}_n$ on $K^n \simeq I^n \times J^n$, that is, $\sigma \cdot (\tuple{i}, \tuple{j}) \coloneqq (\sigma \cdot \tuple{i}, \sigma \cdot \tuple{j})$. We will need the following notation:
\begin{gather*}
I^{\alpha} \coloneqq \{\tuple{i} \in I^n : \exists \tuple{j} \in J^n, (\tuple{i}, \tuple{j}) \in K^{\alpha}\},
\\
J^{\alpha} \coloneqq \{\tuple{j} \in J^n : \exists \tuple{i} \in I^n, (\tuple{i}, \tuple{j}) \in K^{\alpha}\}.
\end{gather*}
The sets $I^{\alpha}$ and $J^{\alpha}$ are finite and stable under the action of $\mathfrak{S}_n$; note that $K^{\alpha}$ is included in $I^{\alpha} \times J^{\alpha}$ (we don't have the equality in general).

Let now $\tuple{\Lambda} \coloneqq (\Lambda_k)_{k \in K} \in \mathbb{N}^{(K)}$  be a weight. The cyclotomic quiver Hecke algebra $\H_\alpha^{\tuple{\Lambda}}(\Gamma)$ is given by the quotient of the quiver Hecke algebra $\H_\alpha(\Gamma)$ by the relations:
\begin{equation}
\label{relation:quiver_cyclo_y1}
y_1^{\Lambda_{k_1}} e(\tuple{k}) = 0, \qquad \forall \tuple{k} \in K^\alpha;
\end{equation}
note that this is indeed a particular case of \eqref{relation:quiver_cyclo_KK}. Note that the grading described in Proposition~\ref{proposition:gradation_quiver_Hecke_algebra} is compatible with this quotient.

\begin{remarque}
\label{remark:kl_algebra_brkl}
The \emph{cyclotomic Khovanov--Lauda algebra} of \cite{BrKl} is the quiver Hecke algebra $\H_{\alpha}^{\tuple{\Lambda}}(\Gamma_e)$, that is, the algebra $\H_{\alpha}^{\tuple{\Lambda}}(\Gamma)$ for $d = 1$. We write $e^{\H}(\tuple{i})$, $y_a^{\H}$ and $\psi_a^{\H}$ the generators of $\H_{\alpha}^{\tuple{\Lambda}}(\Gamma_e)$;  the reason for this notation will appear in \textsection\ref{subsection:definition_images_quiverhecke_generators}.
\end{remarque}

The proof of the following result is the same as in \cite[Lemma 2.1]{BrKl}.
\begin{lemme}
\label{lemme:y_a_nilpotent}
The elements $y_a \in \H_{\alpha}^{\tuple{\Lambda}}(\Gamma)$ are nilpotent for $a \in \{1, \dots, n\}$.
\end{lemme}

As a corollary, together with Theorem~\ref{theorem:base_quiver} we recover a particular case of Theorem~\ref{theorem:generating_family_cyclotomic_quiver}: in particular the algebra $\H_{\alpha}^{\tuple{\Lambda}}(\Gamma)$ is a finite-dimensional $F$-vector space.

\begin{remarque}
Let us assume $e < \infty$; the algebra $\H_n^{\tuple{\Lambda}}(\Gamma)$ is hence finite-dimensional. However, comparing with Remark~\ref{remark:basis_cyclotomic_YH}, it does not seem that easy to get its $F$-dimension; an answer will be given with Theorem~\ref{theorem:main} and \eqref{equation:main_theorem_full_algebra}.
\end{remarque}

\section{Quiver Hecke generators of \texorpdfstring{$\widehat{\Y}_{\alpha}^{\tuple{\Lambda}}(q)$}{YalphaLambda(q)}}
\label{section:quiver_hecke_generators}

Let $\tuple{\Lambda} = (\Lambda_k)_{k \in K} \in \mathbb{N}^{(K)}$ be a weight; we assume that $\ell(\tuple{\Lambda}) = \sum_{k \in K} \Lambda_k$ verifies $\ell(\tuple{\Lambda}) > 0$. Moreover, we suppose that for any $i \in I$ and $j, j' \in J$, we have:
\[
\Lambda_{i, j} = \Lambda_{i, j'} \eqqcolon \Lambda_i.
\]
In particular, we will write $\tuple{\Lambda}$ as well for the weight $(\Lambda_i)_{i \in I}$.

In this section, our first task is to define some central idempotents $e(\alpha) \in \widehat{\Y}_{d, n}^{\tuple{\Lambda}}(q)$ with $\sum_\alpha e(\alpha) = 1$ for $\alpha \comp_K n$. We will then prove the following theorem.

\begin{theoreme}
\label{theorem:morphism_quiver_yh}
For any $\alpha \comp_K n$, we can construct an explicit algebra homomorphism:
\[
\rho : \H_\alpha^{\tuple{\Lambda}}(\Gamma) \to \widehat{\Y}_\alpha^{\tuple{\Lambda}}(q),
\]
where $\widehat{\Y}_\alpha^{\tuple{\Lambda}}(q) \coloneqq e(\alpha) \widehat{\Y}_{d, n}^{\tuple{\Lambda}}(q)$.
\end{theoreme}

Note that $\widehat{\Y}_{\alpha}^{\tuple{\Lambda}}(q)$ is a unitary algebra (if not reduced to $\{0\}$), with unit $e(\alpha)$.
To define this algebra homomorphism, it suffices to define the images of the generators \eqref{equation:quiver_generators_alpha} and check that they verify the defining relations of the cyclotomic  quiver Hecke algebra: the same strategy was used by Brundan and Kleshchev in \cite{BrKl} for $d = 1$.

For a generator $g$ of $\H_{\alpha}^{\tuple{\Lambda}}(\Gamma)$, we will use as well the notation $g$ for the corresponding element that we will define in $\widehat{\Y}_{\alpha}^{\tuple{\Lambda}}(q)$. There will be no possible confusion since we will work with elements of $\widehat{\Y}_{\alpha}^{\tuple{\Lambda}}(q)$.

\subsection{Definition of the images of the generators}
\label{subsection:definition_images_quiverhecke_generators}

We define now our different ``quiver Hecke generators''.

\subsubsection{Image of \texorpdfstring{$e(\tuple{i}, \tuple{j})$}{e(i, j)}}
Let  $M$ be a finite-dimensional $\widehat{\Y}_{d, n}^{\tuple{\Lambda}}(q)$-module. 
Each $X_a$ acts on $M$ as an endomorphism of the finite-dimensional $F$-vector space (see Proposition~\ref{proposition:yokonuma-hecke_finite_dimensional}); in particular, by \eqref{relation:X1_cyclo} the eigenvalues of $X_1$ can be written $q^i$ for $i \in I$. Hence, applying \cite[Lemma 5.2]{CuWa} we know that the eigenvalues of each $X_a$ are of the form $q^i$ for  $i \in I$. Concerning the $t_a$, by $(\ref{relation:ordre_t_a})$ (they are diagonalizable and) their eigenvalues are $d$th roots of unity.

As the elements of the family $\{X_a, t_a\}_{1 \leq a \leq n}$ pairwise commute, 
using Cayley--Hamilton theorem we can write $M$ as the direct sum of its \emph{weight spaces} (simultaneous generalized eigenspaces)
\begin{equation}
\label{equation:weight_space}
M(\tuple{i}, \tuple{j}) \coloneqq \left\lbrace v \in M : (X_a - q^{i_a})^N v = (t_a - \xi^{j_a}) v = 0 \text{ for all } 1 \leq a \leq n\right\rbrace
\end{equation}
for $(\tuple{i}, \tuple{j}) \in I^n \times J^n$, where $N \gg 0$ and $\xi$ is the given primitive $d$th root of unity in $F$ that we considered at the very beginning of Section~\ref{section:setting}. Observe that some $M(\tuple{i}, \tuple{j})$ may be reduced to zero; in fact, only a finite number of them are non-zero.
\begin{remarque}
\label{remark:e_a_e(ij)}
The element $e_a$ acts on $M(\tuple{i}, \tuple{j})$ as $0$ if $j_a \neq j_{a+1}$ and as $1$ if $j_a = j_{a+1}$.
\end{remarque}

We can now consider the family of projections $\{e(\tuple{k})\}_{\tuple{k} \in K^n}$ associated with the decomposition $M = \oplus_{\tuple{k} \in K^n} M(\tuple{k})$: the element $e(\tuple{k})$ is the projection onto $M(\tuple{k})$ along $\oplus_{\tuple{k}' \neq \tuple{k}} M(\tuple{k}')$, in particular $e(\tuple{k})^2 = 0$ and if $\tuple{k} \neq \tuple{k}'$ then $e(\tuple{k}) e(\tuple{k}') = 0$. Moreover, only a finite number of $e(\tuple{k})$ are non-zero.

As the $e(\tuple{k})$ are polynomials in $X_1, \dots, X_n, t_1, \dots t_n$ (in fact $e(\tuple{k})$ is the product of commuting projections onto the corresponding generalized eigenspaces of $X_a$ and $t_a$), they belong to $\widehat{\Y}_{d, n}^{\tuple{\Lambda}}(q)$. 
\begin{remarque}
\label{remark:e(k)_independent_M}
The above polynomials do not depend on the finite-dimensional $\widehat{\Y}_{d, n}^{\tuple{\Lambda}}(q)$-module $M$.
\end{remarque}

We are now able to define our central idempotents. We set, for $\alpha \comp_K n$:
\[
e(\alpha) \coloneqq \sum_{\tuple{k} \in K^{\alpha}} e(\tuple{k})
\]
(since $K^{\alpha}$ is a $\mathfrak{S}_n$-orbit, the element $e(\alpha)$ is indeed central). Though we will not use this fact, we can notice that according to \cite[Corollary 3.2]{PA} and \cite{LyMa}, the subalgebras $\widehat{\Y}_\alpha^{\tuple{\Lambda}}(q)  \coloneqq e(\alpha)\widehat{\Y}_{d, n}^{\tuple{\Lambda}}(q)$ which are not reduced to zero are the \emph{blocks} of the Yokonuma--Hecke algebra $\widehat{\Y}_{d, n}^{\tuple{\Lambda}}(q)$; see also \cite[\textsection 6.3]{CuWa}. For $d = 1$, we recover the element $e_\alpha$ of \cite[(1.3)]{BrKl}.

We will sometimes need the following elements:
\begin{equation}
\label{equation:definition_e(i)_e(j)}
\begin{aligned}
e(\alpha)(\tuple{i}) &\coloneqq \sum_{\tuple{j} \in J^{\alpha}} e(\alpha)e(\tuple{i}, \tuple{j})
= \sum_{\substack{\tuple{j} \in J^\alpha \\ (\tuple{i}, \tuple{j}) \in K^\alpha}} e(\tuple{i}, \tuple{j}),
\\
e(\alpha)(\tuple{j}) &\coloneqq \sum_{\tuple{i} \in I^{\alpha}} e(\alpha) e(\tuple{i}, \tuple{j})
= \sum_{\substack{\tuple{i} \in I^\alpha \\ (\tuple{i}, \tuple{j}) \in K^\alpha}} e(\tuple{i}, \tuple{j}).
\end{aligned}
\end{equation}
For $d = 1$, we recover with $e(\alpha)(\tuple{i})$ the element $e(\tuple{i})$ of \cite[\textsection 4.1]{BrKl}; we denote it by $e^{\H}(\tuple{i})$.
Finally, note that:
\[
e(\alpha)(\tuple{i})\cdot e(\alpha)(\tuple{j}) = e(\alpha)(\tuple{j})\cdot e(\alpha)(\tuple{i}) = e(\tuple{i}, \tuple{j}).
\]

From now, unless  mentioned otherwise,
we always work in $\widehat{\Y}_{\alpha}^{\tuple{\Lambda}}(q)$; every relation should be multiplied by $e(\alpha)$ and we write $e(\tuple{i})$ (respectively $e(\tuple{j})$) for $e(\alpha)(\tuple{i})$ (resp. $e(\alpha)(\tuple{j})$).

We give now a few useful lemmas.

\begin{lemme}
\label{lemma:ga_Xa}
If $1 \leq a < n$ and $\tuple{j} \in J^{\alpha}$ is such that $j_a \neq j_{a+1}$ then we have:
\begin{align*}
g_a^2 e(\tuple{j}) &= q e(\tuple{j}), \\
g_a X_{a+1} e(\tuple{j}) &= X_a g_a e(\tuple{j}), \\
X_{a+1} g_a e(\tuple{j}) &= g_a X_a e(\tuple{j}).
\end{align*}
\end{lemme}

\begin{proof}
We deduce it the from relations \eqref{relation:ordre_ga}, \eqref{relation:g_a_X_a+1} and \eqref{relation:X_a+1_g_a} and from  $e_a e(\tuple{j}) = 0$ (since $j_a \neq j_{a+1}$, see Remark~\ref{remark:e_a_e(ij)}).
\end{proof}

For the next lemma, we should compare with \cite[$(15)$]{JaPA}.

\begin{lemme}
\label{lemme:g_a_e(j)}
For $1 \leq a < n$ and $\tuple{j} \in J^{\alpha}$ we have $g_a e(\tuple{j}) = e(s_a \cdot \tuple{j})g_a$.
\end{lemme}

\begin{proof}
Let $M \coloneqq \widehat{\Y}_{\alpha}^{\tuple{\Lambda}}(q)$.
 Given the relation $(\ref{relation:t_b_g_a})$, we see that $g_a$ maps $M(\tuple{j})$ to $M(s_a \cdot \tuple{j})$. Fix now $\tuple{j} \in J^\alpha$ and let $\tuple{j}' \in J^\alpha$ and $v \in M(\tuple{j}')$. If $\tuple{j}' = \tuple{j}$ then we get:
\[
g_a e(\tuple{j}) v = e(s_a \cdot \tuple{j}) g_a v \; (= g_a v),
\]
and if $\tuple{j}' \neq \tuple{j}$, since $g_a v \in M(s_a \cdot \tuple{j}')$ we have:
\[
g_a e(\tuple{j}) v = e(s_a \cdot \tuple{j})g_a v \; (= 0).
\]
Hence, $e(s_a \cdot \tuple{j})g_a$ and $g_a e(\tuple{j})$ coincide on each $M(\tuple{j}')$ for $\tuple{j}' \in J^\alpha$ thus coincide on $M = \oplus_{\tuple{j}'} M(\tuple{j}')$ and we conclude since $M = \widehat{\Y}_{\alpha}^{\tuple{\Lambda}}(q)$ is a unitary algebra.
\end{proof}

\begin{corollaire}
\label{corollaire:ga_e(j)_commutent}
Let $1 \leq a < n$ and $\tuple{j} \in J^{\alpha}$; if $j_a = j_{a+1}$ then $g_a$ and $e(\tuple{j})$ commute.
\end{corollaire}
%

\begin{remarque}[About Brundan and Kleshchev's proof - \theBKproof]
\stepcounter{BKproof}
\label{remark:brundan_kleshchev_ja=ja+1}
Let $1 \leq a < n$; if $\tuple{j} \in J^{\alpha}$ verifies $j_a = j_{a+1}$, the following relations are verified between $g_a e(\tuple{j})$ and the $X_b e(\tuple{j})$ for $1 \leq b \leq n$ in the unitary algebra $e(\tuple{j})\widehat{\Y}_{\alpha}^{\tuple{\Lambda}}(q)e(\tuple{j})$:
\begin{gather*}
g_a^2 = q  + (q-1)g_a,
\\
q X_{a+1} = g_a X_a g_a,
\\
\phantom{\forall b \neq a, a+1\qquad} g_a  X_b = X_b g_a \qquad \forall b \neq a, a+1,
\\
g_a X_{a+1} = X_a  g_a + (q-1)X_{a+1},
\\
X_{a+1} g_a = g_a X_a + (q-1)X_{a+1}.
\end{gather*}
These are exactly the relations \eqref{relation:ordre_ga}, \eqref{equation:Xa+1_ga_Xa}--\eqref{relation:X_a+1_g_a} for $\widehat{\H}_n^{\tuple{\Lambda}}(q)$ (that is, for $d = 1$).
Hence, when these only elements, together with the $e^{\H}(\tuple{i})$ for any $\tuple{i} \in I^{\alpha}$, and these only relations, together with those involving the $e^{\H}(\tuple{i})$, are used to prove any relation $(*)$ in \cite[\textsection 4]{BrKl} (in $\widehat{\H}_{\alpha}^{\tuple{\Lambda}}(q)$), we claim that the \emph{same proof} in $\widehat{\Y}_{\alpha}^{\tuple{\Lambda}}(q)$ holds for $(*)$, involving $g_a e(\tuple{j})$ instead of $g_a^{\H}$, the elements $X_b^{\pm 1}e(\tuple{j})$ instead of ${X_b^{\H}}^{\pm 1}$ and $e(\tuple{i}, \tuple{j})$ instead of $e^{\H}(\tuple{i})$. If $\tuple{j}$ verifies in addition $j_{a+1} = j_{a+2}$, we can add to the previous list the element $g_{a+1} e(\tuple{j})$, which stands for $g_{a+1}^{\H}$.

\end{remarque}

\begin{lemme}
\label{lemma:ga_e(i,j)}
For $1 \leq a < n$ and $(\tuple{i}, \tuple{j}) \in K^{\alpha}$ such that $j_a \neq j_{a+1}$ we have $g_a e(\tuple{i}) e(\tuple{j}) = e(s_a \cdot \tuple{i}) g_a e(\tuple{j})$, that is,  $g_a e(\tuple{i}, \tuple{j}) = e(s_a \cdot (\tuple{i}, \tuple{j}))g_a$.
\end{lemme}

\begin{proof}
Given Lemma~\ref{lemma:ga_Xa}, we show as in Lemma~\ref{lemme:g_a_e(j)} that $g_a e(\tuple{i}) e(\tuple{j}) = e(s_a \cdot \tuple{i}) g_a e(\tuple{j})$; we get the final result applying Lemma~\ref{lemme:g_a_e(j)}.
\end{proof}

\subsubsection{Image of \texorpdfstring{$y_a$}{ya}}

We are now able to define the elements $y_a$ for $1 \leq a \leq n$. We saw in Lemma~\ref{lemme:y_a_nilpotent} that these elements (in $\H_{\alpha}^{\tuple{\Lambda}}(\Gamma)$) are nilpotent, hence we define the following elements of $\widehat{\Y}_{\alpha}^{\tuple{\Lambda}}(q)$ for $1 \leq a \leq n$:
\[
y_a \coloneqq \sum_{\tuple{i} \in I^{\alpha}} (1 - q^{-i_a}X_a)e(\tuple{i}) \in \widehat{\Y}_\alpha^{\tuple{\Lambda}}(q).
\]
We can notice that $\sum_{\tuple{i}} (q^{i_a} - X_a) e(\tuple{i})$ is the nilpotent part of the Jordan--Chevalley decomposition of $X_a$; in particular, $y_a$ is nilpotent. As a consequence, we will be able to make calculations in the ring $F[[y_1, \dots, y_n]]$ of power series in the commuting variables $y_1, \dots, y_n$. We will sometimes also need the following element:
\[
y_a(\tuple{i}) \coloneqq q^{i_a}(1 - y_a),
\]
which verifies:
\begin{equation}
\label{equation:ya(i)e(i)=Xae(i)}
y_a(\tuple{i}) e(\tuple{i}) = X_a e(\tuple{i}).
\end{equation}

We end this paragraph with a lemma.

\begin{lemme}
\label{lemma:ga_ya}
For $\tuple{j} \in J^{\alpha}$ such that $j_a \neq j_{a+1}$ we have:
\begin{align*}
g_a y_{a+1} e(\tuple{j}) &= y_a g_a e(\tuple{j}),
\\
y_{a+1} g_a e(\tuple{j}) &= g_a y_a e(\tuple{j}).
\end{align*}
\end{lemme}

\begin{proof}
Indeed, $g_a y_{a+1} e(\tuple{j}) = \sum_{\tuple{i}} (g_a - q^{-i_{a+1}}g_a X_{a+1}) e(\tuple{i}, \tuple{j})$ and applying Lemmas~\ref{lemma:ga_Xa} and \ref{lemma:ga_e(i,j)} we get:
\begin{align*}
g_a y_{a+1} e(\tuple{j}) &= \sum_{\tuple{i} \in I^{\alpha}} (1 - q^{-i_{a+1}}X_a) e(s_a \cdot \tuple{i}) g_a e(\tuple{j})
\\
&= \sum_{\tuple{i} \in I^{\alpha}} (1 - q^{-i_a} X_a) e(\tuple{i}) g_a e(\tuple{j})
\\
g_a y_{a+1} e(\tuple{j}) &= y_a g_a e(\tuple{j}).
\end{align*}
\end{proof}

\subsubsection{Image of \texorpdfstring{$\psi_a$}{psia}}

We first define some invertible elements $Q_a(\tuple{i}, \tuple{j}) \in F[[y_a, y_{a+1}]]^{\times}$ for $1 \leq a < n$ and $(\tuple{i}, \tuple{j}) \in K^{\alpha}$ by:
\[
Q_a(\tuple{i}, \tuple{j}) \coloneqq 
\begin{cases}
\begin{cases}
1 - q + qy_{a+1} - y_a & \text{if } i_a = i_{a+1},
\\
(y_a(\tuple{i}) - q y_{a+1}(\tuple{i})) / (y_a(\tuple{i}) - y_{a+1}(\tuple{i})) & \text{if } i_a \nrelbar i_{a+1},
\\
(y_a(\tuple{i}) - q y_{a+1}(\tuple{i})) / (y_a(\tuple{i}) - y_{a+1}(\tuple{i}))^2 & \text{if } i_a \to i_{a+1},
\\
q^{i_a} & \text{if } i_a \leftarrow i_{a+1},
\\
q^{i_a} / (y_a(\tuple{i}) - y_{a+1}(\tuple{i})) & \text{if } i_a \leftrightarrows i_{a+1},
\end{cases}
& \text{if } j_a = j_{a+1},
\\
f_{a, \tuple{j}} & \text{if } j_a \neq j_{a+1},
\end{cases}
\]
where $f_{a, \tuple{j}} \in \{1, q\}$ is given for $j_a \neq j_{a+1}$ by:
\[
f_{a, \tuple{j}} \coloneqq \begin{cases}
q & \text{if } j_a < j_{a+1},
\\
1 & \text{if } j_a > j_{a+1},
\end{cases}
\]
with $<$ being a total ordering on $J = \mathbb{Z} / d \mathbb{Z} \simeq \{1, \dots, d\}$. 
For $d = 1$ or for $j_a = j_{a+1}$, the power series $Q_a(\tuple{i}, \tuple{j})$ coincides with the definition  given at \cite[(4.36)]{BrKl}.

\begin{remarque}
\label{remark:Qa_dependence}
The elements $Q_a(\tuple{i}, \tuple{j})$ depend only on $(i_a, i_{a+1})$ and $(j_a, j_{a+1})$. Moreover, as in \cite{BrKl} the explicit expression of $Q_a(\tuple{i}, \tuple{j})$ for $j_a = j_{a+1}$ does not really matter; only its properties are essential, for instance, those in Lemma~\ref{lemma:relation_Qa}.
\end{remarque}

\begin{remarque}
\label{remark:f_aj}
The scalar $f_{a, \tuple{j}}$ is only an artefact: if $q$ admits a square root $q^{1/2}$ in $F$, we can simply set $f_{a, \tuple{j}} \coloneqq q^{1/2}$.
\end{remarque}

Finally, we give an easy lemma about these $f_{a, \tuple{j}}$.

\begin{lemme}
\label{lemme:f_aj}
If $j_a \neq j_{a+1}$ then $f_{a, \tuple{j}} f_{a, s_a \cdot \tuple{j}} = q$.
\end{lemme} 

We introduce a notation for a power series $Q \in F[[y_1, \dots, y_n]]$: if $\sigma \in \mathfrak{S}_n$ is a permutation, we denote by $Q^{\sigma}$ the power series $Q^{\sigma}(y_1, \dots, y_n) \coloneqq Q(y_{\sigma^{-1}(1)}, \dots, y_{\sigma^{-1}(n)})$; note that $(Q^{\sigma})^{\rho} = Q^{\sigma\rho}$ (we get a right action).

We will use later the following properties verified by $Q_a(\tuple{i}, \tuple{j})$ (see \cite[(4.35)]{BrKl}), where $Q^{\sigma}(\tuple{i}, \tuple{j}) \coloneqq Q(\tuple{i}, \tuple{j})^{\sigma}$.
\begin{lemme}
\label{lemma:relation_Qa}
We have:
\begin{align*}
Q_{a+1}^{s_a}(\tuple{i}, \tuple{j}) &= Q_a^{s_{a+1}}(s_{a+1}s_a \cdot (\tuple{i}, \tuple{j})),
\\
Q_a^{s_{a+1}}(\tuple{i}, \tuple{j}) &= Q_{a+1}^{s_a}(s_a s_{a+1} \cdot (\tuple{i}, \tuple{j})).
\end{align*}
\end{lemme}

\begin{proof}
We check only the first equality, the second being straightforward considering $(\tuple{i}', \tuple{j}') \coloneqq s_a s_{a+1} \cdot (\tuple{i}, \tuple{j})$. For $i \in I$, let us write $y_a(i) \coloneqq q^i(1 - y_a)$; in particular, we have $y_a(i_a) = y_a(\tuple{i})$. Noticing that $s_{a+1}s_a = (a, a+2, a+1)$, we get for $j_{a+1} = j_{a+2}$:
\[
Q_a(s_{a+1} s_a \cdot (\tuple{i}, \tuple{j})) =
\begin{cases}
1 - q + qy_{a+1} - y_a & \text{if } i_{a+1} = i_{a+2},
\\
(y_a(i_{a+1}) - q y_{a+1}(i_{a+2})) / (y_a(i_{a+1}) - y_{a+1}(i_{a+
2})) & \text{if } i_{a+1} \nrelbar i_{a+2},
\\
(y_a(i_{a+1}) - q y_{a+1}(i_{a+2})) / (y_a(i_{a+1}) - y_{a+1}(i_{a+2}))^2 & \text{if } i_{a+1} \to i_{a+2},
\\
q^{i_{a+1}} & \text{if } i_{a+1} \leftarrow i_{a+2},
\\
q^{i_{a+1}} / (y_a(i_{a+1}) - y_{a+1}(i_{a+2})) & \text{if } i_{a+1} \leftrightarrows i_{a+2},
\end{cases}
\]
and we conclude using $y_a(i_{a+1})^{s_{a+1}} = y_a(i_{a+1})$ and $y_{a+1}(i_{a+2})^{s_{a+1}} = y_{a+2}(i_{a+2})$. If $j_{a+1} \neq j_{a+2}$, we have $Q_a(s_{a+1}s_a \cdot (\tuple{i}, \tuple{j})) = f_{a, (a, a+2, a+1)\cdot \tuple{j}} = \begin{cases}
q & \text{if } j_{a+1} < j_{a+2},
\\
1 & \text{if } j_{a+1} > j_{a+2},
\end{cases}$ and this is exactly $f_{a+1, \tuple{j}}$.
\end{proof}

We now introduce the following element of $\widehat{\Y}_{\alpha}^{\tuple{\Lambda}}(q)$:
\[
\Phi_a \coloneqq g_a + (1-q)\sum_{\substack{(\tuple{i}, \tuple{j}) \in K^{\alpha} \\ i_a \neq i_{a+1} \\ j_a = j_{a+1}}} \left(1 - X_a X_{a+1}^{-1}\right)^{-1} e(\tuple{i}, \tuple{j}) + \sum_{\substack{(\tuple{i}, \tuple{j}) \in K^{\alpha} \\ i_a = i_{a+1} \\ j_a = j_{a+1}}} e(\tuple{i}, \tuple{j}),
\]
where ${(1 - X_a X_{a+1}^{-1})}^{-1}e(\tuple{k})$ denotes the inverse of $(1 - X_a X_{a+1}^{-1})e(\tuple{k})$ in $e(\tuple{k})\widehat{\Y}_{\alpha}^{\tuple{\Lambda}}(q)e(\tuple{k})$. Note that this element is indeed invertible, since for $\tuple{k} = (\tuple{i}, \tuple{j})$ with $i_a \neq i_{a+1}$ its only eigenvalue $1 - q^{i_a - i_{a+1}}$ is non-zero, thanks to the definition of $I$.
In particular, we have:
\[
\Phi_a e(\tuple{j}) = g_a e(\tuple{j}) \qquad \text{if } j_a \neq j_{a+1}.
\]
For $d = 1$ we get the ``intertwining element'' defined in \cite[\textsection 4.2]{BrKl}, and we write it $\Phi_a^{\H}$.

\begin{remarque}
Though we will not need this until Section~\ref{section:yokonuma_hecke_generators}, we define now the  power series $P_a(\tuple{i}, \tuple{j}) \in F[[y_a, y_{a+1}]]$ for $1 \leq a < n$ and $(\tuple{i}, \tuple{j}) \in K^{\alpha}$ by:
\[
P_a(\tuple{i}, \tuple{j}) \coloneqq \begin{cases}
\begin{cases}
1 & \text{if } i_a = i_{a+1},
\\
(1-q){(1 - y_a(\tuple{i})y_{a+1}(\tuple{i})^{-1})}^{-1} & \text{if } i_a \neq i_{a+1},
\end{cases}
& \text{if } j_a = j_{a+1},
\\
0 & \text{if } j_a \neq j_{a+1}.
\end{cases}
\]
For $d = 1$ or for $j_a = j_{a+1}$ we recover the definition given at \cite[(4.27)]{BrKl}.  Moreover, we have the following equality:
\begin{equation}
\label{equation:Phia_ga_Pa}
\Phi_a = \sum_{\tuple{k} \in K^{\alpha}} \left(g_a + P_a(\tuple{k})\right)e(\tuple{k}).
\end{equation}
Indeed, the only non-obvious fact to check is $P_a(\tuple{i}, \tuple{j})e(\tuple{i}, \tuple{j}) = (1-q){(1-X_a X_{a+1}^{-1})}^{-1} e(\tuple{i}, \tuple{j})$ if $i_a \neq i_{a+1}$ and $j_a = j_{a+1}$, but this is clear by \eqref{equation:ya(i)e(i)=Xae(i)}.
We will also use the following equality (the same one as in Lemma~\ref{lemma:relation_Qa}):
\begin{equation}
\label{equation:Pa_sa}
P_{a+1}^{s_a}(\tuple{i}, \tuple{j}) = P_a^{s_{a+1}}(s_{a+1}s_a \cdot(\tuple{i}, \tuple{j})).
\end{equation}
\end{remarque}

\begin{lemme}
\label{lemma:Phi_a}
We have the following properties:
\begin{align}
\label{equation:phia_e(j)}
\Phi_a e(\tuple{j}) &= e(s_a \cdot \tuple{j})\Phi_a,
\\
\label{equation:phia_e(ij)}
\Phi_a e(\tuple{i}, \tuple{j}) &= e(s_a \cdot (\tuple{i}, \tuple{j}))\Phi_a,
\\
\label{equation:phia_Xb}
\phantom{\forall b \neq a, a+1 \qquad}\Phi_a X_b &= X_b \Phi_a  \qquad\forall  b \neq a, a+1,
\\
\label{equation:phia_yb}
\phantom{\forall b \neq a, a+1 \qquad}\Phi_a y_b &= y_b \Phi_a  \qquad\forall b \neq a, a+1,
\\
\label{equation:phia_Qb}
\phantom{\forall |b-a| > 1 \qquad}\Phi_a Q_b(\tuple{k}) &= Q_b(\tuple{k}) \Phi_a \qquad\forall |b-a| > 1,
\\
\label{equation:phia_phib}
\phantom{\forall |b-a| > 1 \qquad} \Phi_a \Phi_b &= \Phi_b\Phi_a \qquad \forall |b-a|>1.
\end{align}
\end{lemme}

\begin{proof}
We will use results from \cite[Lemma 4.1]{BrKl} (which is this lemma for $d = 1$).
\begin{itemize}
\item[(\ref{equation:phia_e(j)})] Using Lemma~\ref{lemme:g_a_e(j)}, it is clear if $j_a \neq j_{a+1}$ since then $\Phi_a e(\tuple{j}) = g_a e(\tuple{j}) = e(s_a \cdot \tuple{j}) g_a = e(s_a\cdot \tuple{j})\Phi_a$. Using Corollary~\ref{corollaire:ga_e(j)_commutent} it is clear if $j_a = j_{a+1}$  since $e(\tuple{j})$ commutes with every term in the definition of $\Phi_a$.

\item[(\ref{equation:phia_e(ij)})] If $j_a = j_{a+1}$, we claim that the relation comes applying \eqref{equation:phia_e(j)} and Remark~\ref{remark:brundan_kleshchev_ja=ja+1} on the equality $\Phi_a^{\H}e^{\H}(\tuple{i}) = e^{\H}(s_a \cdot \tuple{i})\Phi_a^{\H}$. 
If $j_a \neq j_{a+1}$ it follows directly from Lemma~\ref{lemma:ga_e(i,j)}.

\item[(\ref{equation:phia_Xb})] Straightforward using \eqref{relation:g_a_X_b} and since $e(\tuple{i}, \tuple{j})$ are polynomials in $X_1, \dots, X_n$.

\item[(\ref{equation:phia_yb})] Using \eqref{equation:definition_e(i)_e(j)}, \eqref{equation:phia_e(ij)} and \eqref{equation:phia_Xb} we get:
\begin{gather*}
\Phi_a y_b = \sum_{\tuple{i}, \tuple{j}} (1 - q^{-i_b}X_b)e(s_a\cdot (\tuple{i}, \tuple{j})) \Phi_a
\\
= \sum_{\tuple{i}} (1 - q^{-i_b}X_b)e(s_a\cdot \tuple{i}) \Phi_a = y_b \Phi_a,
\end{gather*}
since $(s_a \cdot \tuple{i})_b = i_b$.

\item[(\ref{equation:phia_Qb})] Since $Q_b(\tuple{k}) \in F[[y_b, y_{b+1}]]$ and $b \neq a, a+1$ and $b+1 \neq a, a+1$ it follows from \eqref{equation:phia_yb}.

\item[(\ref{equation:phia_phib})] Let us write $\Phi_a' \coloneqq \Phi_a - g_a$. Using \eqref{relation:g_a_X_b} and Lemma~\ref{lemma:ga_e(i,j)} we get:
\begin{gather*}
\Phi_a' g_b = g_b\left((1-q)\sum_{\substack{i_a \neq i_{a+1} \\ j_a = j_{a+1}}} {(1 - X_a X_{a+1}^{-1})}^{-1} e(s_b\cdot (\tuple{i}, \tuple{j}))\right.
\\ + \left.\sum_{\substack{i_a = i_{a+1} \\ j_a = j_{a+1}}} e(s_b\cdot(\tuple{i}, \tuple{j}))\right) = g_b \Phi_a',
\end{gather*}
and exchanging $a$ and $b$ we get $g_a \Phi_b' = \Phi_b' g_a$. Noticing that $\Phi_a' \Phi_b' = \Phi_b' \Phi_a'$ (we don't use here $|b-a| > 1$) and using \eqref{relation:ga_gb}, we get $\Phi_a \Phi_b = (g_a + \Phi_a')(g_b + \Phi_b') = g_a g_b + \Phi_a' g_b + g_a \Phi_b' + \Phi_a'\Phi_b' = g_b g_a + g_b \Phi_a' + \Phi_b' g_a + \Phi_b' \Phi_a' = (g_b + \Phi_b')(g_a + \Phi_a') = \Phi_b \Phi_a$.
\end{itemize}
\end{proof}

We are now ready to define our elements $\psi_a$ for $1 \leq a < n$:
\[
\psi_a \coloneqq \sum_{\tuple{k} \in K^{\alpha}} \Phi_a Q_a(\tuple{k})^{-1} e(\tuple{k}) \in \widehat{\Y}_{\alpha}^{\tuple{\Lambda}}(q).
\]

As usual, we write $\psi_a^{\H}$ for $\psi_a$ when $d = 1$, and this element $\psi_a^{\H}$ corresponds with the $\psi_a$ of \cite[\textsection 4.3]{BrKl}. Note finally that for $\tuple{j} \in J^{\alpha}$ we have:
\[
\psi_a e(\tuple{j}) = f_{a, \tuple{j}}^{-1} g_a e(\tuple{j}) \qquad \text{if } j_a \neq j_{a+1}.
\]

\subsection{Check of the defining relations}
\label{subsection:check_kl_generators}

We now check the defining relations \eqref{relation:quiver_sum_alpha_e(i)}--\eqref{relation:quiver_ya+1_psia}, \eqref{relation:quiver_psia^2}--\eqref{relation:quiver_tresse} and \eqref{relation:quiver_cyclo_y1} for the elements we have just defined.
The idea is the following: when an element $e(\tuple{i}, \tuple{j})$ lies in a relation to check, if $j_a = j_{a+1}$ then we get immediately the result by Remark~\ref{remark:brundan_kleshchev_ja=ja+1} rewriting the same proof as \cite[Theorem 4.2]{BrKl}, and if $j_a \neq j_{a+1}$ then it will be easy (at least, easier than in \cite{BrKl}) to prove the relation. Recall that we always work in $\widehat{\Y}_{\alpha}^{\tuple{\Lambda}}(q)$ (in particular every relation should be multiplied by $e(\alpha)$) and we write $e(\tuple{i})$ and $e(\tuple{j})$ without any $(\alpha)$.

\paragraph*{(\ref{relation:quiver_cyclo_y1})} 
We do exactly the same proof as in $\widehat{\H}_{\alpha}^{\tuple{\Lambda}}(q)$.
Let $(\tuple{i}, \tuple{j}) \in K^{\alpha}$ and set $M \coloneqq \widehat{\Y}_{\alpha}^{\tuple{\Lambda}}(q)$; recall that the action of $X_1$ on $M$ is given by the action of $e(\alpha)X_1$.  By $(\ref{relation:X1_cyclo})$ we have
$\prod_{i \in I} (X_1 - q^{i})^{\Lambda_{i}} = 0$, hence:
\begin{equation}
\label{equation:relation_quiver_cyclo_y1}
\prod_{i \in I} \left[{(X_1 - q^{i})}^{\Lambda_{i}} e(\tuple{i})\right] = 0.
\end{equation}
As an endomorphism of $M(\tuple{i})$, the element ${(X_1 - q^{i})}^{\Lambda_{i}}$ is invertible if $i \neq i_1$ since its only eigenvalue ${(q^{i_1} - q^{i})}^{\Lambda_{i}}$ is non-zero (note that $\Lambda_{i}$ may be equal to $0$). This means that there exist elements ${(X_1 - q^{i})}^{-\Lambda_{i}} e(\tuple{i})$ such that ${(X_1 - q^{i})}^{\Lambda_{i}} e(\tuple{i}) \cdot {(X_1 - q^{i})}^{-\Lambda_{i}} e(\tuple{i}) = e(\tuple{i})$.    Hence, multiplying by all these inverses, the equation \eqref{equation:relation_quiver_cyclo_y1} becomes:
\[
{(X_1 - q^{i_1})}^{\Lambda_{i_1}}e(\tuple{i}) = 0.
\]
Finally, since $y_1^{\Lambda_{i_1}} = \sum_{\tuple{i}' \in I^{\alpha}} {(1 - q^{-i'_1}X_1) }^{\Lambda_{i_1}} e(\tuple{i}')$
we obtain:
\[
y_1^{\Lambda_{i_1}} e(\tuple{i}, \tuple{j}) = {(1 - q^{-i_1}X_1)}^{\Lambda_{i_1}}e(\tuple{i}, \tuple{j}) = 0.\]

\paragraph*{(\ref{relation:quiver_sum_alpha_e(i)})}
Straightforward from the definition of the $e(\tuple{k})$ for $\tuple{k} \in K^{\alpha}$.

\paragraph*{(\ref{relation:quiver_e(i)e(i')})} Idem.

\paragraph*{(\ref{relation:quiver_y_ae(i)})} Straightforward since $y_a$ and $e(\tuple{k})$ both lie in the commutative subalgebra generated by $t_1, \dots, t_n$ and  $X_1, \dots, X_n$.

\paragraph*{(\ref{relation:quiver_psiae(i)})} Straightforward by \eqref{equation:phia_e(ij)} and since $Q_a(\tuple{k}')$ and $e(\tuple{k}')$ commute with $e(\tuple{k})$.

\paragraph*{(\ref{relation:quiver_ya_yb})} True since $\{X_a\}_a$ is commutative.

\paragraph*{(\ref{relation:quiver_psia_yb})} True by \eqref{equation:phia_yb}.

\paragraph*{(\ref{relation:quiver_psia_psib})} Let $|a - b| > 1$. We have, using \eqref{relation:quiver_psiae(i)}, \eqref{equation:phia_Qb} and \eqref{equation:phia_phib}:
\begin{align*}
\psi_a \psi_b 
&= \sum_{\tuple{k}} \Phi_a Q_a(\tuple{k})^{-1} e(\tuple{k}) \psi_b
\\
&= \sum_{\tuple{k}} \Phi_a Q_a(\tuple{k})^{-1} \psi_b e(s_b\cdot \tuple{k})
\\
&= \sum_{\tuple{k}} \Phi_a Q_a(\tuple{k})^{-1} \Phi_b Q_b(s_b \cdot \tuple{k})^{-1} e(s_b\cdot \tuple{k})
\\
&= \sum_{\tuple{k}} \Phi_b Q_b(s_b \cdot \tuple{k})^{-1} \Phi_a Q_a(\tuple{k})^{-1} e(s_b \cdot \tuple{k}).
\end{align*}
Hence, noticing that (see Remark~\ref{remark:Qa_dependence}) $Q_a(\tuple{k}) = Q_a(s_b \cdot \tuple{k})$ and $Q_b(s_b\cdot \tuple{k})= Q_b(s_a s_b\cdot \tuple{k})$ we get:
\begin{align*}
\psi_a \psi_b
&= \sum_{\tuple{k}} \Phi_b Q_b(s_b\cdot \tuple{k})^{-1} \psi_a e(s_b\cdot \tuple{k})
\\
&= \sum_{\tuple{k}} \Phi_b Q_b(s_b\cdot \tuple{k})^{-1} e(s_a s_b\cdot \tuple{k}) \psi_a
\\
\psi_a \psi_b &= \psi_b \psi_a.
\end{align*}

\paragraph*{(\ref{relation:quiver_psia_ya+1})}

First, if $\tuple{k} = (\tuple{i}, \tuple{j})$ verifies $j_a = j_{a+1}$ then by Remark~\ref{remark:brundan_kleshchev_ja=ja+1} we get from:
\[
\psi_a^{\H}y_{a+1}^{\H}e^{\H}(\tuple{i}) = \begin{cases}
(y_a^{\H}\psi_a^{\H} + 1)e^{\H}(\tuple{i}) & \text{if } i_a = i_{a+1},
\\
y_a^{\H} \psi_a^{\H} e^{\H}(\tuple{i}) & \text{if } i_a \neq i_{a+1},
\end{cases}
\]
the following equality:
\[
\psi_a y_{a+1} e(\tuple{i}, \tuple{j}) = \begin{cases}
(y_a \psi_a + 1)e(\tuple{i}, \tuple{j}) & \text{if } i_a = i_{a+1} \text{ and } j_a = j_{a+1},
\\
y_a \psi_a e(\tuple{i}, \tuple{j}) & \text{if } i_a \neq i_{a+1} \text{ and } j_a = j_{a+1}.
\end{cases}
\]
Hence it remains to deal with the case $j_a \neq j_{a+1}$ (and no condition on $\tuple{i}$). Using Lemma~\ref{lemma:ga_ya} we get:
\begin{align*}
\psi_a y_{a+1} e(\tuple{i}, \tuple{j}) &= \Phi_a Q_a(\tuple{i}, \tuple{j})^{-1} y_{a+1} e(\tuple{i}, \tuple{j})
\\
&= f_{a, \tuple{j}}^{-1} g_a y_{a+1} e(\tuple{i}, \tuple{j})
\\
&= f_{a, \tuple{j}}^{-1} y_a g_a e(\tuple{i}, \tuple{j})
\\
&= y_a \Phi_a Q_a(\tuple{i}, \tuple{j})^{-1} e(\tuple{i}, \tuple{j})
\\
\psi_a y_{a+1} e(\tuple{i}, \tuple{j}) &= y_a \psi_a e(\tuple{i}, \tuple{j}).
\end{align*}
Finally, we have proved:
\[
\psi_a y_{a+1} e(\tuple{k}) = 
\begin{cases}
(y_a \psi_a + 1)e(\tuple{k}) & \text{if } k_a = k_{a+1},
\\
y_a \psi_a e(\tuple{k}) & \text{if } k_a \neq k_{a+1},
\end{cases}
\]
which is exactly \eqref{relation:quiver_psia_ya+1}.

\paragraph*{(\ref{relation:quiver_ya+1_psia})} Similar.

\bigskip

\begin{remarque}
\label{remark:f_psia}
Thanks to relations \eqref{relation:quiver_psia_yb}, \eqref{relation:quiver_psia_ya+1} and \eqref{relation:quiver_ya+1_psia}, given $f \in F[[y_1, \dots, y_n]]$ and $\tuple{k} \in K^{\alpha}$ such that $k_a \neq k_{a+1}$ we have $f \psi_a e(\tuple{k}) = \psi_a f^{s_a} e(\tuple{k})$. In particular, this holds if $j_a \neq j_{a+1}$ with $\tuple{k} = (\tuple{i}, \tuple{j})$.
\end{remarque}

\paragraph*{(\ref{relation:quiver_psia^2})} Once again, the result is straightforward if $j_a = j_{a+1}$ using Remark~\ref{remark:brundan_kleshchev_ja=ja+1}. Let us then suppose $j_a \neq j_{a+1}$; hence, necessarily we have $k_a \nrelbar k_{a+1}$ so we have to prove $\psi_a^2 e(\tuple{k}) = e(\tuple{k})$. We have:
\begin{align*}
\psi_a^2 e(\tuple{k}) &= \psi_a e(s_a\cdot \tuple{k}) \psi_a
\\
&= \Phi_a  Q_a(s_a\cdot \tuple{k})^{-1} e(s_a\cdot \tuple{k}) \psi_a
\\
&= f_{a, s_a\cdot \tuple{j}}^{-1} g_a \Phi_a Q_a(\tuple{k})^{-1} e(\tuple{k})
\\
\psi_a^2 e(\tuple{k}) &= (f_{a, s_a\cdot \tuple{j}} f_{a, \tuple{j}})^{-1} g_a^2 e(\tuple{k}).
\end{align*}
Applying Lemmas~\ref{lemma:ga_Xa} and \ref{lemme:f_aj} we find $\psi_a^2 e(\tuple{k}) = e(\tuple{k})$ (recall $e(\tuple{k}) = e(\tuple{i})e(\tuple{j}) = e(\tuple{j}) e(\tuple{i})$) thus we are done.

\paragraph*{(\ref{relation:quiver_tresse})} If $j_a = j_{a+1} = j_{a+2}$, we get the result using Remark~\ref{remark:brundan_kleshchev_ja=ja+1}. Let us then suppose that we are not in that case: we have to prove $\psi_{a+1} \psi_a \psi_{a+1} e(\tuple{k}) = \psi_a \psi_{a+1} \psi_a e(\tuple{k})$.
We will intensively use \eqref{relation:quiver_psiae(i)}; note also that:
\[
\Phi_a e(\tuple{i}, \tuple{j}) = 
\begin{cases}
\left(g_a + (1-q) {(1-X_a X_{a+1}^{-1})}^{-1}\right) e(\tuple{i}, \tuple{j}) & \text{if } i_a \neq i_{a+1} \text{ and } j_a = j_{a+1},
\\
(g_a + 1) e(\tuple{i}, \tuple{j}) &\text{if } i_a = i_{a+1} \text{ and } j_a = j_{a+1},
\\
g_a e(\tuple{i}, \tuple{j})  &\text{otherwise } (j_a \neq j_{a+1}),
\end{cases}
\]
and:
\[
\psi_a e(\tuple{i}, \tuple{j}) = 
\begin{cases}
\Phi_a Q_a(\tuple{i}, \tuple{j})^{-1} e(\tuple{i}, \tuple{j}) & \text{if } j_a = j_{a+1},
\\
f_{a, \tuple{j}}^{-1} g_a e(\tuple{i}, \tuple{j}) &\text{if } j_a \neq j_{a+1}.
\end{cases}
\]

It is convenient to introduce some notation. The couple $(\tuple{i}, \tuple{j})$ shall only be modified by the action of $s_a$ or $s_{a+1}$, hence we only write $((i_a, i_{a+1}, i_{a+2}), (j_a, j_{a+1}, j_{a+2}))$ for $(\tuple{i}, \tuple{j})$. Moreover, for clarity we forget comas and only write the indexation, substituting $0$ to $a$; thus, $((i_a, i_{a+1}, i_{a+2}), (j_a, j_{a+1}, j_{a+2}))$ becomes $((012), (012))$. Finally, as $\mathfrak{S}_n$ acts diagonally on $I \times J$ we can write $(012)$ instead of $((012), (012))$. Because an example beats lines of explanation, here is one: $\psi_0 e(102)$ stands for $\psi_a e(s_a\cdot \tuple{k})$.

\subparagraph*{Case $j_0 = j_1 \neq j_2$.} Let us first compute $\psi_1 \psi_0 \psi_1 e(012)$ and $\psi_0 \psi_1 \psi_0 e(012)$. We have:
\begin{align*}
\psi_1 \psi_0 \psi_1 e(012) &= \psi_1 \psi_0 e(021)\psi_1
\\
&= \psi_1 e(201)\psi_0 \psi_1
\\
\psi_1 \psi_0 \psi_1 e(012) &= \Phi_1 Q_1(201)^{-1}e(201) \psi_0 \psi_1.
\end{align*}
Since $Q_1(201)^{-1} \in F[[y_1, y_2]]$ and recalling Remark~\ref{remark:f_psia} we get:
\begin{gather*}
\psi_1\psi_0\psi_1 e(012) =  \Phi_1 e(201)\psi_0 Q_1^{s_0}(201)^{-1} \psi_1  =  \Phi_1 e(201) \psi_0 \psi_1 Q_1^{s_0 s_1}(201)^{-1}.
\end{gather*}
By Lemma~\ref{lemma:relation_Qa} we have $Q_1^{s_0 s_1}(201) = Q_1(201)^{s_0 s_1} = (Q_1(201)^{s_0})^{s_1} = (Q_0(s_1 s_0\cdot (201))^{s_1})^{s_1} = Q_0(012)$. Hence:
\[
\psi_1 \psi_0 \psi_1 e(012) =  \Phi_1 e(201) \psi_0 \psi_1 Q_0(012)^{-1}.
\]

As:
\[
\psi_0 \psi_1 \psi_0 e(012) = \psi_0 \psi_1 \Phi_0 e(012) Q_0(012)^{-1},
\]
to have \eqref{relation:quiver_tresse} it suffices to prove:
\begin{equation}
\label{equation:j0=j1neqj2}
\Phi_1 e(201) \psi_0 \psi_1 = \psi_0 \psi_1 \Phi_0 e(012).
\end{equation}
Let us distinguish two subcases.
\begin{itemize}
\item If $i_0 \neq i_1$ then:
\[
\Phi_1 e(201) \psi_0 \psi_1 = \left(g_1 + (1-q){(1 - X_1 X_2^{-1})}^{-1} \right)\psi_0 e(021)\psi_1.
\]
Recalling \eqref{relation:g_a_X_b} and Lemma~\ref{lemma:ga_Xa} we get:
\begin{align*}
\Phi_1 e(201)\psi_0 \psi_1  &= f_{0, (021)}^{-1} \left(g_1 g_0 + (1-q)g_0 {(1 - X_0 X_2^{-1})}^{-1}\right) \psi_1 e(012)
\\&=f_{0, (021)}^{-1} f_{1, (012)}^{-1}  \left(g_1 g_0 g_1 + (1-q)g_0 g_1 {(1 - X_0 X_1^{-1})}^{-1} \right)e(012).
\end{align*}
Using the braid relation \eqref{relation:tresse_ga} this becomes, recalling \eqref{equation:phia_e(ij)}:
\begin{align*}
\Phi_1 e(201) \psi_0 \psi_1  &= f_{0, (021)}^{-1} f_{1, (012)}^{-1}\left(g_0 g_1 g_0 + (1-q)g_0 g_1 {(1 - X_0 X_1^{-1})}^{-1}\right) e(012)
\\
&= f_{0, (021)}^{-1} f_{1, (012)}^{-1}g_0 g_1 \left(g_0 + (1-q){(1 - X_0 X_1^{-1})}^{-1}\right)e(012)
\\
&= f_{0, (021)}^{-1} f_{1, (012)}^{-1} g_0 g_1 \Phi_0 e(012)
\\
\Phi_1 e(201) \psi_0 \psi_1 &= f_{0, (021)}^{-1} f_{1, (012)}^{-1} g_0 g_1 e(102) \Phi_0,
\end{align*}
and then, noticing that $f_{1, (012)} = f_{1, (102)}$ and $f_{0, (021)} = f_{0, (120)}$, we obtain:
\begin{align*}
\Phi_1 e(201) \psi_0 \psi_1  
&= f_{0, (120)}^{-1} f_{1, (012)}^{-1} g_0 g_1 e(102) \Phi_0
\\
&= f_{0, (120)}^{-1} g_0 e(120)\psi_1\Phi_0
\\
\Phi_1 e(201) \psi_0 \psi_1 &= \psi_0 \psi_1 \Phi_0 e(012),
\end{align*}
thus \eqref{equation:j0=j1neqj2} is proved.

\item If $i_0 = i_1$ then:
\begin{align*}
\Phi_1 e(201) &= (g_1 + 1)e(201),
\\
\Phi_0 e(012) &= (g_0 + 1) e(012),
\end{align*}
thus with the same calculation as above (even easier) we get:
\begin{gather*}
\Phi_1 e(201) \psi_0 \psi_1 = f_{0, (021)}^{-1} f_{1, (012)}^{-1}(g_1 g_0 g_1 + g_0 g_1)e(012)
\\
= f_{0, (120)}^{-1} f_{1, (102)}^{-1}(g_0 g_1 g_0 + g_0 g_1)e(012)
= \psi_0 \psi_1 \Phi_0 e(012),
\end{gather*}
so we got \eqref{equation:j0=j1neqj2}.
\end{itemize}
Until the end of the proof we use the same arguments as here, arguments which we will thus not recall.

\subparagraph*{Case $j_0 \neq j_1 = j_2$.} Similar.
\subparagraph*{Case $j_0 = j_2 \neq j_1$.} Once again we begin with the computation of $\psi_1 \psi_0 \psi_1 e(012)$ and $\psi_0 \psi_1 \psi_0 e(012)$. We have:
\[
\psi_1 \psi_0 \psi_1 e(012)
=  \psi_1 \Phi_0 Q_0(021)^{-1} e(021)\psi_1
= \psi_1 \Phi_0 e(021)\psi_1  Q_0^{s_1}(021)^{-1},
\]
and:
\[
\psi_0 \psi_1 \psi_0 e(012)
= \psi_0 \Phi_1 Q_1(102)^{-1} e(102)\psi_0
= \psi_0 \Phi_1 e(102) \psi_0  Q_1^{s_0}(102)^{-1}.
\]
Since $Q_0^{s_1}(021)^{-1} = Q_1^{s_0}(102)^{-1}$, it suffices to prove:
\begin{equation}
\label{equation:j0=j2neqj1}
\psi_1 \Phi_0 e(021) \psi_1 = \psi_0 \Phi_1 e(102)\psi_0.
\end{equation}

Once again we distinguish two subcases.
\begin{itemize}
\item If $i_0 \neq i_2$ then:
\begin{align*}
\psi_1 \Phi_0 e(021) \psi_1
&= \psi_1 e(201) \Phi_0 \psi_1
\\
&= f_{1, (201)}^{-1} g_1 \left(g_0 + (1-q){(1 - X_0 X_1^{-1})}^{-1}\right)\psi_1 e(012)
\\
&= f_{1, (201)}^{-1} f_{1, (012)}^{-1} \left(g_1 g_0 g_1 + (1-q)g_1^2 {(1 - X_0 X_2^{-1})}^{-1}\right)e(012)
\\
\psi_1 \Phi_0 e(021) \psi_1 &= f_{1, (201)}^{-1} f_{1, (012)}^{-1}\left(g_1 g_0 g_1 + (1-q)q{(1-X_0 X_2^{-1})}^{-1}\right) e(012).
\end{align*}
Similarly, we find:
\begin{align*}
\psi_0 \Phi_1 e(102)\psi_0 &= \psi_0 e(120) \Phi_1 \psi_0
\\
&= f_{0, (120)}^{-1} g_0 \left (g_1 + (1-q){(1 - X_1 X_2^{-1})}^{-1}\right)\psi_0 e(012)
\\
&= f_{0, (120)}^{-1}  f_{0, (012)}^{-1}  \left(g_0 g_1 g_0 + (1-q)g_0^2 {(1 - X_0 X_2^{-1})}^{-1}\right)e(012)
\\
\psi_0 \Phi_1 e(102)\psi_0 &= f_{0, (120)}^{-1}  f_{0, (012)}^{-1} \left(g_0 g_1 g_0 + (1-q)q {(1 - X_0 X_2^{-1})}^{-1}\right)e(012),
\end{align*}
thus we conclude since $f_{1, (201)} = f_{0, (012)}$ and $f_{1, (012)} = f_{0, (120)}$ (we see it on this particular case or we can use Lemma~\ref{lemma:relation_Qa}).

\item If $i_0 = i_2$ we get as above, with $\alpha \coloneqq f_{1, (201)}^{-1} f_{1, (012)}^{-1} = f_{0, (120)}^{-1}  f_{0, (012)}^{-1}$:
\begin{gather*}
\psi_1 \Phi_0 e(021) \psi_1 = \alpha(g_1 g_0 g_1 + g_1^2)e(012)
\\
= \alpha (g_1 g_0 g_1 + q) e(012) = \alpha (g_0 g_1 g_0 + q) e(012)
\\
= \alpha (g_0 g_1 g_0 + g_0^2)e(012) = \psi_0 \Phi_1 e(102) \psi_0 e(012).
\end{gather*}
\end{itemize}

\subparagraph*{Case $\#\{j_a, j_{a+1}, j_{a+2}\} = 3$.} We have $j_a \neq j_{a+1}$ and $j_a \neq j_{a+2}$ and $j_{a+1} \neq j_{a+2}$ thus we get immediately:
\begin{gather*}
\psi_1 \psi_0 \psi_1 e(012) = f_{1, (201)}^{-1} f_{0, (021)}^{-1} f_{1, (012)}^{-1} g_1 g_0 g_1 e(012)
\\
= f_{0, (120)}^{-1} f_{1, (102)}^{-1} f_{0, (012)}^{-1} g_0 g_1 g_0 e(012) = \psi_0 \psi_1 \psi_0 e(012),
\end{gather*}
since $f_{1, (201)} = f_{0, (012)}, f_{0, (021)} = f_{1, (102)}$ and $f_{1, (012)} = f_{0, (120)}$.

\section{Yokonuma--Hecke generators of \texorpdfstring{$\H_{\alpha}^{\tuple{\Lambda}}(\Gamma)$}{HalphaLambdaGamma}}
\label{section:yokonuma_hecke_generators}

Let $\tuple{\Lambda}$ be a weight as in Section~\ref{section:quiver_hecke_generators}.
The aim of this section is to prove the following theorem.

\begin{theoreme}
\label{theorem:morphism_yh_quiver}
For any $\alpha \comp_K n$, we can construct an explicit algebra homomorphism:
\[
\sigma :  \widehat{\Y}_{d, n}^{\tuple{\Lambda}}(q) \to \H_\alpha^{\tuple{\Lambda}}(\Gamma).
\]
\end{theoreme}

Note that we do not consider yet $\widehat{\Y}_{\alpha}^{\tuple{\Lambda}}(q)$. In particular, it suffices to define the images of the generators \eqref{equation:generators_yokonumahecke} and check if they verify the defining relations of the cyclotomic Yokonuma--Hecke algebra. As in Section~\ref{section:quiver_hecke_generators}, we  use the same notation for a generator and its image.

\subsection{Definition of the images of the generators}
\label{subsection:definition_YH_generators}

It is easier this time to define these images. First, since the elements $y_1, \dots, y_n$ are nilpotent (Lemma~\ref{lemme:y_a_nilpotent}), we can consider power series in these variables. Hence, the quantities $P_a(\tuple{k})$, $Q_a(\tuple{k})$ and $y_a(\tuple{i})$ that we defined in \textsection\ref{subsection:definition_images_quiverhecke_generators} are also well-defined as elements of $\H_{\alpha}^{\tuple{\Lambda}}(\Gamma)$. We define finally as in \eqref{equation:definition_e(i)_e(j)} the elements $e(\tuple{i})$ and $e(\tuple{j})$ of $\H_{\alpha}^{\tuple{\Lambda}}(\Gamma)$ for $\tuple{i} \in I^{\alpha}$ and $\tuple{j} \in J^{\alpha}$.

We recall that $\xi$ is a primitive $d$th root of unity in $F$. Our ``Yokonuma--Hecke generators'' of $\H_{\alpha}^{\tuple{\Lambda}}(\Gamma)$ are given below.
\begin{align*}
\phantom{\text{for } 1 \leq a < n} && g_a &\coloneqq \sum_{\tuple{k} \in K^{\alpha}}\left(\psi_a Q_a(\tuple{k}) - P_a(\tuple{k})\right)e(\tuple{k}) & \text{for } 1 &\leq a < n
\\
\phantom{\text{for } 1 \leq a \leq n}& & t_a &\coloneqq \sum_{\tuple{j} \in J^{\alpha}} \xi^{j_a} e(\tuple{j})  &\text{for } 1 &\leq a \leq n
\\
\phantom{\text{for } 1 \leq a \leq n}&& X_a & \coloneqq \sum_{\tuple{i} \in I^{\alpha}} y_a(\tuple{i}) e(\tuple{i}) &\text{for } 1 &\leq a \leq n
\end{align*}

As usual, we write $g_a^{\H}$ and $X_a^{\H}$ for the corresponding elements when $d = 1$: we recover the elements of \cite[\textsection 4.4]{BrKl}.

\bigskip

\begin{remarque}[About Brundan and Kleshchev's proof - \theBKproof]
\stepcounter{BKproof}
\label{remark:proof_BK_Hecke_generators_ja_ja+1}
This remark is similar to Remark~\ref{remark:brundan_kleshchev_ja=ja+1}. If $\tuple{j} \in J^{\alpha}$ verifies $j_a = j_{a+1}$ and if a relation in \cite[\textsection 4]{BrKl} involves only $\psi_a^{\H}$, $e^{\H}(\tuple{i})$ for $\tuple{i} \in I^{\alpha}$ and $y_b^{\H}$ for $1 \leq b \leq n$, while its proof does not require any cyclotomic relation \eqref{relation:quiver_cyclo_y1}, then by the same proof, the same relation is satisfied between $\psi_a e(\tuple{j})$, $e(\tuple{i}, \tuple{j})$ and $y_b e(\tuple{j})$ in the unitary algebra $e(\tuple{j})\H_{\alpha}^{\tuple{\Lambda}}(\Gamma)e(\tuple{j})$. If $\tuple{j}$ verifies in addition $j_{a+1} = j_{a+2}$, we will be able to add relations with $\psi_{a+1}^{\H}$, which we substitute by $\psi_{a+1}e(\tuple{j})$.
\end{remarque}

\subsection{Check of the defining relations}
\label{subsection:check_yh_generators}

As in \textsection\ref{subsection:check_kl_generators}, we will use Remark~\ref{remark:proof_BK_Hecke_generators_ja_ja+1} when $j_a = j_{a+1}$ to get the result from the same corresponding proof of \cite[Theorem 4.3]{BrKl}, and when $j_a \neq j_{a+1}$ we will need a few calculations.

\paragraph*{(\ref{relation:ordre_t_a})} Straightforward since $e(\tuple{j})e(\tuple{j}') = \delta_{\tuple{j}, \tuple{j}'} e(\tuple{j})$ and $\xi^d = 1$.

\paragraph*{(\ref{relation:ta_tb})} Straightforward since  $e(\tuple{j})e(\tuple{j}') = e(\tuple{j}')e(\tuple{j})$.

\paragraph*{(\ref{relation:t_b_g_a})}
According to \eqref{relation:quiver_sum_alpha_e(i)}, it suffices to prove $t_b g_a e(\tuple{i}, \tuple{j}) = g_a t_{s_a(b)}e(\tuple{i}, \tuple{j})$ for every $(\tuple{i}, \tuple{j}) \in K^{\alpha}$.
For $(\tuple{i}, \tuple{j}) \in K^{\alpha}$, we have, using \eqref{relation:quiver_psiae(i)}:
\begin{align*}
t_b g_a e(\tuple{i}, \tuple{j}) &= t_b (\psi_a Q_a(\tuple{i}, \tuple{j}) - P_a(\tuple{i}, \tuple{j}))e(\tuple{i}, \tuple{j})
\\
&= t_b \left[e(s_a\cdot(\tuple{i}, \tuple{j}))\psi_a Q_a(\tuple{i}, \tuple{j}) - e(\tuple{i}, \tuple{j}) P_a(\tuple{i}, \tuple{j})\right] \notag
\\
&= \xi^{(s_a \cdot \tuple{j})_b} \psi_a Q_a(\tuple{i}, \tuple{j}) e(\tuple{i}, \tuple{j}) - \xi^{j_b} P_a(\tuple{i}, \tuple{j}) e(\tuple{i}, \tuple{j}) 
\\
t_b g_a e(\tuple{i}, \tuple{j})&= \psi_a Q_a(\tuple{i}, \tuple{j})\xi^{(s_a \cdot \tuple{j})_b} e(\tuple{i}, \tuple{j}) - P_a(\tuple{i}, \tuple{j}) \xi^{j_b} e(\tuple{i}, \tuple{j}),
\end{align*}
and:
\[
g_a t_{s_a(b)} e(\tuple{i}, \tuple{j}) = g_a \xi^{j_{s_a(b)}} e(\tuple{i}, \tuple{j})
= \psi_a Q_a(\tuple{i}, \tuple{j})\xi^{j_{s_a(b)}} e(\tuple{i}, \tuple{j}) - P_a(\tuple{i}, \tuple{j})\xi^{j_{s_a(b)}} e(\tuple{i}, \tuple{j}).
\]

As $(s_a \cdot \tuple{j})_b = j_{s_a(b)}$ (by definition of the action of $\mathfrak{S}_n$ on $J^n$), it suffices to prove the following:
\[
P_a(\tuple{i}, \tuple{j}) \xi^{j_b} = P_a(\tuple{i}, \tuple{j}) \xi^{j_{s_a(b)}}.
\]
\begin{itemize}
\item It is clear if $b \notin \{a, a+1\}$ since $b = s_a(b)$.

\item If $b \in \{a, a+1\}$, it is clear if $j_a = j_{a+1}$ and obvious if $j_a \neq j_{a+1}$ since then $P_a(\tuple{i}, \tuple{j}) = 0$.
\end{itemize}

\paragraph*{(\ref{relation:ordre_ga}).} Let $(\tuple{i}, \tuple{j}) \in K^{\alpha}$ and let us prove $g_a^2 e(\tuple{i}, \tuple{j}) = (q + (q-1)g_a e_a)e(\tuple{i}, \tuple{j})$; summing over all $(\tuple{i}, \tuple{j}) \in K^{\alpha}$ will conclude. If $j_a = j_{a+1}$ then it is immediate applying Remark~\ref{remark:proof_BK_Hecke_generators_ja_ja+1} on $(g_a^{\H})^2 = q + (q-1)g_a^{\H}$ and left-multiplying by $e(\tuple{i})$, recalling $e_a e(\tuple{j}) = e(\tuple{j})$ and Corollary~\ref{corollaire:ga_e(j)_commutent}. If now $j_a \neq j_{a+1}$, since $e_a e(\tuple{j}) = 0$ it suffices to prove $g_a^2 e(\tuple{i}, \tuple{j})= q e(\tuple{i}, \tuple{j})$. But, recalling $Q_a(\tuple{i}, \tuple{j}) = f_{a, \tuple{j}}$ and $P_a(\tuple{i}, \tuple{j}) = 0$:
\begin{align*}
g_a^2 e(\tuple{i}, \tuple{j}) &= g_a (\psi_a Q_a(\tuple{i}, \tuple{j}) - P_a(\tuple{i}, \tuple{j}))e(\tuple{i}, \tuple{j})
\\
&= f_{a, \tuple{j}} g_a \psi_a e(\tuple{i}, \tuple{j})
\\
&= f_{a, \tuple{j}} g_a e(s_a\cdot(\tuple{i}, \tuple{j}))\psi_a
\\
&= f_{a, \tuple{j}} (\psi_a Q_a(s_a\cdot(\tuple{i}, \tuple{j})) - P_a(s_a\cdot(\tuple{i}, \tuple{j})))\psi_a e(\tuple{i}, \tuple{j})
\\
g_a^2 e(\tuple{i}, \tuple{j})&= f_{a, \tuple{j}} f_{a, s_a\cdot \tuple{j}} \psi_a^2 e(\tuple{i}, \tuple{j}),
\end{align*}
hence we conclude using Lemma~\ref{lemme:f_aj} and \eqref{relation:quiver_psia^2}, since $j_a \neq j_{a+1}$ implies $(i_a, j_a) \nrelbar (i_{a+1}, j_{a+1})$.

\paragraph*{(\ref{relation:ga_gb}).} Let us prove $g_a g_b e(\tuple{k}) = g_b g_a e(\tuple{k})$ for every $\tuple{k} \in K^{\alpha}$.  By \eqref{relation:quiver_psia_yb} the element $\psi_b$ commutes with the elements $P_a(\tuple{k})$ and $Q_a(\tuple{k})$ of $F[[y_a, y_{a+1}]]$. Moreover, $Q_a(s_b\cdot \tuple{k}) = Q_a(\tuple{k})$ and $P_a(s_b\cdot \tuple{k}) = P_a(\tuple{k})$, hence:
\begin{align*}
g_a g_b e(\tuple{k}) &= g_a (\psi_b Q_b(\tuple{k}) - P_b(\tuple{k})) e(\tuple{k})
\\
&= g_a e(s_b\cdot \tuple{k}) \psi_b Q_b(\tuple{k}) - g_a e(\tuple{k}) P_b(\tuple{k})
\\
&= (\psi_a Q_a(\tuple{k}) - P_a(\tuple{k})) \psi_b Q_b(\tuple{k})e(\tuple{k})  - (\psi_a Q_a(\tuple{k}) - P_a(\tuple{k})) P_b(\tuple{k}) e(\tuple{k})
\\
&= \psi_a \psi_b Q_a(\tuple{k}) Q_b(\tuple{k}) e(\tuple{k}) - \psi_b Q_b(\tuple{k})P_a(\tuple{k})  e(\tuple{k}) - \psi_a Q_a(\tuple{k}) P_b(\tuple{k}) e(\tuple{k})
\\
&\qquad + P_a(\tuple{k}) P_b(\tuple{k}) e(\tuple{k}),
\end{align*}
and we conclude since that expression is symmetric in $a$ and $b$ (recalling \eqref{relation:quiver_psia_psib}).

\paragraph*{(\ref{relation:tresse_ga}).} Again it suffices to prove $g_{a+1} g_a g_{a+1} e(\tuple{i}, \tuple{j}) = g_a g_{a+1} g_a e(\tuple{i}, \tuple{j})$ for all $(\tuple{i}, \tuple{j}) \in K^{\alpha}$. If $j_a = j_{a+1} = j_{a+2}$ we get the result using Remark~\ref{remark:proof_BK_Hecke_generators_ja_ja+1}. Let us then suppose that we are not in that case. We will intensively use \eqref{relation:quiver_psiae(i)}; recall the following fact:
\[
g_a e(\tuple{i}, \tuple{j}) = \begin{cases}
(\psi_a Q_a(\tuple{i}, \tuple{j}) - P_a(\tuple{i}, \tuple{j}))e(\tuple{i}, \tuple{j}) & \text{if } j_a = j_{a+1},
\\
f_{a, \tuple{j}} \psi_a e(\tuple{i}, \tuple{j}) & \text{if } j_a \neq j_{a+1}.
\end{cases}
\]

Finally, as during the proof of \eqref{relation:quiver_tresse} in \textsection\ref{subsection:check_kl_generators}, we  write for example $g_0 e(102)$ instead of $g_a e(s_a\cdot \tuple{k})$. Thus, given our hypothesis on $j_0, j_1$ and $j_2$ we have:
\begin{equation}
\label{equation:demo_yhgenerator_tresse_psi}
\psi_1 \psi_0 \psi_1 e(012) = \psi_0 \psi_1 \psi_0 e(012).
\end{equation}

\subparagraph*{Case $j_0 = j_1 \neq j_2$.} Let us first compute $g_1 g_0 g_1 e(012)$ and $g_0 g_1 g_0 e(012)$. We set $\alpha \coloneqq f_{1, (012)} f_{0, (021)}$; we have:
\begin{align*}
g_1 g_0 g_1 e(012)
&= f_{1, (012)} g_1 g_0 e(021)\psi_1
\\
&= f_{1, (012)} f_{0, (021)} g_1 e(201)\psi_0 \psi_1
\\
&= \alpha \psi_1 Q_1(201)\psi_0 \psi_1 e(012) - \alpha P_1(201)\psi_0 \psi_1 e(012)
\\
g_1 g_0 g_1 e(012)
&= \alpha \psi_1 \psi_0 \psi_1 e(012) Q_1^{s_0 s_1}(201) - \alpha \psi_0 \psi_1 e(012) P_1^{s_0 s_1}(201).
\end{align*}
We have already seen that $Q_1^{s_0 s_1}(201) = Q_0(012)$ and similarly we have  $P_1^{s_0 s_1}(201) = P_0(012)$ (see \eqref{equation:Pa_sa}). Hence we obtain, using \eqref{equation:demo_yhgenerator_tresse_psi} and noticing $f_{1, (012)} = f_{0, (120)}$ and $f_{0, (021)} = f_{1, (102)}$:
\begin{align*}
g_1 g_0 g_1 e(012) &= \alpha \psi_0 \psi_1 \psi_0 e(012) Q_0(012) - \alpha \psi_0 \psi_1 e(012) P_0 (012)
\\
&= \alpha \psi_0 \psi_1 e(102)\psi_0 Q_0(012) - f_{1, (012)}f_{0, (021)} \psi_0 e(021)\psi_1 P_0 (012)
\\
&= f_{1, (102)}f_{0, (120)} \psi_0 e(120)\psi_1 \psi_0 Q_0 (012) - f_{1, (012)} g_0 \psi_1 e(012)P_0(012)
\\
&= f_{1, (102)} g_0 \psi_1 e(102)\psi_0 Q_0(012) - g_0 g_1 P_0(012)e(012)
\\
&= g_0 g_1 (\psi_0 Q_0(012) - P_0(012))e(012)
\\
g_1 g_0 g_1 e(012) &= g_0 g_1 g_0 e(012),
\end{align*}
so we are done.

\subparagraph*{Case $j_0 \neq j_1 = j_2$.} Similar.

\subparagraph*{Case $j_0 = j_2 \neq j_1$.} Given these assumptions we have:
\begin{equation}
\label{equation:demo_yhgenerator_carre_psi}
\psi_0^2 e(012) = \psi_1^2 e(012) = e(012).
\end{equation}

Hence, using \eqref{equation:demo_yhgenerator_carre_psi}, with $\alpha \coloneqq f_{1, (012)} f_{1, (201)}$:
\begin{align*}
g_1 g_0 g_1 e(012) &= f_{1, (012)} g_1 (\psi_0 Q_0(021) - P_0(021))e(021)\psi_1 
\\
&= f_{1, (012)} g_1 e(201)\psi_0 Q_0(021)\psi_1 - f_{1, (012)} g_1 e(021)P_0(021)\psi_1
\\
&= \alpha \psi_1 \psi_0 \psi_1 e(012)Q_0^{s_1}(021) - \alpha \psi_1^2 e(012) P_0^{s_1}(021)
\\
&= \alpha\psi_0 \psi_1 \psi_0 e(012) Q_1^{s_0}(102) - \alpha \psi_0^2 e(012) P_1^{s_0}(102)
\\
&=  \alpha \psi_0 e(120) \psi_1 Q_1(102) \psi_0 - \alpha \psi_0 e(102) P_1(102)\psi_0.
\end{align*}
Noticing $f_{1, (012)} = f_{0, (120)} = f_{0, (102)}$ and $f_{1, (201)} = f_{0, (012)}$ we get finally:
\begin{align*}
g_1 g_0 g_1 e(012) &= f_{0, (012)} g_0 (\psi_1 Q_1(102) - P_1(102))e(102) \psi_0
\\
&= f_{0, (012)} g_0 g_1 \psi_0 e(012) 
\\
g_1 g_0 g_1 e(012)&=   g_0 g_1 g_0 e(012).
\end{align*}

\subparagraph*{Case $\#\{j_0, j_1, j_2\} = 3$.} We get immediately:
\begin{gather*}
g_1 g_0 g_1 e(012)
= f_{1, (201)} f_{0, (021)} f_{1, (012)} \psi_1 \psi_0 \psi_1 e(012)
\\
= f_{0, (120)} f_{1, (102)} f_{0, (012)} \psi_0 \psi_1 \psi_0 e(012)
= g_0 g_1 g_0 e(012),
\end{gather*}
since $f_{1, (201)} = f_{0, (012)}$ and $f_{0, (021)} = f_{1, (102)}$ and $f_{1, (012)} = f_{0, (120)}$.

\paragraph*{(\ref{relation:X1_g1_X1_g1}).} Since for $a \in \{1, \dots, n-1\}$ it is clear that $X_{a+1} X_a = X_a X_{a+1}$, it remains to prove that $q X_{a+1} = g_a X_a g_a$; we will conclude taking $a = 1$. As we proved \eqref{relation:ordre_ga}, it suffices to prove \eqref{relation:g_a_X_a+1}. Let $(\tuple{i}, \tuple{j}) \in K^{\alpha}$ and let us prove:
\[
g_a X_{a+1} e(\tuple{i}, \tuple{j}) = \begin{cases}
X_a g_a e(\tuple{i}, \tuple{j})+ (q-1) X_{a+1}e(\tuple{i}, \tuple{j}) & \text{if } j_a = j_{a+1},
\\
X_a g_a e(\tuple{i}, \tuple{j}) & \text{if } j_a \neq j_{a+1}.
\end{cases}\]

Again, we deduce the case $j_a = j_{a+1}$ from Remark~\ref{remark:proof_BK_Hecke_generators_ja_ja+1}. If $j_a \neq j_{a+1}$ we have, using \eqref{relation:quiver_psiae(i)} and \eqref{relation:quiver_psia_ya+1}:
\begin{align*}
g_a X_{a+1} e(\tuple{i}, \tuple{j}) &= q^{i_{a+1}} g_a e(\tuple{i}, \tuple{j}) (1 - y_{a+1})
\\
&=  q^{i_{a+1}} f_{a, \tuple{j}} \psi_a (1 - y_{a+1}) e(\tuple{i}, \tuple{j})
\\
&= q^{i_{a+1}} f_{a, \tuple{j}}(1 - y_a)e(s_a\cdot(\tuple{i}, \tuple{j})) \psi_a
\\
&= f_{a, \tuple{j}} X_a \psi_a e(\tuple{i}, \tuple{j})
\\
g_a X_{a+1} e(\tuple{i}, \tuple{j}) &= X_a g_a e(\tuple{i}, \tuple{j}).
\end{align*}

\paragraph*{(\ref{relation:X1_ga}).} We prove in fact \eqref{relation:g_a_X_b}, that is, $g_a X_b = X_b g_a$ for $b \neq a, a+1$. As $y_b$ commutes with $\psi_a$ by \eqref{relation:quiver_psia_yb} we have, for any $\tuple{k} \in K^{\alpha}$ (where $y_a(\tuple{k}) \coloneqq y_a(\tuple{i})$ with $\tuple{k} = (\tuple{i}, \tuple{j})$):
\begin{align*}
g_a X_b e(\tuple{k}) 
&= g_a e(\tuple{k}) y_b(\tuple{k})
\\
 &= y_b(\tuple{k}) (\psi_a Q_a(\tuple{k}) - P_a(\tuple{k})) e(\tuple{k})
\\
&= y_b(\tuple{k}) e(s_a\cdot \tuple{k}) \psi_a Q_a(\tuple{k}) - y_b(\tuple{k}) e(\tuple{k}) P_a(\tuple{k})
\\
g_a X_b e(\tuple{k}) &= X_b g_a e(\tuple{k}),
\end{align*}
since $y_b(\tuple{k}) e(s_a\cdot \tuple{k}) = q^{(s_a \cdot \tuple{i})_b}(1 - y_b) e(s_a\cdot \tuple{k}) = q^{i_b}(1 - y_b)e(s_a\cdot \tuple{k}) = X_b e(s_a\cdot \tuple{k})$.

\paragraph*{(\ref{relation:X1_tb}).} We prove in fact $X_a t_b = t_b X_a$ for every $a, b$; that is straightforward from \eqref{relation:quiver_y_ae(i)}.

\paragraph*{(\ref{relation:X1_cyclo}).}
We have, using \eqref{relation:quiver_sum_alpha_e(i)}--\eqref{relation:quiver_y_ae(i)}:
\begin{align*}
\prod_{i \in I}{(X_1 - q^i)}^{\Lambda_i}
&= \prod_{i \in I} \left[\sum_{\tuple{i} \in I^{\alpha}} \left(q^{i_1}(1 - y_1) - q^i\right) e(\tuple{i})\right]^{\Lambda_i}
\\
&= \prod_{i \in I} \left[\sum_{\tuple{i} \in I^{\alpha}} \left(q^{i_1}(1 - y_1) - q^i\right)^{\Lambda_i} e(\tuple{i})\right]
\\
\prod_{i \in I}{(X_1 - q^i)}^{\Lambda_i}
&= \sum_{\tuple{i} \in I^{\alpha}} \prod_{i \in I} \left[ \left(q^{i_1}(1 - y_1) - q^i\right)^{\Lambda_i}e(\tuple{i}) \right].
\end{align*}
Noticing that for each $\tuple{i} \in I^{\alpha}$ the term for $i = i_1$ vanishes by \eqref{relation:quiver_cyclo_y1}, we get the result.

\section{Isomorphism theorem}
\label{section:isomorphism}

We give now the main result of our paper; let $\tuple{\Lambda}$ be a weight as in Sections~\ref{section:quiver_hecke_generators} and \ref{section:yokonuma_hecke_generators}.

\subsection{Statement}

\begin{theoreme}
\label{theorem:main}
There is a presentation of the algebra $\widehat{\Y}_{\alpha}^{\tuple{\Lambda}}(q)$ given by the generators \eqref{equation:quiver_generators_alpha} and the relations \eqref{relation:quiver_e(i)e(i')}--\eqref{relation:quiver_ya+1_psia}, \eqref{relation:quiver_psia^2}--\eqref{relation:quiver_tresse} and \eqref{relation:quiver_cyclo_y1}, that is, we have an algebra isomorphism:
\[
\H_{\alpha}^{\tuple{\Lambda}}(\Gamma) \overset{\sim}\longrightarrow \widehat{\Y}_{\alpha}^{\tuple{\Lambda}}(q).
\]
\end{theoreme}

As finitely many $\widehat{\Y}_{\alpha}^{\tuple{\Lambda}}(q)$ are non-zero, we deduce from this isomorphism that only finitely many $\H_{\alpha}^{\tuple{\Lambda}}(\Gamma)$ are non-zero. Hence, as $\widehat{\Y}_{d, n}^{\tuple{\Lambda}}(q) = \oplus_{\alpha \comp_K n} \widehat{\Y}_{\alpha}^{\tuple{\Lambda}}(q)$, defining the cyclotomic quiver Hecke algebra of degree $n$ by (recall \eqref{equation:Hn(Q)_simeq_oplus_Halpha(Q)}):
\begin{equation}
\label{equation:HnLambda=oplus}
\H_n^{\tuple{\Lambda}}(\Gamma) \coloneqq \bigoplus_{\alpha \comp_K n} \H_{\alpha}^{\tuple{\Lambda}}(\Gamma),
\end{equation}
we get the unitary algebra isomorphism with the cyclotomic Yokonuma--Hecke algebra:
\begin{equation}
\label{equation:main_theorem_full_algebra}
\H_n^{\tuple{\Lambda}}(\Gamma) \simeq \widehat{\Y}_{d, n}^{\tuple{\Lambda}}(q).
\end{equation}

Recalling that the cyclotomic quiver Hecke algebra is naturally graded (Proposition~\ref{proposition:gradation_quiver_Hecke_algebra}), we get the following corollary.
\begin{corollaire}
The cyclotomic Yokonuma--Hecke algebra inherits the grading of the cyclotomic quiver Hecke algebra.
\end{corollaire}
%

Moreover, as we obtain a presentation of $\widehat{\Y}_{d, n}^{\tuple{\Lambda}}(q)$ which does not depend on $q$, we also get another one (see Corollary~\ref{corollary:main_improvement} for a slight improvement).

\begin{corollaire}
Let $q' \in F\setminus \{0, 1\}$. If $\mathrm{char}_{q'}(F) = \mathrm{char}_q(F)$ then the algebras $\widehat{\Y}_{d, n}^{\tuple{\Lambda}}(q)$ and $\widehat{\Y}_{d, n}^{\tuple{\Lambda}}(q')$ are isomorphic.
\end{corollaire}

Let us now prove Theorem~\ref{theorem:main}.
First, as we have a (non-unitary) algebra homomorphism $\widehat{\Y}_{\alpha}^{\tuple{\Lambda}}(q) \to \widehat{\Y}_{d, n}^{\tuple{\Lambda}}(q)$, by Theorem~\ref{theorem:morphism_yh_quiver} we get an algebra homomorphism $\widehat{\Y}_{\alpha}^{\tuple{\Lambda}}(q) \to \H_{\alpha}^{\tuple{\Lambda}}(\Gamma)$, that we still call $\sigma$.
We will prove that $\sigma : \widehat{\Y}_{\alpha}^{\tuple{\Lambda}}(q) \to \H_{\alpha}^{\tuple{\Lambda}}(\Gamma)$ and $\rho : \H_{\alpha}^{\tuple{\Lambda}}(\Gamma) \to \widehat{\Y}_{\alpha}^{\tuple{\Lambda}}(q)$ (from Theorem~\ref{theorem:morphism_quiver_yh})  verify $\sigma \circ \rho = \mathrm{id}_{\H_{\alpha}^{\tuple{\Lambda}}(\Gamma)}$ and $\rho \circ \sigma = \mathrm{id}_{\widehat{\Y}_{\alpha}^{\tuple{\Lambda}}(q)}$. Since these are algebra homomorphisms, it suffices to prove that they are identity on generators. To clarify the proof, let us add a ${}^{\Y}$ on the quiver Hecke generators of $\widehat{\Y}_{\alpha}^{\tuple{\Lambda}}(q)$ and a ${}^{\H}$ on the Yokonuma--Hecke generators of $\H_{\alpha}^{\tuple{\Lambda}}(\Gamma)$ (there isn't any confusion possible with the former notation referring to the case $d = 1$ since we won't use it any more).

\subsection{Proof of \texorpdfstring{$\sigma \circ \rho = \mathrm{id}_{\H_{\alpha}^{\tuple{\Lambda}}(\Gamma)}$}{sigmarho=id}}

We have to check that $\sigma(\rho(e(\tuple{k}))) = e(\tuple{k})$ for all $\tuple{k} \in K^{\alpha}$, that $\sigma(\rho(y_a)) = y_a$ for all $1 \leq a \leq n$ and that $\sigma(\rho(\psi_a)) = \psi_a$ for all $1 \leq a < n$.

Let us start by finding the image of $e(\tuple{k})$ by $\sigma \circ \rho$. By definition of $\rho$ we have $\rho(e(\tuple{k})) = e^{\Y}(\tuple{k})$, so we have to prove $\sigma(e^{\Y}(\tuple{k})) = e(\tuple{k})$. Let $M \coloneqq \H_{\alpha}^{\tuple{\Lambda}}(\Gamma)$; the algebra homomorphism $\sigma$ gives $M$ a structure of $\widehat{\Y}_\alpha^{\tuple{\Lambda}}(q)$-module, finite-dimensional thanks to Theorem~\ref{theorem:generating_family_cyclotomic_quiver}. If $M(\tuple{k})$ denotes the weight space as in \eqref{equation:weight_space}, by Remark~\ref{remark:e(k)_independent_M} we know that the projection onto $M(\tuple{k})$ along $\oplus_{\tuple{k}' \neq \tuple{k}} M(\tuple{k}')$ is given by $\sigma(e^{\Y}(\tuple{k}))$. We prove that $e(\tuple{k})$ is this projection too.

Let $(\tuple{i}, \tuple{j}) \in K^{\alpha}$. For $1 \leq a \leq n$, we have $\sigma(X_a) = \sum_{\tuple{i}'} (q^{i'_a} - q^{i'_a}y_a)e(\tuple{i}')$ so:
\[
\sigma(X_a) - q^{i_a} = \sum_{\tuple{i}' \in I^\alpha} \left[(q^{i'_a} - q^{i_a}) - q^{i'_a} y_a\right] e(\tuple{i}').
\]

Since $y_a$ is nilpotent, thanks to \eqref{relation:quiver_sum_alpha_e(i)}--\eqref{relation:quiver_e(i)e(i')} we have:
\[
\left\lbrace v \in M : (\sigma(X_a) - q^{i_a})^N v = 0\right\rbrace = \left(\sum_{\substack{\tuple{i}' \in I^{\alpha} \\ i'_a = i_a}} e(\tuple{i}')\right) M,
\]
hence, for $N \gg 0$ we have:
\[
M(\tuple{i}) \coloneqq \left\lbrace v \in M : (\sigma(X_a) - q^{i_a})^N v = 0 \; \forall a\right\rbrace = e(\tuple{i})M.
\]

In a similar way we have $M(\tuple{j}) = e(\tuple{j})M$ where $M(\tuple{j}) \coloneqq \{v \in M : (\sigma(t_a) - \xi^{j_a})v = 0$ $\forall a\}$, thus:
\[
M(\tuple{k}) = e(\tuple{k})M.
\]

Hence, as $\oplus_k M(\tuple{k}) = M$ we conclude that $e(\tuple{k})$ is the desired projection and finally $e(\tuple{k}) = \sigma(e^{\Y}(\tuple{k}))$.

The end of the proof is without any difficulty. We have:
\begin{align*}
\sigma(\rho(y_a)) &= \sigma(y_a^{\Y}) = \sum_{\tuple{i} \in I^{\alpha}} [1 - q^{-i_a} \sigma(X_a)] \sigma(e^{\Y}(\tuple{i}))
\\
&= \sum_{\tuple{i} \in I^{\alpha}} [1 - q^{-i_a} X_a^{\H}] e(\tuple{i})
\\
&= \sum_{\tuple{i} \in I^{\alpha}}\left[1 - q^{-i_a} \sum_{\tuple{i}' \in I^{\alpha}} y_a(\tuple{i}')e(\tuple{i}')\right] e(\tuple{i})
\\
&= \sum_{\tuple{i} \in I^{\alpha}} [1 - q^{-i_a} y_a(\tuple{i})] e(\tuple{i})
\\
&= \sum_{\tuple{i} \in I^{\alpha}} [1 - q^{-i_a} q^{i_a}(1 - y_a)] e(\tuple{i})
\\
\sigma(\rho(y_a)) &= y_a.
\end{align*}
Thus, we have $\sigma(Q_a^{\Y}(\tuple{k})) = Q_a(\tuple{k})$ and $\sigma(P_a^{\Y}(\tuple{k})) = P_a(\tuple{k})$, hence, recalling \eqref{equation:Phia_ga_Pa}:
\begin{align*}
\sigma(\rho(\psi_a)) &= \sigma(\psi_a^{\Y}) = \sum_{\tuple{k} \in K^{\alpha}} \sigma(\Phi_a) \sigma(Q_a^{\Y}(\tuple{k}))^{-1} \sigma(e^{\Y}(\tuple{k}))
\\
&= \sum_{\tuple{k} \in K^{\alpha}} \left(\sum_{\tuple{k}' \in K^{\alpha}}\left[\sigma(g_a) + \sigma(P_a^{\Y}(\tuple{k}'))\right]e(\tuple{k}')\right) Q_a(\tuple{k})^{-1} e(\tuple{k})
\\
&= \sum_{\tuple{k} \in K^{\alpha}} (g_a^{\H} + P_a(\tuple{k}))Q_a(\tuple{k})^{-1} e(\tuple{k})
\\
&= \sum_{\tuple{k} \in K^{\alpha}} \left(\left[\sum_{\tuple{k}' \in K^{\alpha}} (\psi_a Q_a(\tuple{k}') - P_a(\tuple{k}'))e(\tuple{k}')\right] + P_a(\tuple{k})\right)Q_a(\tuple{k})^{-1} e(\tuple{k})
\\
&= \sum_{\tuple{k} \in K^{\alpha}} [(\psi_a Q_a(\tuple{k}) - P_a(\tuple{k})) + P_a(\tuple{k})]Q_a(\tuple{k})^{-1} e(\tuple{k})
\\
\sigma(\rho(\psi_a)) &= \psi_a.
\end{align*}

\subsection{Proof of \texorpdfstring{$\rho\circ\sigma = \mathrm{id}_{\widehat{\Y}_{\alpha}^{\tuple{\Lambda}}(q)}$}{rhosigma=id}}

This is even easier: we have to check $\rho(\sigma(g_a)) = g_a$ for $1 \leq a < n$ and $\rho(\sigma(X_a)) = X_a$ and $\rho(\sigma(t_a)) = t_a$ for $1 \leq a \leq n$. We have:
\begin{align*}
\rho(\sigma(g_a)) &= \sum_{\tuple{k} \in K^{\alpha}} [\psi_a^{\Y} Q_a^{\Y}(\tuple{k}) - P_a^{\Y}(\tuple{k})] e^\Y(\tuple{k})
\\
&= \sum_{\tuple{k} \in K^{\alpha}} [\Phi_a Q_a^{\Y}(\tuple{k})^{-1} Q_a^{\Y}(\tuple{k}) - P_a^{\Y}(\tuple{k})]e^\Y(\tuple{k})
\\
&= \sum_{\tuple{k} \in K^{\alpha}} [\Phi_a - P_a^{\Y}(\tuple{k})]e^\Y(\tuple{k})
\\
\rho(\sigma(g_a)) &= g_a.
\end{align*}

Recalling \eqref{equation:ya(i)e(i)=Xae(i)}:
\[
\rho(\sigma(X_a)) = \sum_{\tuple{i} \in I^{\alpha}} y_a^{\Y}(\tuple{i}) e^{\Y}(\tuple{i})
= \sum_{\tuple{i} \in I^{\alpha}} X_a e^\Y(\tuple{i})
= X_a.
\]

Finally:
\[
\rho(\sigma(t_a)) = \sum_{\tuple{j} \in J^{\alpha}} \xi^{j_a} e^{\Y}(\tuple{j})
= \sum_{\tuple{j} \in J^{\alpha}} t_a e^{\Y}(\tuple{j})
 = t_a.
\]

The proof of Theorem~\ref{theorem:main} is now over.

\section{Degenerate case}
\label{section:degenerate}

In this section, we extend the previous results to the case $q = 1$. In particular, we need to define a new ``degenerate'' cyclotomic Yokonuma--Hecke algebra. Many calculations are not written, since they are entirely similar to the non-degenerate case. Note the following  thing: since the cyclotomic quiver Hecke algebra has no $q$ in its presentation, we do not need to define some new cyclotomic quiver Hecke algebra.

Let $\tuple{\Lambda} = (\Lambda_k)_{k \in K} \in \mathbb{N}^{(K)}$ be a weight; we assume that $\ell(\tuple{\Lambda}) = \sum_{k \in K} \Lambda_k$ verifies $\ell(\tuple{\Lambda}) > 0$. Moreover, as in Section~\ref{section:quiver_hecke_generators} we suppose that for any $i \in I$ and $j, j' \in J$, we have:
\[
\Lambda_{i, j} = \Lambda_{i, j'} \eqqcolon \Lambda_i.
\]
In particular, we will write $\tuple{\Lambda}$ as well for the weight $(\Lambda_i)_{i \in I}$.

\subsection{Degenerate cyclotomic Yokonuma--Hecke algebras}
\label{subsection:degenerate_cyclotomic_algebra}

We introduce here the degenerate cyclotomic Yokonuma--Hecke algebra: this algebra can be seen as the rational degeneration of the cyclotomic Yokonuma--Hecke algebra $\widehat{\Y}_{d, n}^{\tuple{\Lambda}}(q)$.

The \emph{degenerate cyclotomic Yokonuma--Hecke algebra of type A}, denoted by $\widehat{\Y}_{d, n}^{\tuple{\Lambda}}(1)$, is the unitary associative $F$-algebra generated by the elements
\begin{equation}
\label{equation:degenerate_generators_yokonumahecke}
f_1, \dots, f_{n-1}, t_1, \dots, t_n, x_1, \dots, x_n
\end{equation}
subject to the following relations:
\begin{align}
\label{degenerate_relation:ordre_t_a}
t_a^d &= 1, \\
\label{degenerate_relation:ta_tb}
t_a t_b &= t_b t_a, \\
\label{degenerate_relation:t_b_f_a}
t_b f_a &= f_a t_{s_a(b)}, \\
\label{degenerate_relation:ordre_fa}
f_a^2 &= 1
\\
\label{degenerate_relation:fa_fb}
f_a f_b &= f_b f_a \qquad \forall |a-b| > 1,\\
\label{degenerate_relation:tresse_fa}
f_{a+1} f_a f_{a+1} &= f_a f_{a+1} f_a, 
\end{align}
where $e_a \coloneqq \frac{1}{d} \sum_{j \in J} t_a^j t_{a+1}^{-j}$, together with the following relations:
\begin{align}
\label{degenerate_relation:xr_xs}
x_a x_b &= x_b x_a,
\\
\label{degenerate_relation:fa_xa+1}
f_a x_{a+1} &= x_a f_a + e_a,
\\
\label{degenerate_relation:fa_xb}
f_a x_b &= x_b f_a \qquad \forall b \neq a, a+1,
\\
\label{degenerate_relation:xa_tb}
x_a t_b &= t_b x_a,
\end{align}
and finally the cyclotomic one:
\begin{equation}
\label{degenerate_relation:x1_cyclo}
\prod_{i \in I} (x_1 - i)^{\Lambda_i} = 0.
\end{equation}

We obtained this presentation by setting $X_a = 1 + (q-1)x_a$ in $\widehat{\Y}_{d, n}^{\tuple{\Lambda}}(q)$, simplifying by $(1-q)$ as much as we can and then setting $q = 1$ (according to the transformation made by Drinfeld \cite{Dr} to define degenerate Hecke algebras). As in the non-degenerate case, the element $e_a$ verifies $e_a^2 = e_a$ and commutes with $f_a$. Finally, note some consequences of  \eqref{degenerate_relation:ordre_fa} and \eqref{degenerate_relation:fa_xa+1}:
\begin{gather}
\label{degenerate_equation:xa+1_fa_xa}
x_{a+1} = f_a x_a f_a + f_a e_a,
\\
\label{degenerate_equation:xa+1_fa}
x_{a+1} f_a = f_a x_a + e_a.
\end{gather}

When $d = 1$, we recover the \emph{degenerate cyclotomic Hecke algebra} $\widehat{\H}_n^{\tuple{\Lambda}}(1)$ of \cite{BrKl}; it is the degenerate cyclotomic Yokonuma--Hecke algebra $\widehat{\Y}^{\tuple{\Lambda}}_{1, n}(1)$. In particular, the element $e_a$ becomes $1$, and $f_a$ (respectively $x_b$) is the element $s_a$ (resp. $x_b$) of \cite[\textsection 3]{BrKl}.

We will use the following lemma (see \cite[Lemma 2.15]{ChPA2} for the non-degenerate case).
\begin{lemme}
For $u, v \in \mathbb{N}$ we have the following equalities:
\begin{gather}
\label{degenerate_equation:fa_xa_xa+1}
f_a x_a x_{a+1} = x_a x_{a+1} f_a,
\\
\label{degenerate_equation:fa_xa+1v}
f_a x_{a+1}^v = x_a^v f_a + e_a\sum_{m = 0}^{v-1} x_a^m x_{a+1}^{v-1-m},
\\
\label{degenerate_equation:fa_xau}
f_a x_a^u = x_{a+1}^u f_a - e_a \sum_{m = 0}^{u-1} x_a^m x_{a+1}^{u-1-m},
\\
\label{degenerate_equation:fa_xau_xa+1v}
f_a x_a^u x_{a+1}^v
=
\begin{dcases}
x_a^v x_{a+1}^u f_a + e_a\sum_{m = 0}^{v-u-1} x_a^{u+m} x_{a+1}^{v-1-m}
&
\text{if } u \leq v,
\\
x_a^v x_{a+1}^u f_a - e_a\sum_{m = 0}^{u-v-1} x_a^{u-1+m} x_{a+1}^{v-m}
&
\text{if } u \geq v.
\end{dcases}
\end{gather}
\end{lemme}

\begin{proof}
We deduce \eqref{degenerate_equation:fa_xa_xa+1} from different previous relations. The relations \eqref{degenerate_equation:fa_xa+1v} and \eqref{degenerate_equation:fa_xau} can be proved by an easy induction. The equality \eqref{degenerate_equation:fa_xau_xa+1v} follows finally from these previous equalities.
\end{proof}

As the elements $g_a$ for $1 \leq a < n$ verify the same braid relations as the $s_a \in \mathfrak{S}_n$, for each $w \in\mathfrak{S}_n$ there is a well-defined element $g_w \coloneqq g_{a_1} \cdots g_{a_r} \in \widehat{\Y}_{d, n}^{\tuple{\Lambda}}(1)$ which does not depend on the reduced expression $w = s_{a_1} \cdots s_{a_r}$.

\begin{proposition}
\label{proposition:degenerate_yh_finite_dimensional}
The algebra $\widehat{\Y}_{d, n}^{\tuple{\Lambda}}(1)$ is a finite-dimensional $F$-vector space and a family of generators is given by the elements $f_w x_1^{u_1} \cdots x_n^{u_n} t_1^{v_1} \cdots t_n^{v_n} $ for $w \in \mathfrak{S}_n, u_a \in \{0, \dots, \ell(\tuple{\Lambda}) - 1\}$ and $v_a \in J$.
\end{proposition}

\begin{proof}
We use a similar method to \cite{ArKo, OgPA}. As the unit element belongs to the above family, it suffices to prove that the $F$-vector space $V$ spanned by these elements is stable under (right-)multiplication by the generators of $\widehat{\Y}_{d, n}^{\tuple{\Lambda}}(1)$.

Let us consider $\alpha \coloneqq f_w x_1^{u_1} \cdots x_n^{u_n} t_1^{v_1} \cdots t_n^{v_n}$ as in the proposition.
By \eqref{degenerate_relation:ordre_t_a} and \eqref{degenerate_relation:ta_tb} the element $\alpha t_a$ remains in $V$.
Moreover, writing (by \eqref{degenerate_relation:t_b_f_a} and \eqref{degenerate_relation:fa_xb}):
\[
\alpha f_a = f_w x_1^{u_1} \cdots x_{a-1}^{u_{a-1}} \left(x_a^{u_a} x_{a+1}^{u_{a+1}} f_a \right)x_{a+2}^{u_{a+2}} \cdots x_n^{u_n} t_1^{v_1} \cdots t_n^{v_n},
\]
and using \eqref{degenerate_equation:fa_xau_xa+1v} we conclude that $\alpha f_a \in V$, noticing that the element
\[
x_1^{u_1} \cdots x_{a-1}^{u_{a-1}} \left(e_a x_a^{u'_a} x_{a+1}^{u'_{a+1}}\right) x_{a+2}^{u_{a+2}} \cdots x_n^{u_n} t_1^{v_1} \cdots t_n^{v_n}
\]
belongs to $V$ for every $0 \leq u'_a, u'_{a+1} < \ell(\tuple{\Lambda})$.
Finally, according to \eqref{degenerate_equation:xa+1_fa_xa}, to prove that $\alpha x_a$ remains in $V$ it suffices now to prove that $\alpha x_1 \in V$, but this is clear by \eqref{degenerate_relation:xr_xs}, \eqref{degenerate_relation:xa_tb} and \eqref{degenerate_relation:x1_cyclo}.
\end{proof}

Let now $M$ be a finite-dimensional $\widehat{\Y}_{d, n}^{\tuple{\Lambda}}(1)$-module; it is a finite-dimensional $F$-vector space thanks to Proposition~\ref{proposition:degenerate_yh_finite_dimensional}. 
By \eqref{degenerate_relation:x1_cyclo}, the eigenvalues of $x_1$ on $M$ belong to $I$. We prove in Lemma~\ref{lemma:degenerate_eigenvalues_I} that all the $x_a$ have in fact their eigenvalues in $I$: this is the degenerate analogue of \cite[Lemma 5.2]{CuWa}, which we used in \textsection\ref{subsection:definition_images_quiverhecke_generators}.
 
\begin{lemme}
\label{lemme:degenerate_intertw}
We have:
\begin{gather*}
x_a \phi_a = \phi_a x_{a+1},
\\
\phi_a^2 = (x_{a+1} - x_a - e_a)(x_a - x_{a+1} - e_a),
\end{gather*}
where $\phi_a$ is the ``intertwining operator'' defined by:
\[
\phi_a \coloneqq f_a (x_a - x_{a+1}) + e_a.
\]
\end{lemme} 

\begin{proof}
These are straightforward calculations. We have, using \eqref{degenerate_relation:fa_xa+1}:
\begin{align*}
x_a \phi_a
&= (f_a x_{a+1} - e_a)(x_a - x_{a+1}) + x_a e_a
\\
&= f_a (x_a - x_{a+1}) x_{a+1} + x_{a+1} e_a
\\
x_a \phi_a &= \phi_a x_{a+1},
\end{align*}
and:
\begin{align*}
\phi_a^2
&= f_a (x_a - x_{a+1})f_a (x_a - x_{a+1}) + 2 f_a (x_a - x_{a+1}) e_a + e_a
\\
&= f_a (f_a x_{a+1} - e_a - f_a x_a - e_a)(x_a - x_{a+1}) + 2 f_a(x_a - x_{a+1})e_a + e_a
\\
&= (x_{a+1} - x_a)(x_a - x_{a+1}) + e_a
\\
\phi_a^2 &= (x_{a+1} - x_a - e_a)(x_a - x_{a+1} - e_a).
\end{align*}
\end{proof}
 
\begin{lemme}
\label{lemma:degenerate_eigenvalues_I}
The eigenvalues of $x_a$ belong to $I$ for every $1 \leq a \leq n$.
\end{lemme}

\begin{proof}
We proceed by induction on $a$. The proposition is true for $a = 1$; we suppose that it is true for some $1 \leq a < n$.
Let $\lambda$ be an eigenvalue of $x_{a+1}$ (in a algebraic closure of $F$). As the family $\{x_a, x_{a+1}, e_a\}$ is commutative, we can find a common eigenvector $v$ in the eigenspace of $x_{a+1}$ associated with $\lambda$: we have $x_a v = i v$ and $e_a v = \delta v$ for some $i \in I$ (by induction hypothesis) and $\delta \in \{0, 1\}$ (since $e_a^2 = e_a$). We distinguish now whether $\phi_a v$ vanishes or not:
\begin{itemize}
\item if $\phi_a v \neq 0$, we get by Lemma~\ref{lemme:degenerate_intertw}:
\[
x_a(\phi_a v) = \phi_a (x_{a+1}v) = \lambda \phi_a v,
\]
hence $\lambda$ is an eigenvalue for $x_a$ and by induction hypothesis we get $\lambda\in I$;
\item if $\phi_a v = 0$, by the same lemma we have:
\[
\phi_a^2 v = (\lambda - i - \delta)(i - \lambda - \delta) v = 0,
\]
hence $\lambda = i \pm \delta \in I$.
\end{itemize} 
\end{proof}

\subsection{Quiver Hecke generators of \texorpdfstring{$\widehat{\Y}_{d, n}^{\tuple{\Lambda}}(1)$}{YdnLambda(1)}}

We proceed as in Section~\ref{section:quiver_hecke_generators}: we define some central idempotents, then some ``quiver Hecke generators'' on which we check the defining relations of $\H_{\alpha}^{\tuple{\Lambda}}(\Gamma)$. The proofs are entirely similar to the non-degenerate case (even easier; note that once again the ``hard work'' has been made in \cite{BrKl}), hence we won't write them down. However, we will still define the different involved elements.

\subsubsection{Image of \texorpdfstring{$e(\tuple{i}, \tuple{j})$}{e(i, j)}}

Let $M$ be a finite-dimensional $\widehat{\Y}_{d, n}^{\tuple{\Lambda}}(1)$-module. We know that the $t_a$ are diagonalizable with eigenvalues in $J$. Hence, recalling Lemma~\ref{lemma:degenerate_eigenvalues_I}, we can write (recall that the family $\{X_a, t_a\}_{1 \leq a \leq n}$ is commutative):
\[
M = \bigoplus_{(\tuple{i}, \tuple{j}) \in I^n \times J^n} M(\tuple{i}, \tuple{j}),
\]
with:
\[
M(\tuple{i}, \tuple{j}) \coloneqq \left\lbrace v \in M : (x_a - i_a)^N v = (t_a - \xi^{j_a})v = 0 \text{ for all } 1 \leq a \leq n\right\rbrace,
\]
with $N \gg 0$; since $M$ is a finite-dimensional $F$-vector space, only finitely many $M(\tuple{i}, \tuple{j})$ are non-zero.
Considering once again the family of projections $\{e(\tuple{k})\}_{\tuple{k} \in K^n}$ associated with $M = \oplus_{\tuple{k} \in K^n} M(\tuple{k})$, we define for $\alpha \comp_K n$:
\[
e(\alpha) \coloneqq \sum_{\tuple{k} \in K^{\alpha}} e(\tuple{k}),
\]
and we set $\widehat{\Y}_{\alpha}^{\tuple{\Lambda}}(1) \coloneqq e(\alpha) \widehat{\Y}_{d, n}^{\tuple{\Lambda}}(1)$. 
We can now define, for $\tuple{i} \in I^{\alpha}$ and $\tuple{j} \in J^{\alpha}$:
\begin{equation}
\label{equation:degenerate_e(i)_e(j)}
\begin{aligned}
e(\alpha)(\tuple{i}) &\coloneqq \sum_{\tuple{j} \in J^{\alpha}} e(\alpha)e(\tuple{i}, \tuple{j}),
\\
e(\alpha)(\tuple{j}) &\coloneqq \sum_{\tuple{i} \in I^{\alpha}} e(\alpha) e(\tuple{i}, \tuple{j}).
\end{aligned}
\end{equation}
In particular, with $e(\alpha)(\tuple{i})$ for $d = 1$ we recover the element $e(\tuple{i})$ of \cite[\textsection 3.1]{BrKl}.

From now on, unless mentioned otherwise we always work in $\widehat{\Y}_{\alpha}^{\tuple{\Lambda}}(1)$; every relation should be multiplied by $e(\alpha)$ and we write $e(\tuple{i})$ and $e(\tuple{j})$ without any $(\alpha)$.

\begin{lemme}
If $1 \leq a < n$ and $\tuple{j} \in J^{\alpha}$ is such that $j_a \neq j_{a+1}$ then we have:
\begin{gather*}
f_a x_{a+1} e(\tuple{j}) = x_a f_a e(\tuple{j}),
\\
x_{a+1} f_a e(\tuple{j}) = f_a x_a e(\tuple{j}).
\end{gather*}
\end{lemme}

\begin{lemme}
For $1 \leq a < n$ and $\tuple{j} \in J^{\alpha}$ we have $f_a e(\tuple{j}) = e(s_a\cdot \tuple{j}) f_a$. In particular, if $j_a = j_{a+1}$ then $f_a$ and $e(\tuple{j})$ commute. Moreover, if $j_a \neq j_{a+1}$ then $f_a e(\tuple{i}, \tuple{j}) = e(s_a\cdot (\tuple{i}, \tuple{j})) f_a$.
\end{lemme}

\begin{remarque}[About Brundan and Kleshchev's proof - \theBKproof]
\stepcounter{BKproof}
\label{remark:about_BK_degenerate_quiver_generators}
Let $1 \leq a < n$; if $\tuple{j} \in J^{\alpha}$ verifies $j_a = j_{a+1}$, when a proof in \cite[\textsection 3.3]{BrKl} needs only the elements $f_a$, $x_b$, $e(\tuple{i})$  and the corresponding relations in $\widehat{\H}_{\alpha}^{\tuple{\Lambda}}(1)$, we claim that the same proof holds in $e(\tuple{j}) \widehat{\Y}_{\alpha}^{\tuple{\Lambda}}(1)e(\tuple{j})$. We extend this claim to the case $j_a = j_{a+1} = j_{a+2}$.
\end{remarque}

\subsubsection{Image of \texorpdfstring{$y_a$}{ya}}

We define the following elements of $\widehat{\Y}_{\alpha}^{\tuple{\Lambda}}(1)$ for $1 \leq a \leq n$:
\[
y_a \coloneqq \sum_{\tuple{i} \in I^{\alpha}} (x_a - i_a)e(\tuple{i}) \in \widehat{\Y}_{\alpha}^{\tuple{\Lambda}}(1).
\]

When $d = 1$ we recover the elements defined in \cite[\textsection 3.3]{BrKl}. 
These elements are nilpotent: we will be able to make calculations in the ring $F[[y_1, \dots, y_n]]$.

\begin{lemme}
For $\tuple{j} \in J^{\alpha}$ such that $j_a \neq j_{a+1}$ we have:
\begin{gather*}
f_a y_{a+1} e(\tuple{j}) = y_a f_a e(\tuple{j}),
\\
y_{a+1} f_a e(\tuple{j}) = f_a y_a e(\tuple{j}).
\end{gather*}
\end{lemme}

\subsubsection{Image of \texorpdfstring{$\psi_a$}{psia}}
\label{subsubsection:degenerate_image_psia}

We first define some elements $p_a(\tuple{i}, \tuple{j}) \in F[[y_a, y_{a+1}]]$ for $1 \leq a < n$ and $(\tuple{i}, \tuple{j}) \in K^{\alpha}$ by:
\[
p_a(\tuple{i}, \tuple{j}) \coloneqq
\begin{cases}
\begin{cases}
1 & \text{if } i_a = i_{a+1},
\\
(i_a - i_{a+1} + y_a - y_{a+1})^{-1}
& \text{if } i_a \neq i_{a+1},
\end{cases} & \text{if } j_a = j_{a+1},
\\
0 & \text{if } j_a \neq j_{a+1},
\end{cases}
\]
and then some invertible elements $q_a(\tuple{i}, \tuple{j}) \in F[[y_a, y_{a+1}]]^{\times}$ for $1 \leq a < n$ and $(\tuple{i}, \tuple{j}) \in K^{\alpha}$ by:
\[
q_a(\tuple{i}, \tuple{j}) \coloneqq \begin{cases}
\begin{cases}
1 + y_{a+1} - y_a & \text{if } i_a = i_{a+1},
\\
1 - p_a(\tuple{i}, \tuple{j}) & \text{if } i_a \nrelbar i_{a+1},
\\
\left(1 - p_a(\tuple{i}, \tuple{j})^2\right)/(y_{a+1} - y_a) & \text{if } i_a \to i_{a+1},
\\
1 & \text{if } i_a \leftarrow i_{a+1},
\\
(1 - p_a(\tuple{i}, \tuple{j}))/(y_{a+1} - y_a) & \text{if } i_a \leftrightarrows i_{a+1},
\end{cases}
& \text{if } j_a = j_{a+1},
\\
1 &\text{if } j_a \neq j_{a+1}.
\end{cases}
\]

\begin{remarque}
As in \cite{BrKl}, the explicit expression of $q_a(\tuple{i}, \tuple{j})$ does not really matter; we only need some properties verified by these power series.
\end{remarque}

\begin{lemme} We have:
\begin{gather*}
p_{a+1}^{s_a}(\tuple{i}, \tuple{j}) = p_a^{s_{a+1}}(s_{a+1} s_a\cdot(\tuple{i}, \tuple{j})),
\\
q_{a+1}^{s_a}(\tuple{i}, \tuple{j}) = q_a^{s_{a+1}}(s_{a+1} s_a\cdot (\tuple{i}, \tuple{j})).
\end{gather*}
\end{lemme}

We now introduce the following element of $\widehat{\Y}_{\alpha}^{\tuple{\Lambda}}(1)$:
\[
\varphi_a \coloneqq f_a + \sum_{\substack{(\tuple{i}, \tuple{j}) \in K^{\alpha} \\ i_a \neq i_{a+1} \\ j_a = j_{a+1}}} (x_a - x_{a+1})^{-1} e(\tuple{i}, \tuple{j}) + \sum_{\substack{(\tuple{i}, \tuple{j}) \in K^{\alpha} \\ i_a = i_{a+1} \\ j_a = j_{a+1}}} e(\tuple{i}),
\]
where $(x_a - x_{a+1})^{-1} e(\tuple{k})$ denotes the inverse of $(x_a - x_{a+1})e(\tuple{k})$ in $e(\tuple{k}) \widehat{\Y}_{\alpha}^{\tuple{\Lambda}}(1)e(\tuple{k})$. In particular, we have:
\begin{gather*}
\varphi_a e(\tuple{j}) = f_a e(\tuple{j}) \qquad \text{if } j_a \neq j_{a+1},
\\
\varphi_a = \sum_{\tuple{k} \in K^{\alpha}} (f_a + p_a(\tuple{k}))e(\tuple{k}).
\end{gather*}
Moreover, for $d = 1$ the element $\varphi_a$ is the ``intertwining element'' defined in \cite[\textsection 3.2]{BrKl}.

\begin{lemme}
We have the following properties:
\begin{align}
\varphi_a e(\tuple{j}) &= e(s_a\cdot \tuple{j})\varphi_a,
\\
\varphi_a e(\tuple{i}, \tuple{j}) &= e(s_a\cdot (\tuple{i}, \tuple{j}))\varphi_a,
\\
\phantom{\forall b \neq a, a+1 \qquad}\varphi_a x_b &= x_b \varphi_a  \qquad\forall  b \neq a, a+1,
\\
\phantom{\forall b \neq a, a+1 \qquad}\varphi_a y_b &= y_b \varphi_a  \qquad\forall b \neq a, a+1,
\\
\phantom{\forall |b-a| > 1 \qquad}\varphi_a q_b(\tuple{k}) &= q_b(\tuple{k}) \varphi_a \qquad\forall |b-a| > 1,
\\
\phantom{\forall |b-a| > 1 \qquad} \varphi_a \varphi_b &= \varphi_b\varphi_a \qquad \forall |b-a|>1.
\end{align}
\end{lemme}

Our element $\psi_a$ is defined for $1 \leq a < n$ by:
\[
\psi_a \coloneqq \sum_{\tuple{k} \in K^{\alpha}} \phi_a q_a(\tuple{k})^{-1} e(\tuple{k}) \in \widehat{\Y}_{\alpha}^{\tuple{\Lambda}}(1).
\]
When $d = 1$ this element $\psi_a$ corresponds to the $\psi_a$ of \cite[\textsection 3.3]{BrKl}. Note finally that for $\tuple{j} \in J^{\alpha}$ we have:
\[
\psi_a e(\tuple{j}) = f_a e(\tuple{j}) \qquad \text{if } j_a \neq j_{a+1}.
\]

\subsubsection{Check of the defining relations}
\label{subsubsection:degenerate_check_kl_generators}
\begin{theoreme}
\label{theoreme:degenerate_kl_generators}
The elements $y_1, \dots, y_n, \psi_1, \dots, \psi_{n-1}$ and $e(\tuple{k})$ for $\tuple{k} \in K^{\alpha}$ verify the defining relations \eqref{relation:quiver_cyclo_y1}--\eqref{relation:quiver_tresse} of $\H_{\alpha}^{\tuple{\Lambda}}(\Gamma)$.
\end{theoreme}

The painstaking verification is exactly the same as in \textsection\ref{subsection:check_kl_generators}: we apply Remark~\ref{remark:about_BK_degenerate_quiver_generators} on the proof of \cite[Theorem 3.2]{BrKl} for the cases $j_a = j_{a+1}$, and when $j_a \neq j_{a+1}$ then entirely similar (even the same) relations as in \textsection\ref{subsection:check_kl_generators} are verified.  Note two small differences with the proof in \textsection\ref{subsection:check_kl_generators}:
\begin{itemize}
\item we shall write $(x_a - x_b)$ instead of $(1-q) (1 - X_a X_b^{-1})$;
\item the elements $f_{a, \tuple{j}}$ are equal to $1$.
\end{itemize}

\subsection{Degenerate Yokonuma--Hecke generators of \texorpdfstring{$\H_{\alpha}^{\tuple{\Lambda}}(\Gamma)$}{HalphaLambdaGamma}}

We proceed as in Section~\ref{section:yokonuma_hecke_generators}. Once again, the proofs are entirely similar to the non-degenerate case, hence we do not write them down.

First of all, since the elements $y_1, \dots, y_n \in \H_{\alpha}^{\tuple{\Lambda}}(\Gamma)$ are nilpotent we can consider power series in these variables. Hence, the quantities $p_a(\tuple{k})$, $q_a(\tuple{k})$ that we defined in \textsection\ref{subsubsection:degenerate_image_psia} are also well-defined as elements of $\H_{\alpha}^{\tuple{\Lambda}}(\Gamma)$. We define finally as in \eqref{equation:degenerate_e(i)_e(j)} the elements $e(\tuple{i})$ and $e(\tuple{j})$ of $\H_{\alpha}^{\tuple{\Lambda}}(\Gamma)$ for $\tuple{i} \in I^{\alpha}$ and $\tuple{j} \in J^{\alpha}$.

We recall that $\xi$ is a primitive $d$th root of unity in $F$. Our ``degenerate Yokonuma--Hecke generators'' of $\H_{\alpha}^{\tuple{\Lambda}}(\Gamma)$ are given below.
\begin{align*}
\phantom{\text{for } 1 \leq a < n} && f_a &\coloneqq \sum_{\tuple{k} \in K^{\alpha}}\left(\psi_a q_a(\tuple{k}) - p_a(\tuple{k})\right)e(\tuple{k}) & \text{for } 1 &\leq a < n
\\
\phantom{\text{for } 1 \leq a \leq n}& & t_a &\coloneqq \sum_{\tuple{j} \in J^{\alpha}} \xi^{j_a} e(\tuple{j})  &\text{for } 1 &\leq a \leq n
\\
\phantom{\text{for } 1 \leq a \leq n}&& x_a & \coloneqq \sum_{\tuple{i} \in I^{\alpha}} (y_a + i_a) e(\tuple{i}) &\text{for } 1 &\leq a \leq n
\end{align*}

When $d = 1$, the element $f_a$ (respectively $x_a$) is the element $s_a$ (resp. $x_a$) of \cite[\textsection 3.4]{BrKl}.

\begin{remarque}[About Brundan and Kleshchev's proof - \theBKproof]
\stepcounter{BKproof}
\label{remark:about_BK_degenerate_YH_generators}
Let $1 \leq a < n$; if $\tuple{j} \in J^{\alpha}$ verifies $j_a = j_{a+1}$, when a proof in \cite[\textsection 3.4]{BrKl} needs only the elements $\psi_a e(\tuple{j}), y_b e(\tuple{j}), e(\tuple{i}, \tuple{j})$ and the corresponding relations in $\H_{\alpha}^{\tuple{\Lambda}}(\Gamma_e)$, we claim that the same proof holds in $e(\tuple{j}) \H_{\alpha}^{\tuple{\Lambda}}(\Gamma)e(\tuple{j})$. We extend this claim to the case $j_a = j_{a+1} = j_{a+2}$.
\end{remarque}

Finally, similarly to \textsection\ref{subsubsection:degenerate_check_kl_generators}  we have the following theorem. Once again the check of the various relations is exactly the same as in \textsection\ref{subsection:check_yh_generators}.

\begin{theoreme}
\label{theoreme:degenerate_yh_generators}
The elements $f_1, \dots, f_{n-1}, t_1, \dots, t_n, x_1, \dots, x_n$ satisfy the defining relations \eqref{degenerate_relation:ordre_t_a}--\eqref{degenerate_relation:x1_cyclo} of $\widehat{\Y}_{d, n}^{\tuple{\Lambda}}(1)$.
\end{theoreme}

\subsection{Isomorphism theorem}

We give now the degenerate version of Theorem~\ref{theorem:main}.

\begin{theoreme}
\label{theorem:main_degenerate}
There is a presentation of the degenerate cyclotomic Yokonuma--Hecke algebra $\widehat{\Y}_\alpha^{\tuple{\Lambda}}(1)$ given by the generators \eqref{equation:quiver_generators_alpha} and the relations \eqref{relation:quiver_sum_alpha_e(i)}--\eqref{relation:quiver_ya+1_psia}, \eqref{relation:quiver_psia^2}--\eqref{relation:quiver_tresse} and \eqref{relation:quiver_cyclo_y1}, that is, we have an algebra isomorphism:
\[
\H_n^{\tuple{\Lambda}}(\Gamma) \overset{\sim}\to \widehat{\Y}_{d, n}^{\tuple{\Lambda}}(1).
\]
\end{theoreme}

The proof of this theorem is entirely similar to the one of Theorem~\ref{theorem:main}.
In particular, by Theorem~\ref{theoreme:degenerate_kl_generators} we can define an algebra homomorphism $\rho : \H_{\alpha}^{\tuple{\Lambda}}(\Gamma) \to \widehat{\Y}_{\alpha}^{\tuple{\Lambda}}(1)$ and by Theorem~\ref{theoreme:degenerate_yh_generators} we can define another algebra homomorphism $\sigma : \widehat{\Y}_{d, n}^{\tuple{\Lambda}}(1) \to \H_{\alpha}^{\tuple{\Lambda}}(\Gamma)$. From the inclusion $\widehat{\Y}_{\alpha}^{\tuple{\Lambda}}(1) \subseteq \widehat{\Y}_{d, n}^{\tuple{\Lambda}}(1)$ we deduce an algebra homomorphism $\sigma : \widehat{\Y}_{\alpha}^{\tuple{\Lambda}}(1) \to \H_{\alpha}^{\tuple{\Lambda}}(\Gamma)$. We prove then that $\rho$ and $\sigma$ are inverse homomorphisms, taking the images of the different defining generators.

Together with Theorem~\ref{theorem:main} we get the following corollaries (cf. \cite[Corollary 1.3]{BrKl}).

\begin{corollaire}
\label{corollary:main_improvement}
If $q$ and $q'$ are two elements of $F^{\times}$ with $\mathrm{char}_q(F) = \mathrm{char}_{q'}(F)$ then $\widehat{\Y}_{d, n}^{\tuple{\Lambda}}(q)$ and $\widehat{\Y}_{d, n}^{\tuple{\Lambda}}(q')$ are isomorphic algebras.
\end{corollaire}

\begin{corollaire}
\label{corollary:link_nondegenerate_degenerate}
If $F$ has characteristic $\mathrm{char}_q(F)$ then the cyclotomic Yokonuma--Hecke algebra $\widehat{\Y}_{d, n}^{\tuple{\Lambda}}(q)$ is isomorphic to its rational degeneration $\widehat{\Y}_{d, n}^{\tuple{\Lambda}}(1)$. This applies in particular when $F$ has characteristic $0$ and $q$ is generic.
\end{corollaire}

\section{Another approach to the result}
\label{section:disjoint_guiver}

In \cite{Lu, JaPA, PA}, the authors have proved the following algebra isomorphism:
\begin{equation}
\label{equation:isom_JaPA}
\widehat{\Y}_{d, n}^{\tuple{\Lambda}}(q) \simeq \bigoplus_{\lambda \comp_d n} \mathrm{Mat}_{m_\lambda} \widehat{\H}_\lambda^{\tuple{\Lambda}}(q),
\end{equation}
with $F = \mathbb{C}, e = \infty,\widehat{\H}_\lambda^{\tuple{\Lambda}}(q) \coloneqq \widehat{\H}_{\lambda_1}^{\tuple{\Lambda}}(q) \otimes \cdots \otimes \widehat{\H}_{\lambda_d}^{\tuple{\Lambda}}(q)$, where we write $\comp_d$ instead of $\comp_J$ and where $m_\lambda$ are some integers (see \eqref{equation:definition_mlambda}). We will see how we can relate this isomorphism to our previous work. To that extent, we prove in Theorem~\ref{theorem:disjoint_guiver}  a general result on quiver Hecke algebras in the case where the quiver is given by a disjoint union of full subquivers. The isomorphism is built from the following map (see \eqref{equation:definition_philambda_psilambda}):
\[
\Psi_{\mathfrak{t}', \mathfrak{t}} : w \mapsto (\psi_{\pi_{\mathfrak{t}'}} w \psi_{\pi_\mathfrak{t}^{-1}}) E_{\mathfrak{t}', \mathfrak{t}},
\]
for $w \in e(\mathfrak{t}')\H_n(Q)e(\mathfrak{t})$, where $E_{\mathfrak{t}', \mathfrak{t}}$ are elementary matrices. In particular:
\begin{itemize}
\item in \textsection\ref{subsubsection:young_subgroups} we introduce the elements $\pi_\mathfrak{t}$, the minimal-length representatives of the right cosets of $\mathfrak{S}_n$ under the action of the Young subgroup $\mathfrak{S}_\lambda$ where $\lambda \comp_d n$: this will lead to some calculations which will only be needed to explicit our homomorphism $\Psi_{\mathfrak{t}', \mathfrak{t}}$;
\item in \textsection\ref{subsection:about_psit} we will study the elements $\psi_{\pi_\mathfrak{t}}$, and we will go on with the previous calculations.
\end{itemize}

\subsection{Setting}

Let $\K$ be a \emph{finite} set; we recall that $d, n \in \mathbb{N}^*$ and $J = \mathbb{Z} / d\mathbb{Z} \simeq \{1, \dots, d\}$. We consider a partition of $\K$ into $d$ parts $\K = \sqcup_{j \in J} \K_j$. We recall that the left action of $w \in \mathfrak{S}_n$ on tuples is given by $w \cdot (x_1, \dots, x_n) \coloneqq (x_{w^{-1}(1)}, \dots, x_{w^{-1}(n)})$.
We may use some elementary theory about Coxeter groups: we refer for instance to \cite{GePf} or \cite{Hum}. In particular, in that context we will denote by $\ell$ the usual length function $\mathfrak{S}_n \to \mathbb{N}$. Finally, let us mention that in this section, we will write $\mathfrak{t}$ for the elements of $J^n$.

\subsubsection{Labellings and shapes}

Let $\lambda = (\lambda_j)_{1 \leq j \leq d}\comp_d n$ be a $d$-composition of $n$: recall that it means $\lambda_j \geq 0$ and $\lambda_1 + \dots + \lambda_d = n$. We  define the integers $\boldsymbol{\lambda}_1, \dots, \boldsymbol{\lambda}_d$, given by $\boldsymbol{\lambda}_j \coloneqq \lambda_1 + \dots + \lambda_j$ for $j \in J$. In particular, $\boldsymbol{\lambda}_1 = \lambda_1$ and $\boldsymbol{\lambda}_d = n$; we also set $\boldsymbol{\lambda}_0 \coloneqq 0$. From now on, the letter $\lambda$ always stands for a $d$-composition of $n$.

\begin{definition}
\label{definition:label_shape}
Let $\tuple{k} \in \K^n$ and $\mathfrak{t} \in J^n$.
\begin{itemize}
\item We say that $\tuple{k}$ is a \emph{labelling} of $\mathfrak{t}$  when the following rule is satisfied:
\[
\forall a \in \{1, \dots, n\}, k_a \in \K_{\mathfrak{t}_a},
\]
that is:
\[
\forall a \in \{1, \dots, n\}, \forall j \in J, k_a \in \K_j \iff \mathfrak{t}_a = j.
\]
We write $\K^\mathfrak{t}$ for the elements $\K^n$ which are labellings of $\mathfrak{t}$.

\item We say that $\mathfrak{t}$ has \emph{shape} $\lambda \comp_d n$ and we write $[\mathfrak{t}] = \lambda$ if for all $j \in J$ there are exactly $\lambda_j$ components of $\mathfrak{t}$ equal to $j$, that is:
\[
\forall j \in J,\#\big\lbrace a \in \{1, \dots, n\} : \mathfrak{t}_a  = j\big\rbrace = \lambda_j.
\]
We write $J^\lambda$ for the elements $J^n$ with shape $\lambda$.
\end{itemize}
\end{definition}


The sets $J^\lambda$ are exactly the orbits of $J^n$ under the action of $\mathfrak{S}_n$, in particular $[w \cdot \mathfrak{t}] = [\mathfrak{t}]$ for every $w \in \mathfrak{S}_n$ and $\mathfrak{t} \in J^n$. Moreover, the cardinality of $J^\lambda$ is:
\begin{equation}
\label{equation:definition_mlambda}
m_{\lambda} \coloneqq \frac{n!}{\lambda_1 ! \dots \lambda_d!}.
\end{equation}

We write $\mathfrak{t}^{\lambda} \in J^\lambda$ for the trivial element of shape $\lambda$, given by:
\begin{equation}
\label{equation:definition_sigmalambda}
\forall a \in \{1, \dots, n\}, \forall j \in J, \mathfrak{t}^\lambda_a = j \iff \boldsymbol{\lambda}_{j - 1} < a \leq \boldsymbol{\lambda}_j,
\end{equation}
that is:
\[
\mathfrak{t}^\lambda \coloneqq (1, \dots, 1, \dots, d, \dots, d),
\]
where each $j \in J$ appears $\lambda_j$ times. Note that $\K^{\mathfrak{t}^\lambda} \simeq \K_1^{\lambda_1} \times \cdots \times \K_d^{\lambda_d}$.

\subsubsection{Young subgroups}
\label{subsubsection:young_subgroups}

Most results of this section are well-known; however, since in the literature they are stated either for a left or a right action (see Remark~\ref{remark:young-mathas}), for the convenience of the reader we state all of them with a left action. 
We remind the reader that some calculations made here will only be used in \textsection\ref{subsubsection:an_application}, namely with Lemmas~\ref{lemma:pisigma_sa_sigma} and \ref{lemma:reduced_expression_pisigma_sa}.

Let $\lambda \comp_d n$; the following group:
\[
\mathfrak{S}_{\lambda} \coloneqq \mathfrak{S}_{\lambda_1} \times \dots \times \mathfrak{S}_{\lambda_d},
\]
can be seen as a subgroup of $\mathfrak{S}_n$ (the ``Young subgroup''), where we consider that $\mathfrak{S}_{\lambda_j} \simeq \mathfrak{S}(\{\boldsymbol{\lambda}_{j-1} + 1, \dots, \boldsymbol{\lambda}_j\})$.
Recall that:
\begin{itemize}
\item the group $\mathfrak{S}_n$ (resp. $\mathfrak{S}_{\lambda_j}$) is generated by $s_1, \dots, s_{n-1}$ (resp. $s_{\boldsymbol{\lambda}_{j-1} + 1}, \dots, s_{\boldsymbol{\lambda}_j - 1}$);
\item the subgroup $\mathfrak{S}_{\lambda}$ is generated by all the $s_a$ for $a \in \{1, \dots, n\} \setminus \{\boldsymbol{\lambda}_1, \dots, \boldsymbol{\lambda}_d\}$.
\end{itemize} 
In particular:
\begin{equation}
\label{equation:composantes_young_commutent}
\forall j \neq j', \forall (w_j, w_{j'}) \in \mathfrak{S}_{\lambda_j} \times \mathfrak{S}_{\lambda_{j'}}, w_j w_{j'} = w_{j'} w_j \text{ in } \mathfrak{S}_n.
\end{equation}

\begin{remarque}
\label{remark:reduced_expression_Slambda}
If $w = s_{a_1} \cdots s_{a_r} \in \mathfrak{S}_\lambda$ is a reduced expression, up to a reindexation we know by \eqref{equation:composantes_young_commutent} that there is a sequence $0 \eqqcolon r_0 \leq r_1 \leq \cdots \leq r_{d-1} \leq r_d \coloneqq r$ such that for each $j \in J$, the word $s_{a_{r_{j-1} + 1}} \cdots s_{a_{r_j}}$ is reduced and lies in $\mathfrak{S}_{\lambda_j}$. The converse is also true: if for each $j \in J$ we have a reduced word $s_{a_{r_{j-1} + 1}} \cdots s_{a_{r_j}} \in \mathfrak{S}_{\lambda_j}$ then their concatenation $s_1 \cdots s_r \in \mathfrak{S}_\lambda$ is reduced.
\end{remarque}

The following proposition is straightforward.

\begin{proposition}
\label{proposition:young_stabiliser_sigma_lambda}
The stabiliser of $\mathfrak{t}^\lambda$ under the action of $\mathfrak{S}_n$ is exactly $\mathfrak{S}_{\lambda}$.
\end{proposition}

We now study the right cosets in $\mathfrak{S}_n$ for the (left) action of $\mathfrak{S}_{\lambda}$.

\begin{lemme}
\label{lemma:equation_coset}
Two words $w, w' \in \mathfrak{S}_n$ are in the same right coset if and only if $w^{-1} \cdot \mathfrak{t}^\lambda = w'^{-1} \cdot \mathfrak{t}^\lambda$.
\end{lemme}

The proof  is straightforward from Proposition~\ref{proposition:young_stabiliser_sigma_lambda}.
An element $C \in \mathfrak{S}_{\lambda}\backslash\mathfrak{S}_n$ is thus determined by the constant value $\mathfrak{t} \coloneqq w^{-1}\cdot \mathfrak{t}^\lambda \in J^\lambda$ for $w \in C$: we write $C_{\mathfrak{t}}$ for the coset $C$ (as each $\mathfrak{t} \in J^n$ has a unique shape, we do not need to precise the underlying composition in the indexation). Noticing that  $m_{\lambda} = |\mathfrak{S}_n| / |\mathfrak{S}_{\lambda}|$, we conclude that the cosets are parametrised by the whole set $J^\lambda$, that is, $\mathfrak{S}_{\lambda}\backslash \mathfrak{S}_n = \{C_{\mathfrak{t}}\}_{\mathfrak{t} \in J^\lambda}$.

We know by \cite[Proposition 2.1.1]{GePf} that each coset $C_{\mathfrak{t}}$ has a unique minimal length element: we write $\pi_\mathfrak{t} \in C_{\mathfrak{t}}$ for this unique element.  In particular, since Lemma~\ref{lemma:equation_coset} gives:
\begin{equation}
\label{equation:equation_coset}
\forall w \in \mathfrak{S}_n, w \in C_{\mathfrak{t}} \iff w \cdot \mathfrak{t} = \mathfrak{t}^\lambda,
\end{equation}
we get the following proposition.

\begin{proposition}
\label{proposition:pi_sigma_lambda_sort}
The element $\pi_\mathfrak{t}$ is the unique minimal length element of $\mathfrak{S}_n$ such that:
\begin{equation}
\label{equation:pisigmalambda_sort}
\pi_\mathfrak{t} \cdot \mathfrak{t} = \mathfrak{t}^\lambda.
\end{equation}
\end{proposition}

\begin{remarque}
\label{remark:decomposition_cosets}
The decomposition into right cosets is obtained in the following way. Given $w \in \mathfrak{S}_n$, we know that $w$ belongs to the coset $C_{\mathfrak{t}}$ with $\mathfrak{t} \coloneqq w^{-1} \cdot \mathfrak{t}^\lambda$.  The element $\widetilde{w}^{-1} \coloneqq \pi_\mathfrak{t} w^{-1}$ stabilises $\mathfrak{t}^\lambda$, thus lies in $\mathfrak{S}_{\lambda}$ and we have $w =  \widetilde{w} \pi_\mathfrak{t}$. 
\end{remarque}

\begin{proposition}
\label{proposition:explicit_pi}
The elements $\pi_\mathfrak{t}$ are given by:
\[
\forall a \in \{1, \dots, n\},
\pi_\mathfrak{t}(a) = \boldsymbol{\lambda}_{\mathfrak{t}_a-1} + \# \{b \leq a : \mathfrak{t}_b = \mathfrak{t}_a\}.
\]
\end{proposition}

An example is given in Figure~\ref{figure:exemple_pisigma}.
\begin{figure}[h]
\centering
\begin{tikzpicture}[>=angle 90]
\node at (-1, 2) {$\{1, \dots, 6\}$};
\foreach \i in {1,...,6}
	\node at (\i, 2) {$\i$};

\node (sigma) at (-1, 1) {$\mathfrak{t}$};
\node (sigma1) at (1, 1) {$3$};
\node (sigma2) at (2, 1) {$1$};
\node (sigma3) at (3, 1) {$3$};
\node (sigma4) at (4, 1) {$2$};
\node (sigma5) at (5, 1) {$3$};
\node (sigma6) at (6, 1) {$1$};

\node (sigmalambda) at (-1, -1) {$\mathfrak{t}^\lambda$};
\node (sigmalambda1) at (1, -1) {$1$};
\node (sigmalambda2) at (2, -1) {$1$};
\node (sigmalambda3) at (3, -1) {$2$};
\foreach \i in {4,5,6}
	\node (sigmalambda\i) at (\i, -1) {$3$};

\node at (-1, -2) {$\{1, \dots, 6\}$};
\foreach \i in {1,...,6}
	\node at (\i, -2) {$\i$};

\draw[->] (sigma1) -- (sigmalambda4);
\draw[->] (sigma2) -- (sigmalambda1);
\draw[->] (sigma3) -- (sigmalambda5);
\draw[->] (sigma4) -- (sigmalambda3);
\draw[->] (sigma5) -- (sigmalambda6);
\draw[->] (sigma6) -- (sigmalambda2);
\end{tikzpicture}
\caption{The permutation $\pi_\mathfrak{t}$ for $\lambda \coloneqq (2, 1, 3) \comp_3 6$ and $\mathfrak{t} \coloneqq (3, 1, 3, 2, 3, 1)$.}
\label{figure:exemple_pisigma}
\end{figure}
To prove Proposition~\ref{proposition:explicit_pi}, we will use the vocabulary of ``tableaux'' (see for example \cite[\textsection 3.1]{Ma}). As a quick reminder, a \emph{$\lambda$-tableau} $\mathcal{T}$ is a bijection $\{(j, m) \in \mathbb{N}^2 : 1 \leq j \leq d \text{ and } 1 \leq m \leq \lambda_j\} \to \{1, \dots, n\}$; the tableau $\mathcal{T}$ is \emph{row-standard} if in each rows, its entries increase from left to right. Here are two examples of $\lambda$-tableaux, with $\lambda \coloneqq (2, 1, 3) \comp_3 6$:
\[
\young(26,4,135),
\qquad \qquad
\young(32,5,614),
\]
the first only being row-standard.

To any $\mathfrak{t} \in J^\lambda$, we associate  the $\lambda$-tableau $\mathcal{T}_\mathfrak{t}$ given by the following rule: for $j \in J$ and $m \in \{1, \dots, \lambda_j\}$, we label the node $(j, m)$ by the index $a$ of the $m$th occurrence of $j$ in $\mathfrak{t}$, that is, by the integer $a \in \{1, \dots, n\}$ determined by:
\begin{equation}
\label{equation:young_subgroups-label_a}
\mathfrak{t}_a = j \text{ and } \# \{b \leq a : \mathfrak{t}_b = \mathfrak{t}_a\} = m.
\end{equation}
In particular, the tableau $\mathcal{T}_\mathfrak{t}$ is row-standard; conversely, each row-standard $\lambda$-tableau is a $\mathcal{T}_\mathfrak{t}$ for a unique $\mathfrak{t} \in J^\lambda$. With the notation of Figure~\ref{figure:exemple_pisigma}, here are two examples of row-standard $\lambda$-tableaux:
\[
\mathcal{T}_\mathfrak{t} = \young(26,4,135),
\qquad \qquad
\mathcal{T}_{\mathfrak{t}^\lambda} = \young(12,3,456).
\]

We consider the natural left action of the symmetric group $\mathfrak{S}_n$ on the set of $\lambda$-tableaux: if $w \in \mathfrak{S}_n$ and $\mathcal{T}$ is a $\lambda$-tableau, the tableau $w \cdot \mathcal{T}$ is obtained by applying $w$ in each box of $\mathcal{T}$.
 If now $\mathcal{T}$ and $\mathcal{T}'$ are two $\lambda$-tableaux, we write $\mathcal{T} \sim \mathcal{T}'$ if for all $j \in J$, the labels of the $j$th row of $\mathcal{T}$ are a permutation of the labels of the  $j$th row of  $\mathcal{T}'$.

\begin{lemme}
\label{lemma:relation_equivalence_tableau}
For $w \in \mathfrak{S}_n$ and $\mathfrak{t} \in J^\lambda$  we have:
\[
w \cdot \mathcal{T}_\mathfrak{t} \sim \mathcal{T}_{\mathfrak{t}^\lambda} \iff w \cdot \mathfrak{t} = \mathfrak{t}^\lambda.
\]
\end{lemme}

\begin{proof}
If $j \in J$ and $m \in \{1, \dots, \lambda_j\}$, we denote by $a[j, m]$ the label of the box $(j, m)$ of $\mathcal{T}_\mathfrak{t}$; by \eqref{equation:young_subgroups-label_a} we have $\mathfrak{t}_{a[j, m]} = j$. We get:
\begin{align*}
w \cdot \mathcal{T}_\mathfrak{t} \sim \mathcal{T}_{\mathfrak{t}^\lambda}
&\iff
\forall j \in J, \forall m \in \{1, \dots, \lambda_j\}, w(a[j, m]) \in \{\boldsymbol{\lambda}_{j-1} + 1, \dots, \boldsymbol{\lambda}_j\}
\\
&\iff
\forall j \in J, \forall m \in \{1, \dots, \lambda_j\}, \mathfrak{t}^\lambda_{w(a[j, m])} = j
\\
&\iff
\forall j \in J, \forall m \in \{1, \dots, \lambda_j\}, \mathfrak{t}^\lambda_{w(a[j, m])} = \mathfrak{t}_{a[j, m]}
\\
&\iff
\forall a \in \{1, \dots, n\}, \mathfrak{t}^\lambda_{w(a)} = \mathfrak{t}_{a}
\\
&\iff
\forall a \in \{1, \dots, n\}, \mathfrak{t}^\lambda_a = \mathfrak{t}_{w^{-1}(a)}
\\
w \cdot \mathcal{T}_\mathfrak{t} \sim \mathcal{T}_{\mathfrak{t}^\lambda}
&\iff
\mathfrak{t}^\lambda = w \cdot \mathfrak{t},
\end{align*}
as desired.
\end{proof}

\begin{proof}[Proof of Proposition~\ref{proposition:explicit_pi}.]
Let $\mathfrak{t} \in J^\lambda$. There is a unique element $d(\mathfrak{t}) \in \mathfrak{S}_n$ such that $\mathcal{T}_\mathfrak{t} = d(\mathfrak{t}) \cdot \mathcal{T}_{\mathfrak{t}^\lambda}$, that is, $d(\mathfrak{t})^{-1} \cdot \mathcal{T}_\mathfrak{t} = \mathcal{T}_{\mathfrak{t}^\lambda}$. By the equation of the coset $C_\mathfrak{t}$ given at \eqref{equation:equation_coset} and Lemma~\ref{lemma:relation_equivalence_tableau}, we get that $d(\mathfrak{t})^{-1} \in C_\mathfrak{t}$. Applying \cite[Proposition 3.3]{Ma}, we know that $d(\mathfrak{t})^{-1}$ is the unique minimal length element of $C_\mathfrak{t}$. As a consequence, we have $d(\mathfrak{t})^{-1} = \pi_{\mathfrak{t}}$ and thus:
\begin{equation}
\label{equation:young_subgroup-piT=T}
\pi_\mathfrak{t} \cdot \mathcal{T}_\mathfrak{t} = \mathcal{T}_{\mathfrak{t}^\lambda}.
\end{equation}
Let $j \in J$ and  $m \in \{1, \dots, \lambda_j\}$, and let $a$ (respectively $\alpha$) be the label of the box $(j, m)$ in $\mathcal{T}_\mathfrak{t}$ (resp. $\mathcal{T}_{\mathfrak{t}^\lambda}$). In particular, by \eqref{equation:young_subgroups-label_a} we have $\alpha = \boldsymbol{\lambda}_{j - 1} + m$.
Moreover, by \eqref{equation:young_subgroup-piT=T} we have $\pi_\mathfrak{t}(a) = \alpha$: we conclude that the announced formula is satisfied, since, by a last use of \eqref{equation:young_subgroups-label_a}, we have $j = \mathfrak{t}_a$ and $m = \#\{b \leq a : \mathfrak{t}_b = \mathfrak{t}_a\}$.
\end{proof}

\begin{remarque}
\label{remark:young-mathas}
In \cite{Ma}, the author considers the elements of $\mathfrak{S}_n$ as acting on $\{1, \dots, n\}$ from the \emph{right}, by $iw := w(i)$ where $i \in \{1, \dots, n\}$ and $w \in \mathfrak{S}_n$ is a permutation. This is the right action of $\mathfrak{S}_n^{\text{op}}$: in such a setting, we read products of permutations from left to right.
\end{remarque}

\begin{lemme}
\label{lemma:pi_really_permutes}
Let $\mathfrak{t} \in J^\lambda$ and let $\pi_{\mathfrak{t}} = s_{a_1} \cdots s_{a_r}$ be a reduced expression. Then:
\[
\forall m \in \{1, \dots, r\},
s_{a_m} \cdot (w_m \cdot \mathfrak{t}) \neq w_m \cdot \mathfrak{t},
\]
where $w_m \coloneqq s_{a_{m+1}} \cdots s_{a_r}$ (with $w_m = 1$ if $m = r$).
\end{lemme}

\begin{proof}
Let us suppose  $s_{a_m} \cdot (w_m \cdot \mathfrak{t}) = w_m \cdot \mathfrak{t}$ and define $\widetilde{\pi}_{\mathfrak{t}} \coloneqq s_{a_1} \cdots s_{a_{m-1}} s_{a_{m+1}} \cdots s_{a_r}$. Using the assumption and the equality $\pi_\mathfrak{t} \cdot \mathfrak{t} = \mathfrak{t}^\lambda$, we see that the element $\widetilde{\pi}_{\mathfrak{t}}$ verifies $\widetilde{\pi}_{\mathfrak{t}} \cdot \mathfrak{t} = \mathfrak{t}^\lambda$ too. As the element $\widetilde{\pi}_{\mathfrak{t}}$  is strictly shorter that $\pi_{\mathfrak{t}}$ (since $s_{a_1} \cdots s_{a_r}$ is reduced), this is in contradiction with  Proposition~\ref{proposition:pi_sigma_lambda_sort}.
\end{proof}

\begin{remarque}
\label{remark:pi-1_really_permutes}
Using $\mathfrak{t} = \pi_{\mathfrak{t}}^{-1} \cdot \mathfrak{t}^\lambda$ in Lemma~\ref{lemma:pi_really_permutes}, we get the following similar result for $\pi_{\mathfrak{t}}^{-1}$. If $\pi_{\mathfrak{t}}^{-1} = s_{a_r} \cdots s_{a_1}$ is a reduced expression, then:
\[
\forall m \in \{1, \dots, r\}, w'_m \cdot \mathfrak{t}^\lambda \neq s_{a_m} \cdot (w'_m \cdot \mathfrak{t}^\lambda),
\]
where $w'_m \coloneqq s_{a_{m-1}} \cdots s_{a_1}$.
\end{remarque}

%

The next two lemmas are not essential to the proof of the main theorem of this section, Theorem~\ref{theorem:disjoint_guiver}; however, they will allow us to relate our construction to the one of \cite{JaPA, PA}.
Let $\mathfrak{t} \in J^n$ and $a \in \{1, \dots, n-1\}$. We give in the next lemma the decomposition of Remark~\ref{remark:decomposition_cosets} for the element $\pi_\mathfrak{t} s_a$; this is in fact a particular case of Deodhar's lemma (see, for instance, \cite[Lemma 2.1.2]{GePf}).

\begin{lemme}
\label{lemma:pisigma_sa_sigma}
Let $\mathfrak{t} \in J^n$ and $a \in \{1, \dots, n-1\}$. The element $\pi_\mathfrak{t} s_a$ belongs to the coset $C_{s_a \cdot \mathfrak{t}}$, more precisely we have:
\[
\pi_\mathfrak{t} s_a = \begin{cases}
s_{\pi_\mathfrak{t}(a)} \pi_\mathfrak{t} & \text{if } \mathfrak{t}_a = \mathfrak{t}_{a+1},
\\
\pi_{s_a \cdot \mathfrak{t}} & \text{if } \mathfrak{t}_a \neq \mathfrak{t}_{a+1}.
\end{cases}
\]
\end{lemme}

\begin{proof}
First, from \eqref{equation:equation_coset} we have $\pi_\mathfrak{t} s_a \cdot (s_a \cdot \mathfrak{t}) = \mathfrak{t}^\lambda$ (where $\lambda \comp_d n$ is the shape of $\mathfrak{t}$) thus $\pi_\mathfrak{t} s_a$ lies in the coset $C_{s_a \cdot \mathfrak{t}}$.

We suppose that $\mathfrak{t}_a = \mathfrak{t}_{a+1}$. We have $\pi_\mathfrak{t} s_a \pi_\mathfrak{t}^{-1} = (\pi_\mathfrak{t}(a), \pi_\mathfrak{t}(a+1))$, and we conclude since $\pi_\mathfrak{t}(a+1) = \pi_\mathfrak{t}(a) + 1$ by Proposition~\ref{proposition:explicit_pi}.

We now suppose that $\mathfrak{t}_a \neq \mathfrak{t}_{a+1}$. Using the same Proposition~\ref{proposition:explicit_pi}, we know that the permutation $w \coloneqq  \pi_\mathfrak{t}^{-1} \pi_{s_a \cdot \mathfrak{t}} \in \mathfrak{S}_n$ is supported in $\{a, a+1\}$; thus either $w = s_a$  or $w=\mathrm{id}$. Since $\mathfrak{t} \neq s_a \cdot \mathfrak{t}$ we have $\pi_\mathfrak{t} \neq \pi_{s_a \cdot \mathfrak{t}}$, hence $w \neq \mathrm{id}$. Hence, we get $w = s_a$, that is, $\pi_\mathfrak{t} s_a = \pi_{s_a \cdot \mathfrak{t}}$.
\end{proof}

We now generalise the result of Lemma~\ref{lemma:pisigma_sa_sigma} in the case $\mathfrak{t}_a = \mathfrak{t}_{a+1}$.

\begin{lemme}
\label{lemma:reduced_expression_pisigma_sa}
Let $\mathfrak{t} \in J^n$ and $a \in \{1, \dots, n-1\}$ with $\mathfrak{t}_a = \mathfrak{t}_{a+1}$. Let $s_{b_1} \cdots s_{b_r}$ be a reduced expression of $\pi_\mathfrak{t}$ and set $w_m \coloneqq s_{b_{m+1}} \cdots s_{b_r}$. If $b \in \{1, \dots, n-1\}$ verifies $s_b = w_m s_a w_m^{-1}$ for some $m \in \{0, \dots, r\}$, then:
\[
\pi_\mathfrak{t} s_a = s_{b_1} \cdots s_{b_m} s_b s_{b_{m+1}} \cdots s_{b_r}
\]
is a reduced expression. Moreover, every reduced expression of $\pi_\mathfrak{t} s_a$ is as above.
\end{lemme}

\begin{proof}
We first make an observation. As $\mathfrak{t}_a = \mathfrak{t}_{a+1}$, we deduce from \eqref{equation:equation_coset} that the element $\pi_\mathfrak{t} s_a$ remains in $C_\mathfrak{t}$. Hence, by minimality of $\pi_\mathfrak{t}$ we have:
\begin{equation}
\label{equation:ell(pi_sigma sa)}
\ell(\pi_\mathfrak{t} s_a) > \ell(\pi_\mathfrak{t}).
\end{equation}
Let now $s_{b_1} \cdots s_{b_r}$ be a reduced expression of $\pi_\mathfrak{t}$ and let $b \in \{1, \dots, n-1\}$ and $m \in \{0, \dots r\}$ such that $s_b = w_m s_a w_m^{-1}$. We have:
\[
\pi_\mathfrak{t} s_a = s_{b_1} \cdots s_{b_m} w_m s_a = s_{b_1} \cdots s_{b_m} s_b w_m =  s_{b_1} \cdots s_{b_m} s_b s_{b_{m+1}} \cdots s_{b_r},
\]
and this expression is reduced since $\ell(\pi_\mathfrak{t} s_a) = \ell(\pi_\mathfrak{t}) + 1$.
Conversely, let $s_{b'_0} \cdots s_{b'_r}$ be  a reduced expression of $\pi_\mathfrak{t} s_a$. Since $\ell((\pi_\mathfrak{t} s_a) s_a) < \ell(\pi_\mathfrak{t} s_a)$, we can apply \cite[\textsection 5.8 Theorem]{Hum}: we know that there is some $m \in \{0, \dots r\}$ such that $s_{b'_0} \cdots \hat{s}_{b'_m} \cdots s_{b'_r}$ is a reduced expression of $\pi_\mathfrak{t}$, where the hat denotes the omission. We have:
\[
s_{b'_0} \cdots  s_{b'_r} = s_{b'_0} \cdots \hat{s}_{b'_m} \cdots s_{b'_r} s_a,
\]
thus:
\[
s_{b'_m} = w_m s_a w_m^{-1},
\]
where $w_m \coloneqq s_{b'_{m+1}} \cdots s_{b'_r}$. We now set $b \coloneqq b'_m$ and:
\[
b_p \coloneqq \begin{cases}
b'_{p-1} & \text{if } p \in \{1, \dots, m\},
\\
b'_p & \text{if } p \in \{m+1, \dots r\}.
\end{cases}
\]
Moreover:
\begin{itemize}
\item the expression $s_{b_1} \cdots s_{b_r} = s_{b'_0} \cdots \hat{s}_{b'_m} \cdots s_{b'_r} = \pi_\mathfrak{t}$ is reduced;
\item we have $w_m = s_{b'_{m+1}} \cdots s_{b'_r} = s_{b_{m+1}} \cdots s_{b_r}$;
\item we have $s_b = s_{b'_m} = w_m s_a w_m^{-1}$;
\end{itemize}
thus the reduced expression $\pi_\mathfrak{t} s_a = s_{b'_0} \cdots s_{b'_r} = s_{b_1} \cdots s_{b_m} s_b s_{b_{m+1}} \cdots s_{b_r}$ is of the desired form.
\end{proof}

\begin{remarque}
Let $s_{b_1} \cdots s_{b_r} = \pi_\mathfrak{t}$ be a reduced expression and set $w_m \coloneqq s_{b_{m+1}} \cdots s_{b_r}$. For $m \in \{0, \dots, r\}$, there exists $b' \in \{1, \dots, n-1\}$ such that $s_{b'} = w_m s_a w_m^{-1}$ if and only if $w_m(a+1) = w_m(a) \pm 1$. Moreover, as
Lemma~\ref{lemma:pi_really_permutes} ensures that $w_m(a+1) > w_m(a)$, we have $w_m(a+1) = w_m(a) \pm 1 \iff w_m(a+1) = w_m(a) + 1$.
\end{remarque}

We end this subsection by introducing a notation. If $\mathfrak{t} \in J^\lambda$ and $\tuple{k} \in \K^{\mathfrak{t}^\lambda}$, we define:
\begin{equation}
\label{equation:definition_isigma}
\tuple{k}^\mathfrak{t} \coloneqq \pi_\mathfrak{t}^{-1} \cdot \tuple{k} \in \K^\mathfrak{t};
\end{equation}
in particular, we may denote by $\tuple{k}^\mathfrak{t}$ the elements of $\K^\mathfrak{t}$.

\subsubsection{A ``disjoint quiver'' Hecke algebra}

We consider the setting of \textsection\ref{subsubsection:general_definition_quiver}.  Note that since $\K$ is finite, we can consider the quiver Hecke algebra $\H_n(Q)$; recall that its generators are given at \eqref{equation:quiver_generators_n}, which are subject to the relations \eqref{relation:quiver_e(i)e(i')}--\eqref{relation:quiverQ_tresse} and \eqref{relation:quiver_sum_n_e(i)}. We already said that the elements $\psi_a$ do not verify the same braid relations as the elements $s_a \in \mathfrak{S}_n$: in particular, if $s_{b_1} \cdots s_{b_r}$ is another reduced expression for $w$, we may have $\psi_{b_1} \cdots \psi_{b_r} \neq \psi_w$. 
However, according to Remark~\ref{remark:reduced_expression_Slambda} we can assume that we chose the reduced expressions such that:
\begin{equation}
\label{equation:psiw_w_in_Slambda}
\forall w = (w_1, \dots, w_d) \in \mathfrak{S}_\lambda, \psi_w = \psi_{w_1} \cdots \psi_{w_d},
\end{equation}
and in the sequel we do suppose that we did so.
To that extent, we can first choose some reduced expressions for the elements of the subgroups $\mathfrak{S}_{\lambda_j}$ for $j \in J$ and then by product we obtain the reduced expressions of the element of $\mathfrak{S}_\lambda$. Concerning the elements of $\mathfrak{S}_n \setminus \mathfrak{S}_\lambda$, we can choose their reduced expressions arbitrary.

\bigskip
We now suppose that the matrix $Q$ verifies:
\begin{equation}
\label{equation:condition_Q}
\forall j \neq j', \forall (k, k') \in \K_j \times \K_{j'}, Q_{k, k'} = 1.
\end{equation}
When the matrix $Q$ is associated with a quiver $\Gamma$ (recall \textsection\ref{subsubsection:case_quivers}), the condition \eqref{equation:condition_Q} is satisfied when $\Gamma$ is the disjoint union of $d$ proper subquivers $\Gamma^1, \dots, \Gamma^d$. It means that:
\begin{itemize}
\item if $v$ is a vertex in $\Gamma$ then there is a unique $1 \leq j \leq d$ such that $v$ is a vertex of $\Gamma^j$;
\item if $(v, w)$ is an edge in $\Gamma$ then there is a (unique) $1 \leq j \leq d$ such that:
\begin{itemize}
\item the vertices $v$ and $w$ are vertices of $\Gamma^j$,
\item the edge $(v, w)$ is an edge of $\Gamma^j$.
\end{itemize}
\end{itemize}
Such a disjoint union in $d$ proper subquivers was encountered at \eqref{equation:gamma=union_gammae}. 
Moreover, regarding the Cartan matrix of $\Gamma$ we have:
\begin{equation}
\label{equation:cartan_matrix_block_diagonal}
\forall j \neq j', \forall (k, k') \in \K_j \times \K_{j'}, c_{k, k'} = 0;
\end{equation}
that is, up to a permutation of the indexing set, the matrix is block diagonal. Finally, for $j \in J$ we define:
\begin{equation}
\label{equation:definition_Qj}
\forall k, k' \in \K_j, Q^j_{k, k'} \coloneqq Q_{k, k'},
\end{equation}
in particular for each $j \in J$ and for each $n' \in \mathbb{N}$ we have an associated quiver Hecke algebra $\H_{n'}(Q^j)$.

\subsubsection{Useful idempotents}

We define in this section some idempotents of $\H_n(Q)$ which are essential for our proof.
Thanks to the defining relations \eqref{relation:quiver_e(i)e(i')}--\eqref{relation:quiver_psiae(i)} and \eqref{relation:quiver_sum_n_e(i)}, for each $\lambda \comp_d n$ the following element:
\begin{equation}
\label{equation:definition_e(lambda)}
e(\lambda) \coloneqq \sum_{\mathfrak{t} \in J^\lambda} \sum_{\tuple{k} \in \K^\mathfrak{t}} e(\tuple{k}),
\end{equation}
is a central idempotent in $\H_n(Q)$, that is, $e(\lambda) = e(\lambda)^2$ commutes with every element of $\H_n(Q)$. Moreover:
\begin{itemize}
\item if $\lambda' \comp_d n$ is different from $\lambda$ then $e(\lambda) e(\lambda') = 0$;
\item we have $\sum_{\lambda \comp_d n} e(\lambda) = 1$;
\end{itemize}
hence we have the following decomposition into subalgebras:
\begin{equation}
\label{equation:decomposition_Hn(Q)_e(lambda)}
\H_n(Q) = \bigoplus_{\lambda \comp_d n} e(\lambda)\H_n(Q).
\end{equation}

For $\mathfrak{t} \in J^\lambda$, we also define the following idempotent:
\[
e(\mathfrak{t}) := \sum_{\tuple{k} \in \K^\mathfrak{t}} e(\tuple{k}).
\]
We can note that $e(\lambda) = \sum_{\mathfrak{t} \in J^\lambda} e(\mathfrak{t})$. 
Moreover, we have $e(\mathfrak{t}) e(\mathfrak{t}') = 0$ if $\mathfrak{t}' \in J^n \setminus \{\mathfrak{t}\}$. We now give some lemmas which involve these elements $e(\mathfrak{t})$.

%

\begin{lemme}
\label{lemma:equations_psi_and_e(sigma)}
Let $\mathfrak{t} \in J^n$. We have the following relations:
\begin{align*}
\psi_a y_{a+1} e(\mathfrak{t}) &= y_a \psi_a e(\mathfrak{t})  & \text{if } \mathfrak{t}_a \neq \mathfrak{t}_{a+1},
\\
\psi_a y_a e(\mathfrak{t}) &= y_{a+1} \psi_a e(\mathfrak{t})  & \text{if } \mathfrak{t}_a \neq \mathfrak{t}_{a+1},
\\
 \psi_a^2 e(\mathfrak{t}) &= e(\mathfrak{t}) &\text{if } \mathfrak{t}_a \neq \mathfrak{t}_{a+1},
\\
 \psi_{a+1} \psi_a \psi_{a+1} e(\mathfrak{t}) &= \psi_a \psi_{a+1} \psi_a e(\mathfrak{t}) &\text{if } \mathfrak{t}_a \neq \mathfrak{t}_{a+2}.
\end{align*}
\end{lemme}

\begin{proof}
Let us first prove the first one. Let $\tuple{k} \in \K^\mathfrak{t}$; we have $k_a \in \K_{\mathfrak{t}_a}$ and $k_{a+1} \in \K_{\mathfrak{t}_{a+1}}$ with $\mathfrak{t}_a \neq \mathfrak{t}_{a+1}$ thus $k_a \neq k_{a+1}$. Hence, we get the result using the defining relation \eqref{relation:quiver_psia_ya+1} by summing over all $\tuple{k} \in \K^\mathfrak{t}$. The proofs of the second and the last equalities are similar.

Let us now prove $\psi_a^2 e(\mathfrak{t}) = e(\mathfrak{t})$ if $\mathfrak{t}_a \neq \mathfrak{t}_{a+1}$. Let $\tuple{k} \in \K^\mathfrak{t}$; we have $k_a \in \K_{\mathfrak{t}_a}$ and $k_{a+1} \in \K_{\mathfrak{t}_{a+1}}$ with $\mathfrak{t}_a \neq \mathfrak{t}_{a+1}$ thus $Q_{k_a, k_{a+1}} = 1$ (see \eqref{equation:condition_Q}). Hence, the defining relation \eqref{relation:quiverQ_psia^2} gives $\psi_a^2 e(\tuple{k}) = e(\tuple{k})$, and we again conclude by summing over all $\tuple{k} \in \K^\mathfrak{t}$.
\end{proof}

%

\subsection{About the \texorpdfstring{$\psi_{\pi_\mathfrak{t}}$}{psipisigma}}
\label{subsection:about_psit}

Here we prove some identities which are satisfied by the elements we have just introduced; some of them will be essential to the proof of Theorem~\ref{theorem:disjoint_guiver}, while others will only be used in \textsection\ref{subsubsection:an_application}, namely with Lemmas~\ref{lemma:ya_psipisigma} to \ref{lemma:independence_psi_pi_sa_sigma}.
We first study some properties about the elements $\psi_{\pi_\mathfrak{t}}$ for $\mathfrak{t} \in J^n$. We begin by the most important one, which is mentioned in the proof of \cite[Lemma 3.17]{SVV}.

\begin{lemme}
\label{lemma:independence_psi_pi_sigma}
Let $\mathfrak{t} \in J^n$. If $s_{a_1} \cdots s_{a_r}$ and $s_{b_1} \cdots s_{b_r}$ are two reduced expressions of $\pi_{\mathfrak{t}}$, then:
\[
\psi_{a_1} \cdots \psi_{a_r} e(\mathfrak{t}) = \psi_{b_1} \cdots \psi_{b_r} e(\mathfrak{t}).
\]
In other words the element $\psi_{\pi_{\mathfrak{t}}} e(\mathfrak{t}) \in \H_n(Q)$ does not depend on the choice of a reduced expression for $\pi_{\mathfrak{t}}$.
\end{lemme}

\begin{proof}
By Matsumoto's theorem, it suffices to check that every braid relation in $s_{a_1} \cdots s_{a_r}$ also occurs in $\psi_{a_1} \cdots \psi_{a_r} e(\mathfrak{t})$. By \eqref{relation:quiver_psia_psib}, it is true for length $2$-braids so it remains to check the case of the braids of length $3$.

Suppose that we have a braid of length $3$ in $s_{a_1} \cdots s_{a_r}$, at rank $m$: we have $a_m = a_{m+2} = a_{m+1} \pm 1$. We set $a \coloneqq \min(a_m, a_{m+1})$. With $w_{\mathrm{l}} \coloneqq s_{a_1} \cdots s_{a_{m-1}}$ and $w_{\mathrm{r}} \coloneqq s_{a_{m + 3}} \cdots s_{a_r}$, we have:
\[
w_{\mathrm{l}} (s_a s_{a + 1} s_a) w_{\mathrm{r}}  = w_{\mathrm{l}} (s_{a + 1} s_a s_{a + 1}) w_{\mathrm{r}},
\]
and we have to prove, with $\psi_{\mathrm{l}} \coloneqq \psi_{a_1} \cdots \psi_{a_{m-1}}$ and $\psi_{\mathrm{r}} \coloneqq \psi_{a_{m + 3}} \cdots \psi_{a_r}$:
\[
\psi_{\mathrm{l}}(\psi_a \psi_{a + 1} \psi_a)\psi_{\mathrm{r}}  e(\mathfrak{t})= \psi_{\mathrm{l}} (\psi_{a + 1} \psi_a \psi_{a + 1} )\psi_{\mathrm{r}}e(\mathfrak{t}).
\]
Using \eqref{relation:quiver_psiae(i)}, this becomes, where $\mathfrak{s} \coloneqq w_{\mathrm{r}} \cdot \mathfrak{t}$:
\begin{equation}
\label{equation:proof_psi_pi_independent}
\psi_{\mathrm{l}}(\psi_a \psi_{a + 1} \psi_a)  e(\mathfrak{s}) \psi_{\mathrm{r}}= \psi_{\mathrm{l}} (\psi_{a + 1} \psi_a \psi_{a + 1} )e(\mathfrak{s}) \psi_{\mathrm{r}}.
\end{equation}
By Lemma~\ref{lemma:pi_really_permutes}, we have $s_{a_{m+1}}\cdot (s_{a_{m+2}} \cdot \mathfrak{s}) \neq s_{a_{m+2}} \cdot \mathfrak{s}$. Thus, we have either $s_a\cdot(s_{a+1} \cdot \mathfrak{s}) \neq s_{a+1} \cdot \mathfrak{s}$ or $s_{a+1} \cdot (s_a \cdot \mathfrak{s}) \neq s_a \cdot \mathfrak{s}$; both cases give $\mathfrak{s}_a \neq \mathfrak{s}_{a+2}$. Hence, applying Lemma~\ref{lemma:equations_psi_and_e(sigma)} we know that \eqref{equation:proof_psi_pi_independent} holds.
\end{proof}

\begin{remarque}
\label{remark:independence_psi_pi_sigma_e(i)}
In particular, if $\tuple{k} \in \K^\mathfrak{t}$ then $\psi_{\pi_{\mathfrak{t}}} e(\tuple{k}) \in \H_n(Q)$ does not depend on the choice of a reduced expression for $\pi_{\mathfrak{t}}$ (note that $\psi_{\pi_{\mathfrak{t}}} e(\tuple{k}) = \psi_{\pi_{\mathfrak{t}}} e(\mathfrak{t}) e(\tuple{k})$).
\end{remarque}

Similarly to Lemma~\ref{lemma:independence_psi_pi_sigma}, using Remark~\ref{remark:pi-1_really_permutes} we prove that for $\mathfrak{t} \in J^\lambda$ the element:
\begin{equation}
\label{equation:psi_sigma-1_bien_def}
e(\mathfrak{t}) \psi_{\pi_\mathfrak{t}^{-1}} = \psi_{\pi_{\mathfrak{t}}^{-1}} e(\mathfrak{t}^\lambda) \in \H_n(Q),
\end{equation}
does not depend on the chosen reduced expression for $\pi_{\mathfrak{t}}^{-1}$. 
We now give some analogues of the results of Lemma~\ref{lemma:equations_psi_and_e(sigma)}.

\begin{proposition}
\label{proposition:inversion_psi_e(sigma)}
Let $\mathfrak{t} \in J^\lambda$. We have:
\begin{align*}
\psi_{\pi_{\mathfrak{t}}^{-1}} \psi_{\pi_{\mathfrak{t}}} e(\mathfrak{t})
&=
e(\mathfrak{t}),
\\
\psi_{\pi_{\mathfrak{t}}} \psi_{\pi_{\mathfrak{t}}^{-1}} e(\mathfrak{t}^\lambda)
&=
e(\mathfrak{t}^\lambda).
\end{align*}
\end{proposition}

\begin{remarque}
Both factors $\psi_{\pi_\mathfrak{t}^{-1}}$ and $\psi_{\pi_\mathfrak{t}}$ do not depend on the choices of reduced expressions: for instance, using \eqref{relation:quiver_psiae(i)} we have $\psi_{\pi_\mathfrak{t}^{-1}} \psi_{\pi_\mathfrak{t}} e(\mathfrak{t}) = \psi_{\pi_\mathfrak{t}^{-1}} e(\mathfrak{t}^\lambda) \psi_{\pi_\mathfrak{t}}$ thus we can apply Lemma~\ref{lemma:independence_psi_pi_sigma} and \eqref{equation:psi_sigma-1_bien_def}.
\end{remarque}

\begin{proof}
We only prove the first equality, the proof of the second one being entirely similar.
Let $s_{a_1} \cdots s_{a_r}$ be a reduced expression for $\pi_{\mathfrak{t}}$. We prove by induction that for every $m \in \{1, \dots, r+1\}$ we have:
\begin{equation}
\label{equation:recurrence_psi_psi_e(sigma)}
\psi_{\pi_{\mathfrak{t}}^{-1}} \psi_{\pi_{\mathfrak{t}}} e(\mathfrak{t})
= \psi_{a_r} \cdots \psi_{a_m} \psi_{a_m} \cdots \psi_{a_r} e(\mathfrak{t}).
\end{equation}
First, the case  $a = 1$ comes with the definition of $\psi_{\pi_{\mathfrak{t}}^{-1}} e(\mathfrak{t}^\lambda)$ and $\psi_{\pi_{\mathfrak{t}}}e(\mathfrak{t})$. Now, if \eqref{equation:recurrence_psi_psi_e(sigma)} is true for some $m \in \{1, \dots, r\}$ we have, using \eqref{relation:quiver_psiae(i)}:
\begin{equation}
\label{equation:proof_psi_inv_1}
\psi_{\pi_{\mathfrak{t}}^{-1}} \psi_{\pi_{\mathfrak{t}}} e(\mathfrak{t})
= \psi_{a_r} \cdots \psi_{a_{m+1}} \psi_{a_m}^2 e(w_m \cdot \mathfrak{t}) \psi_{a_{m+1}} \cdots \psi_{a_r},
\end{equation}
where $w_m \coloneqq s_{a_{m+1}} \cdots s_{a_r}$. By Lemma~\ref{lemma:pi_really_permutes}, we know that $(w_m \cdot \mathfrak{t})_{a_m} \neq (w_m \cdot \mathfrak{t})_{a_m + 1}$. Hence, by Lemma~\ref{lemma:equations_psi_and_e(sigma)} we have $\psi_{a_m}^2 e(w_m \cdot \mathfrak{t}) = e(w_m \cdot \mathfrak{t})$ thus \eqref{equation:proof_psi_inv_1} becomes:
\[
\psi_{\pi_{\mathfrak{t}}^{-1}} \psi_{\pi_{\mathfrak{t}}} e(\mathfrak{t})
= \psi_{a_r} \cdots \psi_{a_{m+1}} e(w_m \cdot \mathfrak{t}) \psi_{a_{m+1}} \cdots \psi_{a_r},
\]
which becomes, with a last use of \eqref{relation:quiver_psiae(i)}:
\[
\psi_{\pi_{\mathfrak{t}}^{-1}} \psi_{\pi_{\mathfrak{t}}} e(\mathfrak{t})
= \psi_{a_r} \cdots \psi_{a_{m+1}} \psi_{a_{m+1}} \cdots \psi_{a_r} e(\mathfrak{t}).
\]
Thus \eqref{equation:recurrence_psi_psi_e(sigma)} holds for every $m \in \{1, \dots, r+1\}$, in particular for $m = r + 1$ we get the statement of the Proposition.
\end{proof}

Once again, what follows is not essential to the proof of the main result Theorem~\ref{theorem:disjoint_guiver}; however, it will allow us to relate our construction to the one of \cite{JaPA, PA}.
With a similar proof as Proposition~\ref{proposition:inversion_psi_e(sigma)}, we obtain the following lemma.

\begin{lemme}
\label{lemma:ya_psipisigma}
Let $a \in \{1, \dots, n\}$ and $\mathfrak{t} \in J^\lambda$. We have:
\begin{gather*}
y_a \psi_{\pi_\mathfrak{t}} e(\mathfrak{t}) = \psi_{\pi_\mathfrak{t}} y_{\pi_\mathfrak{t}^{-1}(a)} e(\mathfrak{t}),
\\
y_a \psi_{\pi_\mathfrak{t}^{-1}} e(\mathfrak{t}^\lambda) = \psi_{\pi_\mathfrak{t}^{-1}} y_{\pi_\mathfrak{t}(a)} e(\mathfrak{t}^\lambda).
\end{gather*}
\end{lemme}

%

We now want to see what is happening with Lemma~\ref{lemma:pisigma_sa_sigma} for the associated elements $\psi_w$.

\begin{lemme}
\label{lemma:psi_pi-1_e(t)_psi_pisa}
Let $\mathfrak{t} \in J^n$ and $a \in \{1, \dots, n-1\}$ such that $\mathfrak{t}_a \neq \mathfrak{t}_{a+1}$. We have:
\[
 e(\mathfrak{t}) \psi_{\pi_\mathfrak{t}^{-1}}  \psi_{\pi_{s_a \cdot \mathfrak{t}}} = e(\mathfrak{t})\psi_a.
\]
\end{lemme}

\begin{proof}
By Lemma~\ref{lemma:pisigma_sa_sigma} we have $\ell(\pi_{s_a \cdot \mathfrak{t}}) = \ell(\pi_\mathfrak{t}) \pm 1$. We now simply distinguish cases.
\begin{itemize}
\item
We first assume that $\ell(\pi_{s_a \cdot \mathfrak{t}}) = \ell(\pi_\mathfrak{t}) + 1$. Hence, applying Lemma~\ref{lemma:independence_psi_pi_sigma} for the elements $\psi_{\pi_{s_a \cdot \mathfrak{t}}}$ and $\psi_{\pi_\mathfrak{t}}$ we have $e(\mathfrak{t}^\lambda) \psi_{\pi_{s_a \cdot \mathfrak{t}}} = e(\mathfrak{t}^\lambda) \psi_{\pi_\mathfrak{t}} \psi_a$. Finally, using Proposition~\ref{proposition:inversion_psi_e(sigma)} we have:
\[
e(\mathfrak{t})
\psi_{\pi_\mathfrak{t}^{-1}}  \psi_{\pi_{s_a \cdot \mathfrak{t}}}
=
\psi_{\pi_\mathfrak{t}^{-1}} e(\mathfrak{t}^\lambda) \psi_{\pi_\mathfrak{t}} \psi_a
=
\psi_{\pi_\mathfrak{t}^{-1}} \psi_{\pi_\mathfrak{t}} e(\mathfrak{t}) \psi_a
=
e(\mathfrak{t}) \psi_a.
\]
\item
We now suppose that $\ell(\pi_{s_a \cdot \mathfrak{t}}) = \ell(\pi_\mathfrak{t}) - 1$; hence, $\ell(\pi_\mathfrak{t}^{-1}) = \ell(\pi_{s_a \cdot \mathfrak{t}}^{-1}) + 1$. Recall that $\pi_\mathfrak{t}^{-1} = s_a \pi_{s_a \cdot \mathfrak{t}}^{-1}$; using the extension \eqref{equation:psi_sigma-1_bien_def} of Lemma~\ref{lemma:independence_psi_pi_sigma}, we get $e(\mathfrak{t})\psi_{\pi_\mathfrak{t}^{-1}} = e(\mathfrak{t})\psi_a \psi_{\pi_{s_a \cdot \mathfrak{t}}^{-1}}$ and finally:
\[
e(\mathfrak{t}) \psi_{\pi_\mathfrak{t}^{-1}}  \psi_{\pi_{s_a \cdot \mathfrak{t}}}
=
\psi_a e(s_a \cdot \mathfrak{t}) \psi_{\pi_{s_a \cdot \mathfrak{t}}^{-1}} \psi_{\pi_{s_a \cdot \mathfrak{t}}} 
=
\psi_a e(s_a \cdot \mathfrak{t})
=
e(\mathfrak{t}) \psi_a. 
\]
\end{itemize}
\end{proof}

\begin{lemme}
\label{lemma:independence_psi_pi_sa_sigma}
Let $\mathfrak{t} \in J^n$ and $a \in \{1, \dots, n-1\}$ such that $\mathfrak{t}_a = \mathfrak{t}_{a+1}$. The element $\psi_{\pi_\mathfrak{t} s_a} e(\mathfrak{t}) \in \H_n(Q)$ does not depend on the choice of a reduced expression for $\pi_\mathfrak{t} s_a$. In particular:
\[
\psi_{\pi_\mathfrak{t}} \psi_a e(\mathfrak{t}) = \psi_{\pi_\mathfrak{t}(a)} \psi_{\pi_\mathfrak{t}} e(\mathfrak{t}).\]
\end{lemme}

\begin{proof}
As in the proof of Lemma~\ref{lemma:independence_psi_pi_sigma}, it suffices to prove that every 3-braid relation which occurs in a reduced expression of $\pi_\mathfrak{t} s_a$ is also verified in the corresponding element of $\H_n(Q) e(\mathfrak{t})$. By Lemma~\ref{lemma:reduced_expression_pisigma_sa}, we know that any reduced expression of $\pi_\mathfrak{t} s_a$ can be written $s_{b_0} \cdots s_{b_r}$, where there is $m \in \{0, \dots r\}$ such that:
\begin{itemize}
\item the word $s_{b_0} \cdots \hat{s}_{b_m} \cdots s_{b_r}$ is a reduced expression of $\pi_\mathfrak{t}$;
\item we have $s_{b_m} = w_m s_a w_m^{-1}$ with $w_m \coloneqq s_{b_{m+1}} \cdots s_{b_r}$. 
\end{itemize}
We suppose that a 3-braid appears in $s_{b_0} \cdots s_{b_r}$ at index $l$, that is, we have $b_l = b_{l + 2} = b_{l+1} \pm 1$; we set $b \coloneqq \min(b_l, b_{l + 1})$. We want to prove that, with $\psi_\mathrm{l} \coloneqq \psi_{b_0} \cdots \psi_{b_{l - 1}}$ and $\psi_\mathrm{r} \coloneqq \psi_{b_{l + 3}} \cdots \psi_{b_r}$:
\[
\psi_{\mathrm{l}}(\psi_b \psi_{b + 1} \psi_b)\psi_{\mathrm{r}}  e(\mathfrak{t})= \psi_{\mathrm{l}} (\psi_{b + 1} \psi_b \psi_{b + 1} )\psi_{\mathrm{r}}e(\mathfrak{t}).
\]
To that extent, as in the proof of Lemma~\ref{lemma:independence_psi_pi_sigma}, thanks to Lemma~\ref{lemma:equations_psi_and_e(sigma)} it suffices to prove that $\mathfrak{s}_b \neq \mathfrak{s}_{b + 2}$ where $\mathfrak{s} \coloneqq s_{b_{l + 3}} \cdots s_{b_r} \cdot \mathfrak{t}$.
\begin{itemize}
\item If $m < l + 1$ then applying  Lemma~\ref{lemma:pi_really_permutes} we have $s_{b_{l+1}} \cdot(s_{b_{l + 2}} \cdot \mathfrak{s}) \neq s_{b_{l + 2}} \cdot \mathfrak{s}$  thus $\mathfrak{s}_b \neq \mathfrak{s}_{b + 2}$.
\item The case $m = l + 1$ is impossible: as $b_l = b_{l + 2}$, if $m = l + 1$ we would have $b_{m-1} = b_{m+1}$ and this is nonsense since the expression $s_{b_{m-1}} s_{b_{m+1}}$ is reduced (as a subexpression of the reduced expression $s_{b_0} \cdots \hat{s}_{b_m} \cdots s_{b_r} = \pi_\mathfrak{t}$).
\item If $m = l + 2$ then by Lemma~\ref{lemma:pi_really_permutes} we get $s_{b_l} \cdot (s_{b_{l + 1}} \cdot \mathfrak{s}) \neq s_{b_l} \cdot \mathfrak{s}$ thus $\mathfrak{s}_b \neq \mathfrak{s}_{b+2}$.
\item Finally, if $m > l + 2$ then we can notice that:
\[
\mathfrak{s} = s_{b_{l + 3}} \cdots \hat{s}_{b_m} \cdots s_{b_r} \cdot \mathfrak{t}
\]
(since $s_{b_m} \cdot( w_m \cdot \mathfrak{t}) = w_m \cdot (s_a \cdot \mathfrak{t}) = w_m \cdot \mathfrak{t}$; recall that $s_{b_m} = w_m s_a w_m^{-1}$ and $\mathfrak{t}_a = \mathfrak{t}_{a+1}$). Hence, we deduce once again the result from Lemma~\ref{lemma:pi_really_permutes}.
\end{itemize}

The last statement of the lemma is now immediate. As $\mathfrak{t}_a = \mathfrak{t}_{a+1}$, we can use \eqref{equation:ell(pi_sigma sa)} hence $\psi_{\pi_\mathfrak{t} s_a} e(\mathfrak{t}) = \psi_{\pi_\mathfrak{t}} \psi_a e(\mathfrak{t})$. Moreover, applying Lemma~\ref{lemma:pisigma_sa_sigma} another consequence of \eqref{equation:ell(pi_sigma sa)} is $\ell(s_{\pi_\mathfrak{t}(a)} \pi_\mathfrak{t}) = \ell(\pi_\mathfrak{t}) + 1$, thus we get $\psi_{\pi_\mathfrak{t} s_a} e(\mathfrak{t}) = \psi_{\pi_\mathfrak{t}(a)} \psi_{\pi_\mathfrak{t}} e(\mathfrak{t})$. Finally, we have $\psi_{\pi_\mathfrak{t} s_a} e(\mathfrak{t}) = \psi_{\pi_\mathfrak{t}} \psi_a e(\mathfrak{t}) = \psi_{\pi_\mathfrak{t}(a)} \psi_{\pi_\mathfrak{t}} e(\mathfrak{t})$.
\end{proof}

\subsection{Decomposition along the subquiver Hecke algebras}

We are now ready to prove the main result of this section: we will give in Theorem~\ref{theorem:disjoint_guiver} a decomposition of $\H_n(Q)$ involving the  algebras $\H_{n_j}(Q^j)$ for $j \in J$.

\subsubsection{A distinguished subalgebra}

In this paragraph, we prove the key of Theorem~\ref{theorem:disjoint_guiver}.
Recall that we have set in \eqref{equation:definition_Qj}:
\[
\forall k, k' \in \K_j, Q^j_{k, k'} \coloneqq Q_{k, k'},
\]
for any $j \in J$.
Let $\lambda \comp_d n$ be a $d$-composition of $n$. We define the following algebra:
\[
\H_\lambda(Q) \coloneqq \H_{\lambda_1}(Q^1) \otimes \cdots \otimes \H_{\lambda_d}(Q^d).
\]

With $e_\lambda \coloneqq e(\mathfrak{t}^\lambda)$, we prove here that we can identify $\H_\lambda(Q)$ with the subalgebra $e_\lambda \H_n(Q) e_\lambda$ (with unit $e_\lambda$).  
We reindex the generators $\psi_1, \dots, \psi_{\lambda_j - 1}$ and $y_1, \dots, y_{\lambda_j}$ of $\H_{\lambda_j}(Q^j)$, respectively by $\psi_{\boldsymbol{\lambda}_{j - 1} + 1}, \dots, \psi_{\boldsymbol{\lambda}_j - 1}$ and $y_{\boldsymbol{\lambda}_{j - 1} + 1},\dots, y_{\boldsymbol{\lambda}_j}$. 
In particular, we set:
\[
\psi^\otimes_w \coloneqq \psi_{w_1} \otimes \cdots \otimes \psi_{w_d} \in \H_\lambda(Q),
\]
for $w = (w_1, \dots, w_d) \in \mathfrak{S}_\lambda$ and:
\begin{equation}
\label{equation:e_otimes}
e^{\otimes}(\tuple{k}^1, \dots, \tuple{k}^d) \coloneqq e(\tuple{k}^1)\otimes \cdots \otimes e(\tuple{k}^d) \in \H_\lambda(Q),
\end{equation}
for $\tuple{k} = (\tuple{k}^1, \dots, \tuple{k}^d) \in \K_1^{\lambda_1} \times \cdots \times \K_d^{\lambda_d}$.

\begin{lemme}
\label{lemma:bases_e(sigmalambda)He(sigmalambda)}
The following family:
\begin{equation}
\label{equation:basis_e(sigmalambda)He(sigmalambda)}
\left\lbrace
\psi_w y_1^{r_1} \cdots y_n^{r_n} e(\tuple{k}), w \in \mathfrak{S}_\lambda, r_a \in \mathbb{N}, \tuple{k} \in \K^{\mathfrak{t}^\lambda}\right\rbrace,
\end{equation}
is an $A$-basis of $e_\lambda\H_n(Q)e_\lambda$. Moreover, the algebra $e_\lambda \H_n(Q) e_\lambda$ is exactly the (non-unitary) subalgebra of $\H_n(Q)$ generated by:
\begin{itemize}
\item the $\psi_a e_\lambda$ for $a \in \{1, \dots, n\} \setminus \{\boldsymbol{\lambda}_1, \dots, \boldsymbol{\lambda}_d\}$;
\item the $y_a e_\lambda$ for $1 \leq a \leq n$;
\item the $e(\tuple{k})$ for $\tuple{k} \in \K^{\mathfrak{t}^\lambda}$.
\end{itemize}
\end{lemme}

\begin{proof}
The first part is a immediate application of Theorem~\ref{theorem:base_quiver}, \eqref{relation:quiver_y_ae(i)}, \eqref{relation:quiver_psiae(i)} and Proposition~\ref{proposition:young_stabiliser_sigma_lambda}.
It remains to check that $e_\lambda \H_n(Q) e_\lambda$ is the described subalgebra, subalgebra that we temporarily denote by $\H$. First, all the listed elements belong to $e_\lambda \H_n(Q) e_\lambda$. In particular, for $a \neq \boldsymbol{\lambda}_j$ we have indeed $\psi_a e_\lambda \in e_\lambda \H_n(Q) e_\lambda$: we can either use the above basis, or we can simply use \eqref{relation:quiver_psiae(i)} to get $ \psi_a e_\lambda = \psi_a e_\lambda^2 = e_\lambda \psi_a e_\lambda$. Hence, we have $\H \subseteq e_\lambda \H_n(Q) e_\lambda$.  Finally, we conclude since every element of the basis \eqref{equation:basis_e(sigmalambda)He(sigmalambda)} lies in $\H$.
\end{proof}

\begin{lemme}
\label{lemme:morphisme_Hlambda_H}
There is a unitary algebra homomorphism from $\H_\lambda(Q)$ to $e_\lambda\H_n(Q) e_\lambda$.
\end{lemme}

\begin{proof}
We define the algebra homomorphism from $\H_\lambda(Q)$ to $e_\lambda \H_n(Q)e_\lambda$ by sending:
\begin{itemize}
\item the generators $\psi^\otimes_a \in \H_\lambda(Q)$ for $a \in \{1, \dots, n \} \setminus \{\boldsymbol{\lambda}_1, \dots, \boldsymbol{\lambda}_d\}$ to $\psi_a e_\lambda \in e_\lambda\H_n(Q)e_\lambda$;
\item the generators $y_b \in \H_\lambda(Q)$ for $b \in \{1, \dots, n\}$ to $y_b e_\lambda \in e_\lambda\H_n(Q)e_\lambda$;
\item the generators $e^{\otimes}(\tuple{k}) \in \H_\lambda(Q)$ for $\tuple{k} \in \K^{\mathfrak{t}^\lambda}$ to $e(\tuple{k}) \in e_\lambda\H_n(Q)e_\lambda$.
\end{itemize}

It suffices now to check the defining relations of $\H_\lambda(Q)$. We will only check \eqref{relation:quiver_psia_ya+1}--\eqref{relation:quiverQ_tresse}, the remaining ones being straightforward.

\begin{description}
\item[\eqref{relation:quiver_psia_ya+1}.] Let $a \notin \{\boldsymbol{\lambda}_1, \dots, \boldsymbol{\lambda}_d\}$ and $\tuple{k} \in \K^{\mathfrak{t}^\lambda}$. If $k_a \in \K_j$, as $a \neq \boldsymbol{\lambda}_{j'}$ for any $j'$ we have $k_{a+1} \in \K_j$ (cf. \eqref{equation:definition_sigmalambda}). Hence, in $\H_\lambda(Q)$ the relation \eqref{relation:quiver_psia_ya+1}:
\[
\psi^\otimes_a y_{a+1} e^{\otimes}(\tuple{k}) = \begin{cases}
(y_a \psi^\otimes_a + 1)e^{\otimes}(\tuple{k}) & \text{if } k_a = k_{a+1},
\\
y_a \psi^\otimes_a e^{\otimes}(\tuple{k}) & \text{if } k_a \neq k_{a+1},
\end{cases}
\]
comes from the corresponding relation in $\H_{\lambda_j}(Q^j)$. The same relation:
\[
\psi_a y_{a+1} e(\tuple{k}) = \begin{cases}
(y_a \psi_a + 1)e(\tuple{k}) & \text{if } k_a = k_{a+1},
\\
y_a \psi_a e(\tuple{k}) & \text{if } k_a \neq k_{a+1},
\end{cases}
\]
is verified in $e_\lambda \H_n(Q) e_\lambda$, as a relation in $\H_n(Q)$.

\item[\eqref{relation:quiver_ya+1_psia}.] Similar.

\item[\eqref{relation:quiverQ_psia^2}.] Similarly, the indices $k_a, k_{a+1}$ are in a same $\K_j$. Hence, the relation \eqref{relation:quiverQ_psia^2} in $\H_\lambda(Q)$ comes from a relation in $\H_{\lambda_j}(Q^j)$, and the same relation is verified in $e_\lambda \H_n(Q) e_\lambda$.

\item[\eqref{relation:quiverQ_tresse}.] For $j \in J$ and $a \in \{\boldsymbol{\lambda}_{j-1} + 1, \dots, \boldsymbol{\lambda}_j - 2\}$, the relation \eqref{relation:quiverQ_tresse} is a relation from $\H_{\lambda_j}(Q^j)$, and this same relation is verified in $e_\lambda \H_n(Q) e_\lambda$.
\end{description}
\end{proof}

\begin{proposition}
\label{proposition:Hlambda_ssalg}
The previous algebra homomorphism $\H_\lambda(Q) \to e_\lambda \H_n(Q) e_\lambda$ is an isomorphism. In particular, we can identify $\H_\lambda(Q)$ to a (non-unitary) subalgebra of $\H_n(Q)$.
\end{proposition}

\begin{proof}
We know by Theorem~\ref{theorem:base_quiver} that $\H_{\lambda_j}(Q^j)$ has for basis:
\[
\left\lbrace\psi_{w_j} y_{\boldsymbol{\lambda}_{j-1}+1}^{r_{\boldsymbol{\lambda}_{j-1}+1}} \cdots y_{\boldsymbol{\lambda}_j}^{r_{\boldsymbol{\lambda}_j}} e(\tuple{k}^j) : w_j \in \mathfrak{S}_{\lambda_j}, r_a \in \mathbb{N}, \tuple{k}^j \in \K_j^{\lambda_j}\right\rbrace,
\]
hence,  the following family:
\[
\left\lbrace\psi^\otimes_w y_1^{r_1} \cdots y_n^{r_n} e^{\otimes}(\tuple{k}) : w \in \mathfrak{S}_\lambda, r_a \in \mathbb{N}, \tuple{k} \in \K_1^{\lambda_1} \times \cdots \times \K_d^{\lambda_d}
\right\rbrace,
\]
is a basis of $\H_\lambda(Q)$. We conclude since by \eqref{equation:psiw_w_in_Slambda} the homomorphism of Lemma~\ref{lemme:morphisme_Hlambda_H} sends this basis onto the basis of $e_\lambda \H_n(Q) e_\lambda$ given in Lemma~\ref{lemma:bases_e(sigmalambda)He(sigmalambda)}.
 (In particular, note that $\psi^\otimes_w \in \H_\lambda(Q)$ for $w \in \mathfrak{S}_\lambda$ is sent to $\psi_w e_\lambda \in e_\lambda \H_n(Q) e_\lambda$.)
\end{proof}


\subsubsection{Decomposition theorem}
\label{subsubsection:theorem_disjoint_quiver}

We recall the notation $m_\lambda$ for $\lambda \comp_d n$ introduced at \eqref{equation:definition_mlambda}.

\begin{theoreme}
\label{theorem:disjoint_guiver}
We have an $A$-algebra isomorphism:
\[
\H_n(Q) \simeq \bigoplus_{\lambda \comp_d n} \Mat_{m_\lambda}\H_\lambda(Q).
\]
\end{theoreme}

The remaining part of this paragraph is devoted to the proof of Theorem~\ref{theorem:disjoint_guiver}. Due to \eqref{equation:decomposition_Hn(Q)_e(lambda)}, it suffices to prove that we have an $A$-algebra isomorphism:
\begin{equation}
\label{equation:reduction_theoreme_lambda}
e(\lambda)\H_n(Q) \simeq \Mat_{m_{\lambda}}\H_\lambda(Q).
\end{equation}
%
Let us label the rows and the columns of the elements of $\mathrm{Mat}_{m_\lambda} \H_\lambda(Q)$ by $(\mathfrak{t}', \mathfrak{t}) \in ({J^\lambda})^2$,  and let us write $E_{\mathfrak{t}', \mathfrak{t}}$ for the elementary matrix with one $1$ at position $(\mathfrak{t}', \mathfrak{t})$ and $0$ everywhere else. Recall the following property verified by the $E_{\mathfrak{t}', \mathfrak{t}}$:
\begin{equation}
\label{equation:Esigmasigma'_tautau'}
\forall \mathfrak{t}, \mathfrak{t}', \mathfrak{s}, \mathfrak{s}' \in J^\lambda, E_{\mathfrak{t}', \mathfrak{t}} E_{\mathfrak{s}', \mathfrak{s}} = \delta_{\mathfrak{t}, \mathfrak{s}'} E_{\mathfrak{t}', \mathfrak{s}}.
\end{equation}

We have the following $A$-module isomorphism, where $\mathfrak{t}, \mathfrak{t}' \in J^\lambda$:
\[
e(\mathfrak{t}') \H_n(Q) e(\mathfrak{t}) \simeq \H_\lambda(Q) E_{\mathfrak{t}', \mathfrak{t}}.
\]
Indeed, let us define:
\begin{equation}\label{equation:definition_philambda_psilambda}
\begin{gathered}
\Phi_{\mathfrak{t}', \mathfrak{t}} : \H_\lambda(Q) E_{\mathfrak{t}', \mathfrak{t}} \to e(\mathfrak{t}') \H_n(Q) e(\mathfrak{t}),
\\
\Psi_{\mathfrak{t}', \mathfrak{t}} : e(\mathfrak{t}') \H_n(Q) e(\mathfrak{t}) \to  \H_\lambda(Q) E_{\mathfrak{t}', \mathfrak{t}},
\end{gathered}
\end{equation}
by:
\begin{gather*}
\forall v \in \H_\lambda(Q), \Phi_{\mathfrak{t}', \mathfrak{t}}(v E_{\mathfrak{t}', \mathfrak{t}}) \coloneqq \psi_{\pi_{\mathfrak{t}'}^{-1}} v \psi_{\pi_\mathfrak{t}},
\\
\forall w \in e(\mathfrak{t}')\H_n(Q) e(\mathfrak{t}), \Psi_{\mathfrak{t}', \mathfrak{t}}(w) \coloneqq (\psi_{\pi_{\mathfrak{t}'}} w \psi_{\pi_\mathfrak{t}^{-1}}) E_{\mathfrak{t}', \mathfrak{t}}.
\end{gather*}

The goal sets of \eqref{equation:definition_philambda_psilambda} are respected, according to \eqref{relation:quiver_psiae(i)}, \eqref{equation:pisigmalambda_sort}  and Proposition~\ref{proposition:Hlambda_ssalg}. Indeed, for instance we have, for $v \in \H_\lambda(Q) \simeq e_\lambda \H_n(Q) e_\lambda$:
\begin{align*}
\Phi_{\mathfrak{t}', \mathfrak{t}}(v E_{\mathfrak{t}', \mathfrak{t}})
&= 
\psi_{\pi_{\mathfrak{t}'}^{-1}} v \psi_{\pi_\mathfrak{t}}
\\
&=
\psi_{\pi_{\mathfrak{t}'}^{-1}} e(\mathfrak{t}^\lambda) v  e(\mathfrak{t}^\lambda) \psi_{\pi_\mathfrak{t}}
\\
&=
e(\pi_{\mathfrak{t}'}^{-1} \cdot \mathfrak{t}^\lambda) \psi_{\pi_{\mathfrak{t}'}^{-1}} v \psi_{\pi_\mathfrak{t}} e(\pi_\mathfrak{t}^{-1} \cdot \mathfrak{t}^\lambda)
\\
\Phi_{\mathfrak{t}', \mathfrak{t}}(v E_{\mathfrak{t}', \mathfrak{t}})
&=
e(\mathfrak{t}') \psi_{\pi_{\mathfrak{t}'}^{-1}} v \psi_{\pi_\mathfrak{t}} e(\mathfrak{t})
\in
e(\mathfrak{t}')\H_n(Q)e(\mathfrak{t}).
\end{align*}

\begin{remarque}
Our map $\Phi_{\mathfrak{t}', \mathfrak{t}}$ is similar to \cite[$(20)$]{SVV}.
\end{remarque}

Furthermore, these two maps $\Phi_{\mathfrak{t}', \mathfrak{t}}$ and $\Psi_{\mathfrak{t}', \mathfrak{t}}$ are clearly $A$-linear and by Proposition~\ref{proposition:inversion_psi_e(sigma)} these are inverse isomorphisms.
We now set:
\begin{equation}
\begin{gathered}
\Phi_\lambda \coloneqq \bigoplus_{\mathfrak{t}, \mathfrak{t}' \in J^\lambda} \Phi_{\mathfrak{t}', \mathfrak{t}} : \mathrm{Mat}_{m_\lambda} \H_\lambda(Q) \to e(\lambda) \H_n(Q),
\\
\Psi_\lambda \coloneqq \bigoplus_{\mathfrak{t}, \mathfrak{t}' \in J^\lambda} \Psi_{\mathfrak{t}', \mathfrak{t}} : e(\lambda) \H_n(Q) \to \mathrm{Mat}_{m_\lambda} \H_\lambda(Q).
\end{gathered}
\end{equation}
From the properties of $\Phi_{\mathfrak{t}', \mathfrak{t}}$ and $\Psi_{\mathfrak{t}', \mathfrak{t}}$, the above maps are inverse $A$-module isomorphisms; it now suffices to check that $\Psi_\lambda$ is an $A$-algebra homomorphism. This property comes from the following one:
\begin{equation}
\label{equation:proof_Psilambda_morphism}
\Psi_{\mathfrak{t}', \mathfrak{t}}(w_{\mathfrak{t}', \mathfrak{t}}) \Psi_{\mathfrak{s}', \mathfrak{s}}(w_{\mathfrak{s}', \mathfrak{s}}) = \Psi_{\mathfrak{t}', \mathfrak{s}}(w_{\mathfrak{t}', \mathfrak{t}} w_{\mathfrak{s}', \mathfrak{s}}),
\end{equation}
where $\mathfrak{t}, \mathfrak{t}', \mathfrak{s}, \mathfrak{s}' \in J^\lambda$, $w_{\mathfrak{t}', \mathfrak{t}} \in  e(\mathfrak{t}')\H_n(Q) e(\mathfrak{t})$ and $w_{\mathfrak{s}', \mathfrak{s}} \in e(\mathfrak{s}') \H_n(Q) e(\mathfrak{s})$. The equality \eqref{equation:proof_Psilambda_morphism} is obviously satisfied when $\mathfrak{t} \neq \mathfrak{s}'$ since both sides are zero, thus we assume $\mathfrak{t} = \mathfrak{s}'$. We have, using Proposition~\ref{proposition:inversion_psi_e(sigma)} and noticing that $w_{\mathfrak{t}, \mathfrak{s}} = e(\mathfrak{t}) w_{\mathfrak{t}, \mathfrak{s}}$:
\begin{align*}
\Psi_{\mathfrak{t}', \mathfrak{t}}(w_{\mathfrak{t}', \mathfrak{t}}) \Psi_{\mathfrak{s}', \mathfrak{s}}(w_{\mathfrak{s}', \mathfrak{s}})
&=
\Psi_{\mathfrak{t}', \mathfrak{t}}(w_{\mathfrak{t}', \mathfrak{t}}) \Psi_{\mathfrak{t}, \mathfrak{s}}(w_{\mathfrak{t}, \mathfrak{s}})
\\
&=
(\psi_{\pi_{\mathfrak{t}'}} w_{\mathfrak{t}', \mathfrak{t}} \psi_{\pi_\mathfrak{t}^{-1}}) (\psi_{\pi_\mathfrak{t}} w_{\mathfrak{t}, \mathfrak{s}} \psi_{\pi_\mathfrak{s}^{-1}}) E_{\mathfrak{t}', \mathfrak{t}} E_{\mathfrak{t}, \mathfrak{s}}
\\
&=
\psi_{\pi_{\mathfrak{t}'}} w_{\mathfrak{t}', \mathfrak{t}} [\psi_{\pi_\mathfrak{t}^{-1}} \psi_{\pi_\mathfrak{t}} e(\mathfrak{t})] w_{\mathfrak{t}, \mathfrak{s}} \psi_{\pi_\mathfrak{s}^{-1}} E_{\mathfrak{t}', \mathfrak{s}}
\\
&=
\psi_{\pi_{\mathfrak{t}'}} w_{\mathfrak{t}', \mathfrak{t}} w_{\mathfrak{t}, \mathfrak{s}} \psi_{\pi_\mathfrak{s}^{-1}} E_{\mathfrak{t}', \mathfrak{s}}
\\
\Psi_{\mathfrak{t}', \mathfrak{t}}(w_{\mathfrak{t}', \mathfrak{t}}) \Psi_{\mathfrak{s}', \mathfrak{s}}(w_{\mathfrak{s}', \mathfrak{s}})
&=
\Psi_{\mathfrak{t}', \mathfrak{s}}( w_{\mathfrak{t}', \mathfrak{t}} w_{\mathfrak{s}', \mathfrak{s}}).
\end{align*} 
Finally, the maps $\Phi_\lambda$ and $\Psi_\lambda$ are inverse $A$-algebra isomorphisms; we deduce the isomorphism \eqref{equation:reduction_theoreme_lambda} and thus Theorem~\ref{theorem:disjoint_guiver}.

\begin{remarque}
\label{remark:image_generators}
For $\tuple{k} \in \K^\mathfrak{t}$, we write here $\tuple{k}^* \coloneqq \pi_\mathfrak{t} \cdot \tuple{k} \in \K^{\mathfrak{t}^\lambda}$.
Using Proposition~\ref{proposition:inversion_psi_e(sigma)} and Lemmas~\ref{lemma:ya_psipisigma}, \ref{lemma:independence_psi_pi_sa_sigma}, we can give the images of the generators of $e(\lambda)\H_n(Q)$ for each $\mathfrak{t} \in J^\lambda$ and $\tuple{k} \in \K^\mathfrak{t}$:
\begin{align*}
\Psi_\lambda(e(\tuple{k})) &=
e(\tuple{k}^*) E_{\mathfrak{t}, \mathfrak{t}},
\\
\forall a \in \{1, \dots, n\},
\Psi_\lambda(y_a e(\tuple{k}))
&=
y_{\pi_\mathfrak{t}(a)} e(\tuple{k}^*) E_{\mathfrak{t}, \mathfrak{t}},
\\
\forall a \in \{1, \dots, n-1\}, 
\Psi_\lambda(\psi_a e(\tuple{k}))
&=
\psi_{\pi_{s_a \cdot \mathfrak{t}} s_a \pi_\mathfrak{t}^{-1}} e(\tuple{k}^*) E_{s_a \cdot \mathfrak{t}, \mathfrak{t}}
\\
&= 
\begin{cases}
e(\tuple{k}^*) E_{s_a \cdot \mathfrak{t}, \mathfrak{t}} & \text{if } \mathfrak{t}_a \neq \mathfrak{t}_{a+1},
\\
\psi_{\pi_\mathfrak{t}(a)} e(\tuple{k}^*) E_{\mathfrak{t}, \mathfrak{t}} & \text{if } \mathfrak{t}_a = \mathfrak{t}_{a+1}
\end{cases}
\end{align*}
(note that $\pi_{s_a \cdot \mathfrak{t}} s_a \pi_{\mathfrak{t}}^{-1} =
\begin{cases}
\mathrm{id} & \text{if } \mathfrak{t}_a \neq \mathfrak{t}_{a+1}
\\
s_{\pi_\mathfrak{t}(a)} & \text{if } \mathfrak{t}_a = \mathfrak{t}_{a+1}
\end{cases}$, cf. Lemma~\ref{lemma:pisigma_sa_sigma}). We observe that these images look like those described in \cite[(22)]{JaPA} and \cite[(3.2)--(3.4)]{PA}.
\end{remarque}

\begin{remarque}
We consider the setting of Proposition~\ref{proposition:gradation_quiver_Hecke_algebra}. We can prove that the algebra isomorphism $\H_n(Q) \simeq \oplus_{\lambda \comp_d n} \mathrm{Mat}_{m_\lambda} \H_\lambda(Q)$ is  a graded isomorphism  (with the canonical gradings on the direct sum,  matrix algebras and tensor products).
In particular, we have (recall the notation $\tuple{k}^\mathfrak{t}$ of \eqref{equation:definition_isigma}):
\[
\deg \psi_{\pi_{\mathfrak{t}'}^{-1}} e(\tuple{k}) = \deg \psi_{\pi_\mathfrak{t}} e(\tuple{k}^\mathfrak{t}) = 0,
\]
as a consequence of Lemma~\ref{lemma:pi_really_permutes} and Remark~\ref{remark:pi-1_really_permutes}. Indeed,
if for $\mathfrak{s} \in J^n$  and $a \in \{1, \dots, n-1\}$ we have $\mathfrak{s}_a \neq \mathfrak{s}_{a+1}$ then for any $\tuple{k} \in \K^\mathfrak{s}$ we have
$c_{k_a, k_{a+1}} = 0$ (cf. \eqref{equation:cartan_matrix_block_diagonal}).
\end{remarque}

\subsection{Cyclotomic version}

Let us consider a weight $\tuple{\Lambda} = (\Lambda_k)_{k \in \K} \in \mathbb{N}^{\K}$; for $j \in J$, we write $\tuple{\Lambda}^j \in \mathbb{N}^{\K_j}$ the restriction of $\tuple{\Lambda}$ to $\K_j$. We show here how the isomorphism of Theorem~\ref{theorem:disjoint_guiver} is compatible with cyclotomic quotients, as defined in \eqref{relation:quiver_cyclo_KK}.

\subsubsection{Factorisation theorem}

For $\lambda \comp_d n$, we define the cyclotomic quotient of $\H_\lambda(Q)$  by:
\[
\H^{\tuple{\Lambda}}_\lambda(Q) \coloneqq \H_{\lambda_1}^{\tuple{\Lambda}^1}(Q^1) \otimes \cdots \otimes \H_{\lambda_d}^{\tuple{\Lambda}^d}(Q^d),
\]
that is, $\H_\lambda^{\tuple{\Lambda}}(Q)$ is the quotient of $\H_\lambda(Q)$ by the two-sided ideal generated by the elements:
\begin{equation}
\label{equation:relation_cyclo_tensor}
\sum_{m = 0}^{\Lambda_{k_a}} c_m y_a^m e(\tuple{k}) = 0 \qquad \forall \tuple{k} \in \K^{\mathfrak{t}^\lambda}, \forall a \in \{\boldsymbol{\lambda}_0 + 1, \dots, \boldsymbol{\lambda}_{d-1} + 1\}.
\end{equation}

\begin{theoreme}
\label{theorem:disjoint_guiver_cyclotomic}
The isomorphism of Theorem~\ref{theorem:disjoint_guiver} factors through the cyclotomic quotients, in other words we have:
\[
\H_n^{\tuple{\Lambda}}(Q) \simeq \bigoplus_{\lambda \comp_d n} \mathrm{Mat}_{m_\lambda} \H_\lambda^{\tuple{\Lambda}}(Q).
\]
\end{theoreme}

\begin{proof}
Let $\lambda \comp_d n$ and let $\mathfrak{I}$ (resp. $\mathfrak{I}_\lambda$) be the two-sided ideal of $e(\lambda)\H_n(Q)$ (resp. $\H_\lambda(Q)$) generated by the elements in \eqref{relation:quiver_cyclo_KK} (resp. \eqref{equation:relation_cyclo_tensor}). It suffices to prove that $\Psi_\lambda(\mathfrak{I})  = \mathrm{Mat}_{m_\lambda}\mathfrak{I}_{\lambda}$: we will prove $\Psi_\lambda(\mathfrak{I}) \subseteq \mathrm{Mat}_{m_\lambda} \mathfrak{I}_\lambda$ and $\mathfrak{I} \supseteq \Phi_\lambda(\mathrm{Mat}_{m_\lambda}\mathfrak{I}_\lambda)$.
\begin{itemize}
\item Each element of $\mathfrak{I}$ can be written as:
\[
\sum_{v, w \in \H_n(Q)} \sum_{\mathfrak{t} \in J^\lambda} \sum_{\tuple{k}^\mathfrak{t} \in \K^\mathfrak{t}} v \left[\sum_{m = 0}^{\Lambda_{k^\mathfrak{t}_1}} c_m y_1^m e(\tuple{k}^\mathfrak{t})\right] w.
\]
By \eqref{relation:quiver_psiae(i)}, Proposition~\ref{proposition:inversion_psi_e(sigma)} and Lemma~\ref{lemma:ya_psipisigma}, the previous element becomes:
\[
\sum_{v, w \in \H_n(Q)} \sum_{\mathfrak{t} \in J^\lambda} \sum_{\tuple{k}^\mathfrak{t} \in \K^\mathfrak{t}} v \psi_{\pi_\mathfrak{t}^{-1}} \left[ \sum_{m = 0}^{\Lambda_{k^\mathfrak{t}_1}} c_m y_{\pi_\mathfrak{t}(1)}^m  e(\pi_\mathfrak{t} \cdot \tuple{k}^\mathfrak{t})\right] \psi_{\pi_\mathfrak{t}} w.
\]
As $\Psi_\lambda$ is an algebra homomorphism and $\mathrm{Mat}_{m_\lambda} \mathfrak{I}_\lambda$ is a two-sided ideal, it suffices to prove that for any $\tuple{k} \in \K^{\mathfrak{t}^\lambda}$ and $\mathfrak{t} \in J^\lambda$ we have, with $a \coloneqq \pi_\mathfrak{t}(1)$:
\[
\sum_{m = 0}^{\Lambda_{k_a}} c_m \Psi_\lambda (y_a^m e(\tuple{k})) \in \mathrm{Mat}_{m_\lambda}(\mathfrak{I}_\lambda).
\]
The above element is exactly:
\[
\sum_{m = 0}^{\Lambda_{k_a}} c_m y_a^m e(\tuple{k}) E_{\mathfrak{t}^\lambda, \mathfrak{t}^\lambda}
\]
(recall that $\pi_{\mathfrak{t}^\lambda} = \mathrm{id}$), hence we are done since the element $\sum_{m = 0}^{\Lambda_{k_a}} c_m y_a^m e(\tuple{k})$ lies in $\mathfrak{I}_\lambda$; in particular, there is a $j \in J$ such that $a = \boldsymbol{\lambda}_{j-1} + 1$, cf. Proposition~\ref{proposition:explicit_pi}.

\item Each element of $\mathrm{Mat}_{m_\lambda} \mathfrak{I}_\lambda$ can be written:
\[
\sum_{\mathfrak{t}, \mathfrak{t}' \in J^\lambda} \sum_{v, w \in \H_\lambda(Q)} \sum_{\tuple{k} \in \K^{\mathfrak{t}^\lambda}} \sum_{a \in \{\boldsymbol{\lambda}_0 + 1, \dots, \boldsymbol{\lambda}_{d-1} + 1\}} v \left[\sum_{m = 0}^{\Lambda_{k_a}} c_m y_a^m e(\tuple{k}) E_{\mathfrak{t}', \mathfrak{t}}\right] w.
\]

As $\Phi_\lambda$ is an algebra homomorphism and $\mathfrak{I}$ is a two-sided ideal, it suffices to prove that for any $\mathfrak{t}, \mathfrak{t}' \in J^\lambda, \tuple{k} \in \K^{\mathfrak{t}^\lambda}$ and $a = \boldsymbol{\lambda}_{j-1} + 1$ for $j \in J$ we have:
\[
\sum_{m = 0}^{\Lambda_{k_a}} c_m \Phi_\lambda(y_a^m e(\tuple{k}) E_{\mathfrak{t}', \mathfrak{t}}) \in \mathfrak{I}.
\]

We consider an element $\mathfrak{s} \in J^\lambda$  which verifies $\mathfrak{s}_1 = j$ (we can take for instance $\mathfrak{s} \coloneqq (1, a) \cdot \mathfrak{t}^\lambda$); note that $\pi_\mathfrak{s}(1) = \boldsymbol{\lambda}_{j-1} + 1 = a$. Using \eqref{equation:Esigmasigma'_tautau'}, we can write the above element:
\[
\sum_{m = 0}^{\Lambda_{k_a}} c_m \Phi_\lambda(y_a^m e(\tuple{k}) E_{\mathfrak{t}', \mathfrak{s}}) \Phi_\lambda (E_{\mathfrak{s}, \mathfrak{t}}),
\]
hence it suffices to prove that
\[
\alpha \coloneqq \sum_{m = 0}^{\Lambda_{k_a}} c_m \Phi_\lambda(y_a^m e(\tuple{k}) E_{\mathfrak{t}', \mathfrak{s}}),
\]
belongs to the ideal $\mathfrak{I}$. We get:
\begin{align*}
\alpha
&= \sum_{m = 0}^{\Lambda_{k_a}} c_m \psi_{\pi_{\mathfrak{t}'}^{-1}} y_a^m e(\tuple{k}) \psi_{\pi_\mathfrak{s}}
\\
&= \sum_{m = 0}^{\Lambda_{k_a}} c_m \psi_{\pi_{\mathfrak{t}'}^{-1}} \psi_{\pi_\mathfrak{s}} y_{\pi_\mathfrak{s}^{-1}(a)}^m e(\tuple{k}^\mathfrak{s})
\\
\alpha &= \psi_{\pi_{\mathfrak{t}'}^{-1}} \psi_{\pi_\mathfrak{s}} \sum_{m = 0}^{\Lambda_{k_a}} c_m  y_1^m e(\tuple{k}^\mathfrak{s}),
\end{align*}
where we recall the notation $\tuple{k}^\mathfrak{s} \in \K^\mathfrak{s}$ from \eqref{equation:definition_isigma}. We have $k^\mathfrak{s}_1 = k_a$, thus we get:
\[
\alpha = \psi_{\pi_{\mathfrak{t}'}^{-1}} \psi_{\pi_\mathfrak{s}} \underbrace{\sum_{m = 0}^{\Lambda_{k^\mathfrak{s}_1}} c_m  y_1^m e(\tuple{k}^\mathfrak{s})}_{\in \mathfrak{I}},
\]
and we are done.
\end{itemize}
\end{proof}

\subsubsection{An alternative proof}

We explain here how we can get Theorem~\ref{theorem:disjoint_guiver_cyclotomic} from \cite[\textsection 3.2.3]{SVV}. We make the following assumption:
\[
Q_{k, k'}(u, v) = r_{k, k'} (u-v)^{-c_{k, k'}},
\]
where the $r_{k, k'}, c_{k, k'}$ are some scalars; we refer to \cite[\textsection 3.2.2]{SVV} for more details about the setting. In particular, we suppose that $\tuple{\Lambda} = \sum_{k \in \K} \Lambda_k \omega_k$ where the $\omega_k$ are the fundamental weights related to the ambient Cartan datum. 
We define the set $\K_{\tuple{\Lambda}}$ by (cf. \cite[\textsection 3.1.3]{SVV}):
\[
\K_{\tuple{\Lambda}} \coloneqq \big\lbrace (k, m) : k \in \K, m \in \{1, \dots, \Lambda_k\}\big\rbrace;
\]
we will also use the sets $\K_{\tuple{\Lambda}^j} \coloneqq {(\K_j)}_{\tuple{\Lambda}^j}$ for $j \in J$. For $t = (k, m) \in \K_{\tuple{\Lambda}}$, we write $k_t \coloneqq k$ and $\omega_t \coloneqq \omega_k$.
Finally, for convenience we introduce some notation:
\begin{itemize}
\item we write $\comp_{\tuple{\Lambda}}$ instead of $\comp_{\K_{\tuple{\Lambda}}}$;
\item if $\mu$ (respectively $\nu$) is an $r$-composition (resp. $s$-composition), we set $m_\mu^\nu \coloneqq \frac{\nu_1!\dots \nu_s!}{\mu_1 ! \dots \mu_r!}$; in particular with the $1$-partition $\nu \coloneqq (\mu_1 + \cdots + \mu_r)$ we recover $m_\mu^\nu = m_\mu$.
\end{itemize}

\begin{theoreme}[{\cite[Theorem~$3.15$]{SVV}}]
\label{theorem:svv}
There is an algebra isomorphism
\[
\H_n^{\tuple{\Lambda}}(Q) \simeq \bigoplus_{(n_t)_t \comp_{\tuple{\Lambda}} n} \mathrm{Mat}_{m_{(n_t)}} \left( \bigotimes_{t \in \K_{\tuple{\Lambda}}} \H_{n_t}^{\omega_t}(Q)\right).
\]
\end{theoreme}

If we look closer to the proof of \cite{SVV}, comparing to the proof of Theorem~\ref{theorem:disjoint_guiver} the idea is still to consider some elements $e(\mathfrak{t})$ but with idempotents which ``refine'' the idempotents $e(\tuple{k})$ (these idempotents are indexed by $\K_{\tuple{\Lambda}}^n$, which is much bigger than $\K^n$). In particular, the following isomorphism for any $\lambda \comp_d n$,  writing $(n^j_t)_t$ for the restriction of $(n_t)_{t \in \K_{\tuple{\Lambda}}}$ to $\K_{\tuple{\Lambda}^j}$:
\begin{equation}
\label{equation:svv_e(lambda)part}
\mathrm{Mat}_{m_\lambda} \H_\lambda^{\tuple{\Lambda}}(Q) \simeq \bigoplus_{\substack{
(n_t)_t \comp_{\tuple{\Lambda}} n
\\
\text{s.t.}\, \forall j, (n^j_t)_t \comp_{\tuple{\Lambda}^j} \lambda_j
}}
\mathrm{Mat}_{m_{(n_t)}} \left( \bigotimes_{t \in \K_{\tuple{\Lambda}}} \H_{n_t}^{\omega_t}(Q)\right)
\end{equation}
implies our Theorem~\ref{theorem:disjoint_guiver_cyclotomic} by summing over all $\lambda \comp_d n$.
In order to prove \eqref{equation:svv_e(lambda)part}, we can simply apply Theorem~\ref{theorem:svv} to the factors $\H_{\lambda_j}^{\tuple{\Lambda}^j}(Q^j)$ of $\H_\lambda^{\tuple{\Lambda}}(Q)$. We have, for $j \in J$:
\begin{equation}
\label{equation:svv_facteur}
\H_{\lambda_j}^{\tuple{\Lambda}_j}(Q^j)
\simeq
\bigoplus_{\left(n^j_t\right)_t \comp_{\tuple{\Lambda}^j} \lambda_j} \mathrm{Mat}_{m_{\left(n^j_t\right)}} \left( \bigotimes_{t \in \K_{\tuple{\Lambda}^j}} \H_{n^j_t}^{\omega_t}(Q^j)\right).
\end{equation}

Before going further, we give the following lemma. Let us mention that we can find a non-cyclotomic statement in \cite[Corollary~3.8]{Rou}; see also \cite[Proposition 2.4.6]{Ma2} and \cite[Lemma 1.16]{BoMa}.

\begin{lemme}
\label{lemma:svv_restriction_carquois}
Let $j \in J$. If $k \in \K_j$ then
$
\H_n^{\omega_k}(Q) \simeq \H_n^{\omega_k}(Q^j)$.
\end{lemme}

\begin{proof}
It suffices to prove that every $e(\tuple{k}) \in \H_n(Q)$ with $\tuple{k} \in \K^n \setminus \K_j^n$ vanishes in $\H_n^{\omega_k}(Q)$. To that extent, we prove by induction on $a \in \{1, \dots, n\}$ the following statement:
\begin{equation}
\label{equation:proof_svv_induction}
\forall \tuple{k} \in \K^n, [\exists b \in \{1, \dots, a\}, k_b \notin \K_j \implies e(\tuple{k}) = 0 \text{ in } \H_n^{\omega_k}(Q)].
\end{equation}
First, we shall verify this proposition for $a = 1$: let $\tuple{k} \in \K^n$ such that $\exists b \in \{1, \dots, 1\}, k_b \notin \K_j$. We obviously have $k_1 \notin \K_j$, in particular $k_1 \neq k$ thus it follows directly from the cyclotomic condition \eqref{relation:quiver_cyclo_KK} that $e(\tuple{k}) = 0$ in $\H_n^{\omega_k}(Q)$. We now assume that \eqref{equation:proof_svv_induction} is verified for some $a \in \{1, \dots, n - 1\}$ and we let $\tuple{k} \in \K^n$ such that $\exists b \in \{1, \dots, a+1\}, k_b \notin \K_j$. We know from the induction hypothesis that $e(\tuple{k}) = 0$ in $\H_n^{\omega_k}(Q)$ if $k_a \notin \K_j$ or $k_{a+1} \in \K_j$, hence it remains to deal with the case $k_a \in \K_j$ and $k_{a+1} \notin \K_j$. Recalling \eqref{equation:condition_Q}, this implies that $Q_{k_a, k_{a+1}} = 1$. In particular, the defining relation \eqref{relation:quiverQ_psia^2} becomes:
\begin{equation}
\label{equation:proof_svv_psi2}
\psi_a^2 e(\tuple{k}) = e(\tuple{k}).
\end{equation}
Besides, using \eqref{relation:quiver_psiae(i)} and the induction hypothesis, we have:
\[
\psi_a e(\tuple{k}) = \underbrace{e(s_a \cdot \tuple{k})}_{= 0} \psi_a = 0 \text{ in } \H_n^{\omega_k}(Q).
\]
Thus, left-multiplying by $\psi_a$ and using \eqref{equation:proof_svv_psi2} we get $e(\tuple{k}) = 0$ in $\H_n^{\omega_k}(Q)$ which ends the induction. Finally, for $a = n$, we get that if $\tuple{k} \in \K^n \setminus \K_j^n$ then $e(\tuple{k}) = 0$ in $\H_n^{\omega_k}(Q)$.
\end{proof}

We now go back to the proof of \eqref{equation:svv_e(lambda)part}.
Using Lemma~\ref{lemma:svv_restriction_carquois}, the isomorphism \eqref{equation:svv_facteur} gives:
\[
\H_\lambda^{\tuple{\Lambda}}(Q)
\simeq
\bigotimes_{j \in J} \H_{\lambda_j}^{\tuple{\Lambda}_j}(Q^j)
\simeq
\bigotimes_{j \in J}
\bigoplus_{\left(n^j_t\right)_t \comp_{\tuple{\Lambda}^j} \lambda_j}
\mathrm{Mat}_{m_{\left(n^j_t\right)}} \left( \bigotimes_{t \in \K_{\tuple{\Lambda}^j}} \H_{n^j_t}^{\omega_t}(Q)\right).
\]
We obtain:
\begin{align*}
\H_\lambda^{\tuple{\Lambda}}(Q)
&\simeq 
 \bigoplus_{\substack{
 (n_t)_t \comp_{\tuple{\Lambda}} n
 \\
 \text{s.t.}\, \forall j, (n^j_t)_t \comp_{\tuple{\Lambda}^j} \lambda_j
 }}
 \bigotimes_{j \in J}
 \mathrm{Mat}_{m_{\left(n^j_t\right)}} \left( \bigotimes_{t \in \K_{\tuple{\Lambda}^j}} \H_{n^j_t}^{\omega_t}(Q)\right)
\\
&\simeq
\bigoplus_{\substack{
 (n_t)_t \comp_{\tuple{\Lambda}} n
 \\
 \text{s.t.}\, \forall j, (n^j_t)_t \comp_{\tuple{\Lambda}^j} \lambda_j
 }}
 \mathrm{Mat}_{m_{\left(n^1_t\right)} \cdots m_{\left(n^d_t\right)}} \left(
 \bigotimes_{j \in J}
  \bigotimes_{t \in \K_{\tuple{\Lambda}^j}} \H_{n^j_t}^{\omega_t}(Q)\right)
\\
\H_\lambda^{\tuple{\Lambda}}(Q)
&\simeq
\bigoplus_{\substack{
 (n_t)_t \comp_{\tuple{\Lambda}} n
 \\
 \text{s.t.}\, \forall j, (n^j_t)_t \comp_{\tuple{\Lambda}^j} \lambda_j
 }}
 \mathrm{Mat}_{m_{(n_t)}^\lambda} \left(
\bigotimes_{t \in \K_{\tuple{\Lambda}}} \H_{n_t}^{\omega_t}(Q)\right).
\end{align*}
Finally, we deduce \eqref{equation:svv_e(lambda)part} and thus Theorem~\ref{theorem:disjoint_guiver_cyclotomic} from the equality $m_\lambda m_{(n_t)}^\lambda = m_{(n_t)}$.

\subsubsection{An application}
\label{subsubsection:an_application}

Here we explain how our Theorem~\ref{theorem:disjoint_guiver_cyclotomic} is related to our previous result Theorem~\ref{theorem:main}. We suppose that $A = \mathbb{C}$ and that all the $\K_j$ for $j \in J \simeq \{1, \dots, d\}$ have the same finite cardinality $e \geq 2$. Identifying the sets $\K_j$ with $I = \mathbb{Z}/e\mathbb{Z}$, we have $\K \simeq K = I \times J$. We also assume that $\tuple{\Lambda}^1 = \cdots = \tuple{\Lambda}^d$; in particular, we will simply write $\tuple{\Lambda}$ for each of these parts. Finally, we consider $q \in \mathbb{C}^{\times}$ a primitive $e$th root of unity, and we write:
\[
\mathrm{BK} : \widehat{\H}_n^{\tuple{\Lambda}}(q) \overset{\sim}\to \H_n^{\tuple{\Lambda}}(\Gamma_e)
\]
for the $\mathbb{C}$-algebra isomorphism of \cite{BrKl}.

We have an algebra isomorphism, constructed by Poulain d'Andecy \cite{PA}:
\[
\mathrm{JPA} : 
\widehat{\Y}_{d, n}^{\tuple{\Lambda}}(q) \overset{\sim}{\to} \bigoplus_{\lambda \comp_d n} \mathrm{Mat}_{m_\lambda} \widehat{\H}_\lambda^{\tuple{\Lambda}}(q),
\]
where $\widehat{\H}_\lambda^{\tuple{\Lambda}}(q) \coloneqq \widehat{\H}_{\lambda_1}^{\tuple{\Lambda}}(q) \otimes \cdots \otimes \widehat{\H}_{\lambda_d}^{\tuple{\Lambda}}(q)$. This isomorphism is a generalisation of the main result of \cite{JaPA} (which is in fact a particular case of a result of Lusztig~\cite{Lu}), and is defined on the generators as follows:
\begin{align}
\label{equation:JPA_ta}
\mathrm{JPA}(t_a) &= \sum_{\mathfrak{t} \in J^n} \xi^{\mathfrak{t}_a} E_{\mathfrak{t}, \mathfrak{t}},
\\
\label{equation:JPA_Xa}
\mathrm{JPA}(X_a) &= \sum_{\mathfrak{t} \in J^n} X_{\pi_\mathfrak{t}(a)} E_{\mathfrak{t}, \mathfrak{t}},
\\
\label{equation:JPA_psia}
\mathrm{JPA}(g_a) &= \sum_{\substack{\mathfrak{t} \in J^n \\ \mathfrak{t}_a = \mathfrak{t}_{a+1}}} g_{\pi_\mathfrak{t}(a)} E_{\mathfrak{t}, \mathfrak{t}}
+ \sum_{\substack{\mathfrak{t} \in J^n \\ \mathfrak{t}_a \neq \mathfrak{t}_{a+1}}} \sqrt{q} E_{\mathfrak{t}, s_a \cdot \mathfrak{t}}.
\end{align}

\begin{remarque}
Note two slight differences with \cite{PA}:
\begin{itemize}
\item our elements $E_{\mathfrak{t}', \mathfrak{t}}$ are written $E_\chi$, where $\chi$ is a character of $(\mathbb{Z}/d\mathbb{Z})^n = J^n$;
\item Poulain d'Andecy considers left cosets instead of our right ones, in particular his minimal length representatives $\pi_\chi$ verify $\pi_\chi = \pi_\mathfrak{t}^{-1}$.
\end{itemize}
\end{remarque}

We recall from \eqref{equation:gamma=union_gammae} that $\Gamma = \amalg_{j \in J} \Gamma_e$; in particular, its vertex set is exactly $\K \simeq K = I \times J$.
The two previous results, together with our Theorem~\ref{theorem:disjoint_guiver_cyclotomic}, gives straightforwardly the following theorem.

\begin{theoreme}
\label{theorem:yh_are_cyclotomic_qh}
We have an algebra isomorphism:
\[
\Phi_n^{\tuple{\Lambda}} \circ \mathrm{BK} \circ \mathrm{JPA} : 
\widehat{\Y}_{d, n}^{\tuple{\Lambda}}(q) \simeq \H_n^{\tuple{\Lambda}}(\Gamma),
\]
where:
\begin{itemize}
\item the homomorphism $\mathrm{BK} : \oplus_{\lambda} \mathrm{Mat}_{m_\lambda} \widehat{\H}_\lambda^{\tuple{\Lambda}}(q) \to \oplus_\lambda \mathrm{Mat}_{m_\lambda} \H_\lambda^{\tuple{\Lambda}}(\Gamma)$ is naturally induced by
$\mathrm{BK} : \widehat{\H}_n^{\tuple{\Lambda}}(q) \to \H_n^{\tuple{\Lambda}}(\Gamma_e)$;
\item the homomorphism $\Phi_n^{\tuple{\Lambda}} : \oplus_\lambda \mathrm{Mat}_{m_\lambda} \H_\lambda^{\tuple{\Lambda}}(\Gamma) \to \H_n^{\tuple{\Lambda}}(\Gamma)$ is the isomorphism of Theorem~\ref{theorem:disjoint_guiver_cyclotomic}, that is, induced by the homomorphism $\Phi_n : \oplus_\lambda \mathrm{Mat}_{m_\lambda} \H_\lambda(\Gamma) \to \H_n(\Gamma)$.
\end{itemize}
\end{theoreme}

An algebra isomorphism $\widehat{\Y}_{d, n}^{\tuple{\Lambda}}(q) \to \H_n^{\tuple{\Lambda}}(\Gamma)$ was already constructed in Theorem~\ref{theorem:main}; we shall denote it by $\widetilde{\mathrm{BK}}$.
An interesting question is to know whether we recover the same isomorphism as above. In other words, does the diagram of Figure~\ref{figure:commutative_diagram} commute?
\begin{figure}[h]
\begin{center}
\begin{tikzpicture}
[descr/.style={fill=white,inner sep=5pt},
>=angle 90]
\matrix (m) [matrix of math nodes, row sep=3em,
column sep=5em]
{\widehat{\Y}_{d, n}^{\tuple{\Lambda}}(q)
&
\displaystyle\bigoplus_{\lambda \comp_d n}^{\phantom{\lambda \comp_d n}}
\mathrm{Mat}_{m_\lambda} \widehat{\H}_\lambda^{\tuple{\Lambda}}(q)
\\
\H_n^{\tuple{\Lambda}}(\Gamma)
&\displaystyle\bigoplus_{\lambda \comp_d n}^{\phantom{\lambda \comp_d n}}
\mathrm{Mat}_{m_\lambda} \H_\lambda^{\tuple{\Lambda}}(\Gamma)
\\};

\coordinate [shift={(0ex, -2.3ex)}]
(m12) at (m-1-2);
\coordinate [shift={(0ex, 2.3ex)}]
(m22) at (m-2-2);

\path[->]
(m-1-1)
edge node[auto] {$\mathrm{JPA}$} 
(m-1-2);

\path[->] (m-1-1) edge node[auto] {$\widetilde{\mathrm{BK}}$} (m-2-1);

\path[->]
(m12) edge node[auto] {$\mathrm{BK}$} (m22);

\path[->]
(m-2-2)
edge node[auto] {$\Phi_n^{\tuple{\Lambda}}$} (m-2-1);
\end{tikzpicture}
\caption{A commutative diagram?}
\label{figure:commutative_diagram}
\end{center}
\end{figure}

As we deal with algebra homomorphisms, it suffices to check that the images of the generators of $\widehat{\Y}_{d, n}^{\tuple{\Lambda}}(q)$ are the same. We will use the following notation: for $\mathfrak{t} \in J^n$ we set $\mathfrak{t}^* \coloneqq \pi_\mathfrak{t} \cdot \mathfrak{t}$. (With $\lambda \coloneqq [\mathfrak{t}]$, we have of course $\mathfrak{t}^* = \mathfrak{t}^\lambda$.) Moreover, we will keep on using the notation $\mathfrak{t} \in J^n$ of Section~\ref{section:disjoint_guiver} for the elements we denoted by $\tuple{j} \in J^n$ from Section~\ref{section:setting} to Section~\ref{section:degenerate}.

\paragraph*{Image of $t_a$.} Let $1 \leq a \leq n$. Recall from \textsection\ref{subsection:definition_YH_generators} that:
\begin{equation}
\label{equation:image_BKtilde_ta}
\widetilde{\mathrm{BK}}(t_a) = \sum_{\tuple{j} \in J^n} e(\tuple{j}) \xi^{j_a} = \sum_{\mathfrak{t} \in J^n} e(\mathfrak{t}) \xi^{\mathfrak{t}_a} \in \H_n^{\tuple{\Lambda}}(\Gamma).
\end{equation}
Recalling \eqref{equation:JPA_ta}, we obtain:
\[
\mathrm{BK} \circ \mathrm{JPA}(t_a) = \sum_{\mathfrak{t} \in J^n} \sum_{\tuple{k} \in K^{\mathfrak{t}^*}} \xi^{\mathfrak{t}_a} e(\tuple{k}) E_{\mathfrak{t}, \mathfrak{t}}
\in \bigoplus_{\lambda \comp_d n} \mathrm{Mat}_{m_\lambda}\H_\lambda^{\tuple{\Lambda}}(\Gamma).
\]
Hence, with the usual manipulations, we get:
\begin{align*}
\Phi_n^{\tuple{\Lambda}} \circ \mathrm{BK} \circ \mathrm{JPA}(t_a)
&=
 \sum_{\mathfrak{t} \in J^n} \sum_{\tuple{k} \in K^{\mathfrak{t}^*}} \xi^{\mathfrak{t}_a} \psi_{\pi_\mathfrak{t}^{-1}} \psi_{\pi_\mathfrak{t}} e(\pi_\mathfrak{t}^{-1} \cdot \tuple{k})
\\
&=
\sum_{\mathfrak{t} \in J^n} \sum_{\tuple{k} \in K^{\mathfrak{t}}} \xi^{\mathfrak{t}_a} e(\tuple{k})
\\
\Phi_n^{\tuple{\Lambda}} \circ \mathrm{BK} \circ \mathrm{JPA}(t_a)
&=
\sum_{\mathfrak{t} \in J^n} \xi^{\mathfrak{t}_a} e(\mathfrak{t}) \in \H_n^{\tuple{\Lambda}}(\Gamma),
\end{align*}
thus it coincides with \eqref{equation:image_BKtilde_ta}.

\paragraph*{Image of $X_a$.} (It is in fact enough to study the case $a = 1$).
Let $1 \leq a \leq n$. Recall from  \textsection\ref{subsection:definition_YH_generators} that:
\begin{equation}
\label{equation:image_BKtilde_Xa}
\widetilde{\mathrm{BK}}(X_a) = \sum_{\tuple{i} \in I^n} q^{i_a}(1-y_a) e(\tuple{i}) = \sum_{\tuple{k} \in K^n} q^{k_a} (1 - y_a) e(\tuple{k}) \in \H_n^{\tuple{\Lambda}}(\Gamma),
\end{equation}
where we write $q^k \coloneqq q^i$ for $k = (i, j) \in K = I \times J$. Recalling \eqref{equation:JPA_Xa}, we get:
\[
\mathrm{BK} \circ \mathrm{JPA}(X_a) = \sum_{\mathfrak{t} \in J^n} X_{\pi_\mathfrak{t}(a)} E_{\mathfrak{t}, \mathfrak{t}}
\in
\bigoplus_{\lambda \comp_d n} \mathrm{Mat}_{m_\lambda}\H_\lambda^{\tuple{\Lambda}}(\Gamma).
\]
We write, where $j \coloneqq \mathfrak{t}_a$ and $\lambda \coloneqq [\mathfrak{t}]$:
\[
X_{\pi_\mathfrak{t}(a)} = \sum_{\tuple{k}^j \in \K_j^{\lambda_j}} q^{k^j_{\pi_\mathfrak{t}(a)}}(1-y_{\pi_\mathfrak{t}(a)})e(\tuple{k}^j) \in \H_{\lambda_j}^{\tuple{\Lambda}}(\Gamma_e) \subseteq \H_\lambda^{\tuple{\Lambda}}(\Gamma),
\]
where $\tuple{k}^j \in \K_j^{\lambda_j}$ is indexed by $(\boldsymbol{\lambda}_{j-1}+1,\dots, \boldsymbol{\lambda}_j)$. Hence, we have:
\[
X_{\pi_\mathfrak{t}(a)} = \sum_{\tuple{k} \in K^{\mathfrak{t}^\lambda}}  q^{k_{\pi_\mathfrak{t}(a)}}(1-y_{\pi_\mathfrak{t}(a)})e(\tuple{k}) \in \H_\lambda^{\tuple{\Lambda}}(\Gamma),
\]
and thus:
\begin{align*}
\Phi_n^{\tuple{\Lambda}} \circ \mathrm{BK} \circ \mathrm{JPA} (X_a)
&=
\sum_{\mathfrak{t} \in J^n} \sum_{\tuple{k} \in K^{\mathfrak{t}^*}} q^{k_{\pi_\mathfrak{t}(a)}} \Phi_n^{\tuple{\Lambda}}\big((1-y_{\pi_\mathfrak{t}(a)})e(\tuple{k}) E_{\mathfrak{t}, \mathfrak{t}}\big)
\\
&=
\sum_{\mathfrak{t} \in J^n} \sum_{\tuple{k} \in K^{\mathfrak{t}^*}} q^{k_{\pi_\mathfrak{t}(a)}} \psi_{\pi_\mathfrak{t}^{-1}}  (1-y_{\pi_\mathfrak{t}(a)})e(\tuple{k}) \psi_{\pi_\mathfrak{t}}
\\
&=
\sum_{\mathfrak{t} \in J^n} \sum_{\tuple{k} \in K^{\mathfrak{t}^*}} q^{k_{\pi_\mathfrak{t}(a)}} \psi_{\pi_\mathfrak{t}^{-1}}  (1-y_{\pi_\mathfrak{t}(a)})\psi_{\pi_\mathfrak{t}} e(\tuple{k}^\mathfrak{t}) 
\\
\Phi_n^{\tuple{\Lambda}} \circ \mathrm{BK} \circ \mathrm{JPA} (X_a)
&=
\sum_{\mathfrak{t} \in J^n} \sum_{\tuple{k}^\mathfrak{t} \in K^\mathfrak{t}} q^{k^\mathfrak{t}_a} \psi_{\pi_\mathfrak{t}^{-1}} \psi_{\pi_\mathfrak{t}} (1 - y_a) e(\tuple{k}^\mathfrak{t}),
\end{align*}
where we have used Lemma~\ref{lemma:ya_psipisigma}. Hence, using Proposition~\ref{proposition:inversion_psi_e(sigma)} we finally get:
\[
\Phi_n^{\tuple{\Lambda}} \circ \mathrm{BK} \circ \mathrm{JPA} (X_a)
=
\sum_{\mathfrak{t} \in J^n} \sum_{\tuple{k} \in K^\mathfrak{t}} q^{k_a} (1 - y_a) e(\tuple{k})
=
\sum_{\tuple{k} \in K^n} q^{k_a} (1 - y_a) e(\tuple{k}),
\]
which is \eqref{equation:image_BKtilde_Xa}.

\paragraph*{Image of $g_a$.}
Let $1 \leq a < n$. Recall from \textsection\ref{subsection:definition_YH_generators} that:
\begin{equation}
\label{equation:image_BKtilde_ga}
\widetilde{\mathrm{BK}}(g_a) = \sum_{\tuple{k} \in K^n} (\psi_a Q_a(\tuple{k}) - P_a(\tuple{k})) e(\tuple{k})
\in \H_n^{\tuple{\Lambda}}(\Gamma),
\end{equation}
where $Q_a(\tuple{k}), P_a(\tuple{k}) \in \mathbb{C}[[y_a, y_{a+1}]]$ are some power series. For convenience, we write $\widehat{Q}_a(\tuple{k}), \widehat{P}_a(\tuple{k}) \in \mathbb{C}[[Y_a, Y_{a+1}]]$ the underlying power series, which verify $\widehat{Q}_a(\tuple{k})(y_a, y_{a+1}) = Q_a(\tuple{k})$ and $\widehat{P}_a(\tuple{k})(y_a, y_{a+1}) = P_a(\tuple{k})$. These power series depend only on $k_a$ and $k_{a+1}$, that is, $\widehat{Q}_a(\tuple{k})(Y, Y') = \widehat{Q}_{a'}(\tuple{k}')(Y, Y')$ and $\widehat{P}_a(\tuple{k})(Y, Y') = \widehat{P}_{a'}(\tuple{k}')(Y, Y')$ if $k_a = k'_{a'}$ and $k_{a+1} = k'_{a'+1}$.
Moreover, recall that if $\tuple{k}$ and $s_a \cdot \tuple{k}$ are labellings of two different $\mathfrak{t} \in J^n$, we have  $Q_a(\tuple{k}) = \sqrt{q}$ and $P_a(\tuple{k}) = 0$ (cf. Remark~\ref{remark:f_aj}).

Recalling \eqref{equation:JPA_psia}, we get:
\[
\mathrm{BK} \circ \mathrm{JPA}(g_a) = \sum_{\substack{\mathfrak{t} \in J^n \\ \mathfrak{t}_a = \mathfrak{t}_{a+1}}} g_{\pi_\mathfrak{t}(a)} E_{\mathfrak{t}, \mathfrak{t}}
+ \sum_{\substack{\mathfrak{t} \in J^n \\ \mathfrak{t}_a \neq \mathfrak{t}_{a+1}}} \sqrt{q} E_{\mathfrak{t}, s_a \cdot \mathfrak{t}}
\in
\bigoplus_{\lambda \comp_d n} \mathrm{Mat}_{m_\lambda}\H_\lambda^{\tuple{\Lambda}}(\Gamma).
\]
With $j \coloneqq \mathfrak{t}_a$ and $\lambda \coloneqq [\mathfrak{t}]$, we have:
\[
g_{\pi_\mathfrak{t}(a)} = \sum_{\tuple{k}^j \in \K_j^{\lambda_j}} (\psi_{\pi_\mathfrak{t}(a)} Q_{\pi_\mathfrak{t}(a)}(\tuple{k}^j) - P_{\pi_\mathfrak{t}(a)}(\tuple{k}^j)) e(\tuple{k}^j)
\in \H_{\lambda_j}^{\tuple{\Lambda}}(\Gamma_e) \subseteq \H_\lambda^{\tuple{\Lambda}}(\Gamma)
\]
(recall that $\tuple{k}^j \in \K_j^{\lambda_j}$ is indexed by $(\boldsymbol{\lambda}_{j-1}+1,\dots, \boldsymbol{\lambda}_j)$),
hence:
\[
g_{\pi_\mathfrak{t}(a)} = \sum_{\tuple{k} \in K^{\mathfrak{t}^\lambda}} (\psi_{\pi_\mathfrak{t}(a)} Q_{\pi_\mathfrak{t}(a)}(\tuple{k}) - P_{\pi_\mathfrak{t}(a)}(\tuple{k})) e(\tuple{k})
\in \H_\lambda^{\tuple{\Lambda}}(\Gamma).
\]

We obtain:
\begin{multline}
\label{equation:image_ga_beginning}
\Phi_n^{\tuple{\Lambda}} \circ \mathrm{BK} \circ \mathrm{JPA}(g_a)
=
\underbrace{
\sum_{\substack{\mathfrak{t} \in J^n \\ \mathfrak{t}_a = \mathfrak{t}_{a+1}}}
\sum_{\tuple{k} \in K^{\mathfrak{t}^*}} \psi_{\pi_\mathfrak{t}^{-1}} (\psi_{\pi_\mathfrak{t}(a)} Q_{\pi_\mathfrak{t}(a)}(\tuple{k}) - P_{\pi_\mathfrak{t}(a)}(\tuple{k}))  e(\tuple{k}) \psi_{\pi_\mathfrak{t}}
}_{S_1}
\\
+ \underbrace{\sum_{\substack{\mathfrak{t} \in J^n \\ \mathfrak{t}_a \neq \mathfrak{t}_{a+1}}}  \sum_{\tuple{k} \in K^{\mathfrak{t}^*}} \sqrt{q}\psi_{\pi_\mathfrak{t}^{-1}} e(\tuple{k}) \psi_{\pi_{s_a \cdot \mathfrak{t}}}.
}_{S_2}
\end{multline}

We first focus on the first sum $S_1$; let $\mathfrak{t} \in J^n$ such that $\mathfrak{t}_a = \mathfrak{t}_{a+1}$ and $\tuple{k} \in K^{\mathfrak{t}^*}$. 
We can notice that thanks to Proposition~\ref{proposition:explicit_pi} we have $\pi_\mathfrak{t}(a+1) = \pi_\mathfrak{t}(a) + 1$. Using Lemma~\ref{lemma:ya_psipisigma} and the properties of $\widehat{Q}$ we have (recalling the notation $\tuple{k}^\mathfrak{t}$ introduced at \eqref{equation:definition_isigma}):
\begin{align*}
Q_{\pi_\mathfrak{t}(a)}(\tuple{k}) e(\tuple{k}) \psi_{\pi_\mathfrak{t}}
&=
\widehat{Q}_{\pi_\mathfrak{t}(a)}(\tuple{k})(y_{\pi_\mathfrak{t}(a)}, y_{\pi_\mathfrak{t}(a)+1}) \psi_{\pi_\mathfrak{t}} e(\tuple{k}^\mathfrak{t})
\\
&= \psi_{\pi_\mathfrak{t}} \widehat{Q}_{\pi_\mathfrak{t}(a)}(\tuple{k})(y_a, y_{a+1}) e(\tuple{k}^\mathfrak{t})
\\
&= \psi_{\pi_\mathfrak{t}} \widehat{Q}_{a}(\tuple{k}^\mathfrak{t})(y_a, y_{a+1}) e(\tuple{k}^\mathfrak{t})
\\
Q_{\pi_\mathfrak{t}(a)}(\tuple{k}) e(\tuple{k}) \psi_{\pi_\mathfrak{t}}
&=
\psi_{\pi_\mathfrak{t}} Q_a(\tuple{k}^\mathfrak{t}) e(\tuple{k}^\mathfrak{t}).
\end{align*}
The same proof gives $P_{\pi_\mathfrak{t}(a)}(\tuple{k}) e(\tuple{k}) \psi_{\pi_\mathfrak{t}}
=
\psi_{\pi_\mathfrak{t}} P_a(\tuple{k}^\mathfrak{t}) e(\tuple{k}^\mathfrak{t})$. 
We have thus:
\[
\psi_{\pi_\mathfrak{t}^{-1}} (\psi_{\pi_\mathfrak{t}(a)} Q_{\pi_\mathfrak{t}(a)}(\tuple{k}) - P_{\pi_\mathfrak{t}(a)}(\tuple{k}))  e(\tuple{k}) \psi_{\pi_\mathfrak{t}}
=
 \psi_{\pi_\mathfrak{t}^{-1}} (\psi_{\pi_\mathfrak{t}(a)} \psi_{\pi_\mathfrak{t}} Q_a(\tuple{k}^\mathfrak{t})
- \psi_{\pi_\mathfrak{t}} P_a(\tuple{k}^\mathfrak{t}) )  e(\tuple{k}^\mathfrak{t}).
\]
Using \eqref{relation:quiver_y_ae(i)} and Lemma~\ref{lemma:independence_psi_pi_sa_sigma}, we obtain:
\[
\psi_{\pi_\mathfrak{t}^{-1}} (\psi_{\pi_\mathfrak{t}(a)} Q_{\pi_\mathfrak{t}(a)}(\tuple{k}) - P_{\pi_\mathfrak{t}(a)}(\tuple{k}))  e(\tuple{k}) \psi_{\pi_\mathfrak{t}}
=
\psi_{\pi_\mathfrak{t}^{-1}} \psi_{\pi_\mathfrak{t}}(\psi_a  Q_a(\tuple{k}^\mathfrak{t}) - P_a(\tuple{k}^\mathfrak{t}) )  e(\tuple{k}^\mathfrak{t}).
\]
Finally, since $s_a \cdot \mathfrak{t} = \mathfrak{t}$, we get by Proposition~\ref{proposition:inversion_psi_e(sigma)}:
\[
\psi_{\pi_\mathfrak{t}^{-1}} (\psi_{\pi_\mathfrak{t}(a)} Q_{\pi_\mathfrak{t}(a)}(\tuple{k}) - P_{\pi_\mathfrak{t}(a)}(\tuple{k}))  e(\tuple{k}) \psi_{\pi_\mathfrak{t}}
=
(\psi_a Q_a(\tuple{k}^\mathfrak{t}) - P_a(\tuple{k}^\mathfrak{t})) e(\tuple{k}^\mathfrak{t}),
\]
so the first sum becomes:
\begin{multline}
\label{equation:image_ga_sigmaa=}
\sum_{\substack{\mathfrak{t} \in J^n \\ \mathfrak{t}_a = \mathfrak{t}_{a+1}}}
\sum_{\tuple{k} \in K^{\mathfrak{t}^*}} \psi_{\pi_\mathfrak{t}^{-1}} (\psi_{\pi_\mathfrak{t}(a)} Q_{\pi_\mathfrak{t}(a)}(\tuple{k}) - P_{\pi_\mathfrak{t}(a)}(\tuple{k}))  e(\tuple{k}) \psi_{\pi_\mathfrak{t}}
\\
= \sum_{\substack{\mathfrak{t} \in J^n \\ \mathfrak{t}_a = \mathfrak{t}_{a+1}}} \sum_{\tuple{k}^\mathfrak{t} \in K^\mathfrak{t}} (\psi_a Q_a(\tuple{k}^\mathfrak{t}) - P_a(\tuple{k}^\mathfrak{t}))e(\tuple{k}^\mathfrak{t}).
\end{multline}

We now focus on the second sum $S_2$; let $\mathfrak{t} \in J^n$ with $\mathfrak{t}_a \neq \mathfrak{t}_{a+1}$ and let $\tuple{k} \in K^{\mathfrak{t}^*}$. By Lemma~\ref{lemma:psi_pi-1_e(t)_psi_pisa} we have directly:
\[
\psi_{\pi_\mathfrak{t}^{-1}} e(\tuple{k}) \psi_{\pi_{s_a \cdot \mathfrak{t}}} = e(\tuple{k}^\mathfrak{t}) \psi_{\pi_\mathfrak{t}^{-1}} \psi_{\pi_{s_a \cdot \mathfrak{t}}}
=
e(\tuple{k}^\mathfrak{t}) \psi_a,
\]
thus:
\begin{align}
\sum_{\substack{\mathfrak{t} \in J^n \\ \mathfrak{t}_a \neq \mathfrak{t}_{a+1}}}  \sum_{\tuple{k} \in K^{\mathfrak{t}^*}} \sqrt{q}\psi_{\pi_\mathfrak{t}^{-1}} e(\tuple{k}) \psi_{\pi_{s_a \cdot \mathfrak{t}}}
&=
\sum_{\substack{\mathfrak{t} \in J^n \\ \mathfrak{t}_a \neq \mathfrak{t}_{a+1}}}  \sum_{\tuple{k}^\mathfrak{t} \in K^\mathfrak{t}}\sqrt{q} e(\tuple{k}^\mathfrak{t}) \psi_a
\notag
\\
&=
\sum_{\substack{\mathfrak{t} \in J^n \\ \mathfrak{t}_a \neq \mathfrak{t}_{a+1}}} \sum_{\tuple{k}^\mathfrak{t} \in K^\mathfrak{t}} \sqrt{q} \psi_a e(s_a \cdot \tuple{k}^\mathfrak{t})
\notag
\\
&=
\sum_{\substack{\mathfrak{t} \in J^n \\ \mathfrak{t}_a \neq \mathfrak{t}_{a+1}}} \sum_{\tuple{k}^{s_a \cdot \mathfrak{t}} \in K^{s_a \cdot \mathfrak{t}}} \sqrt{q} \psi_a e(\tuple{k}^{s_a \cdot \mathfrak{t}})
\notag
\\
&=
\sum_{\substack{\mathfrak{t} \in J^n \\ \mathfrak{t}_a \neq \mathfrak{t}_{a+1}}} \sum_{\tuple{k}^\mathfrak{t} \in K^\mathfrak{t}} \sqrt{q} \psi_a e(\tuple{k}^\mathfrak{t})
\notag
\\
\label{equation:image_ga_sigmaa<>}
\sum_{\substack{\mathfrak{t} \in J^n \\ \mathfrak{t}_a \neq \mathfrak{t}_{a+1}}}  \sum_{\tuple{k} \in K^{\mathfrak{t}^*}} \sqrt{q}\psi_{\pi_\mathfrak{t}^{-1}} e(\tuple{k}) \psi_{\pi_{s_a \cdot \mathfrak{t}}}
&=
\sum_{\substack{\mathfrak{t} \in J^n \\ \mathfrak{t}_a \neq \mathfrak{t}_{a+1}}} \sum_{\tuple{k}^\mathfrak{t} \in K^\mathfrak{t}} (Q_a(\tuple{k}^\mathfrak{t}) \psi_a - P_a(\tuple{k}^\mathfrak{t})) e(\tuple{k}^\mathfrak{t}).
\end{align}

Finally, by \eqref{equation:image_ga_beginning}--\eqref{equation:image_ga_sigmaa<>} we get:
\begin{align*}
\Phi_n^{\tuple{\Lambda}} \circ \mathrm{BK} \circ \mathrm{JPA}(g_a)
&=
\sum_{\substack{\mathfrak{t} \in J^n \\ \mathfrak{t}_a = \mathfrak{t}_{a+1}}} \sum_{\tuple{k} \in K^\mathfrak{t}} (\psi_a Q_a(\tuple{k}) - P_a(\tuple{k})) e(\tuple{k})
+
\sum_{\substack{\mathfrak{t} \in J^n \\ \mathfrak{t}_a \neq \mathfrak{t}_{a+1}}} \sum_{\tuple{k} \in K^\mathfrak{t}} (\psi_a Q_a(\tuple{k}) - P_a(\tuple{k})) e(\tuple{k})
\\
&=
\sum_{\mathfrak{t} \in J^n} \sum_{\tuple{k} \in K^\mathfrak{t}}
(\psi_a Q_a(\tuple{k}) - P_a(\tuple{k})) e(\tuple{k})
\\
\Phi_n^{\tuple{\Lambda}} \circ \mathrm{BK} \circ \mathrm{JPA}(g_a)
&=
\sum_{\tuple{k} \in K^n} (\psi_a Q_a(\tuple{k}) - P_a(\tuple{k})) e(\tuple{k}),
\end{align*}
which is \eqref{equation:image_BKtilde_ga}.

To conclude, we have checked that the algebra homomorphisms $\widetilde{\mathrm{BK}}$ and $\Phi_n^{\tuple{\Lambda}} \circ \mathrm{BK} \circ \mathrm{JPA}$ coincide on every generator of $\widehat{\Y}_{d, n}^{\tuple{\Lambda}}(q)$, hence we have the following theorem.

\begin{theoreme}
\label{theorem:commutative_diagram_JaPA}
We have
\[
\widetilde{\mathrm{BK}} = \Phi_n^{\tuple{\Lambda}} \circ \mathrm{BK} \circ \mathrm{JPA},
\]
and the diagram of Figure~\ref{figure:commutative_diagram} commutes.
\end{theoreme}


\begin{thebibliography}{1111111}
\bibitem[Ar]{Ar} S. \textsc{Ariki}, \textit{On the decomposition numbers of the Hecke algebra of $G(m, 1, n)$.} J. Math. Kyoto Univ. \textbf{36} (1996) 789--808.

\bibitem[ArKo]{ArKo} S. \textsc{Ariki} and K. \textsc{Koike}, \textit{A Hecke algebra of $\mathbb{Z}/n\mathbb{Z}\wr\mathfrak{S}_n$ and construction of its irreducible representations}. Adv. Math. \textbf{106} (1994) 216--243.

\bibitem[BoMa]{BoMa} C. \textsc{Boys} and A. \textsc{Mathas}, \textit{Quiver Hecke algebras for alternating groups}. arXiv:1602.07028v1.

\bibitem[BMR]{BMR} M. \textsc{Broué}, G. \textsc{Malle} and R. \textsc{Rouquier}, \textit{Complex reflection groups, braid groups, Hecke algebras}. J. Reine Angew. Math. \textbf{500} (1998) 127--190.

\bibitem[BrKl]{BrKl} J. \textsc{Brundan} and A. \textsc{Kleshchev}, \textit{Blocks of cyclotomic Hecke algebras and Khovanov--Lauda algebras}. Invent. Math. \textbf{178} (2009) 451--484.

\bibitem[BrKl2]{BrKl2} J. \textsc{Brundan} and A. \textsc{Kleshchev}, \textit{Graded decomposition numbers for cyclotomic Hecke algebras}. Adv. in Math. \textbf{222} (2009) 1883--1942.

\bibitem[CJKL]{CJKL} M. \textsc{Chlouveraki}, J. \textsc{Juyumaya}, K. \textsc{Karvounis} and S. \textsc{Lambropoulou}, \textit{Identifying the invariants for classical knots and links from the Yokonuma--Hecke algebra}. arXiv:1505.06666.

\bibitem[ChPou]{ChPou} M. \textsc{Chlouveraki} and G. \textsc{Pouchin}, \textit{Representation theory and an isomorphism theorem for the framisation of the Temperley--Lieb algebra}. arXiv:1503.03396v2.

\bibitem[ChPA]{ChPA} M. \textsc{Chlouveraki} and L. \textsc{Poulain d'Andecy}, \textit{Representation theory of the Yokonuma--Hecke algebra}. Adv. Math. \textbf{259} (2014) 134--172.

\bibitem[ChPA2]{ChPA2} M. \textsc{Chlouveraki} and L. \textsc{Poulain d'Andecy}, \textit{Markov traces on affine and cyclotomic Yokonuma--Hecke algebras}. Int. Math. Res. Notices (2015) rnv257.

\bibitem[CuWa]{CuWa}W. \textsc{Cui} and J. \textsc{Wan}, \textit{Modular representations and branching rules for affine and cyclotomic Yokonuma--Hecke algebras}. arXiv:1506.06570v2.

\bibitem[Dr]{Dr}V. G. \textsc{Drinfeld}, \textit{Degenerate affine Hecke algebras and Yangians}. Functional Anal. Appl. \textbf{20} (1986) 62--64.

\bibitem[GePf]{GePf} M. \textsc{Geck} and G. \textsc{Pfeiffer}, \textit{Characters of finite Coxeter groups and Iwahori-Hecke algebras}. London Math. Soc. Monographs, New Series \textbf{21}. Oxford University Press, New York (2000).

\bibitem[Hum]{Hum} J. E. \textsc{Humphreys}, \textit{Reflection groups and Coxeter groups}. Cambridge Stud. Avd. Math. \textbf{29}, Cambridge University Press (1990).

\bibitem[JaPA]{JaPA} N. \textsc{Jacon} and L. \textsc{Poulain d'Andecy}, \textit{An isomorphism Theorem for Yokonuma--Hecke algebras and applications to link invariants}. arXiv:1501.06389v2, to appear in Math. Z.


\bibitem[Ju1]{Ju1} J. \textsc{Juyumaya}, \textit{Sur les nouveaux générateurs de l'algèbre de Hecke $\H(G, U, 1)$}. J. Algebra \textbf{204} (1998) 49--68.

\bibitem[Ju2]{Ju2} J. \textsc{Juyumaya}, \textit{Markov traces on the Yokonuma--Hecke algebra}. J. Knot Theory Ramifications \textbf{13} (2004) 25--39.

\bibitem[JuKa]{JuKa} J. \textsc{Juyumaya} and S. \textsc{Kannan}, \textit{Braid relations in the Yokonuma--Hecke algebra}. J. Algebra \textbf{239} (2001) 272--297.


\bibitem[KanKa]{KanKa}S.-J. \textsc{Kang} and M. \textsc{Kashiwara}, \textit{Categorification of highest weight modules via Khovanov--Lauda--Rouquier algebras}. Invent math \textbf{190} (2012) 699--742.

\bibitem[KhLau1]{KhLau1}M. \textsc{Khovanov} and A. D. \textsc{Lauda}, \textit{A diagrammatic approach to categorification of quantum groups I}. Represent. Theory \textbf{13} (2009) 309--347.

\bibitem[KhLau2]{KhLau2}M. \textsc{Khovanov} and A. D. \textsc{Lauda}, \textit{A diagrammatic approach to categorification of quantum groups II}. Trans. Amer. Math. Soc. \textbf{363} (2011) 2685--2700.

\bibitem[LLT]{LLT} A. \textsc{Lascoux}, B. \textsc{Leclerc} and J.-Y. \textsc{Thibon}, \textit{Hecke algebras at roots of unity and crystal bases of quantum affine algebras}. Comm. Math. Phys. \textbf{183} (1996) 205--263.

\bibitem[Lu]{Lu}G. \textsc{Lusztig}, \textit{Character sheaves on disconnected groups} VII. Represent. Theory (electronic) \textbf{9} (2005) 209--266.

\bibitem[LyMa]{LyMa} S. \textsc{Lyle} and A. \textsc{Mathas}, \textit{Blocks of cyclotomic Hecke algebras}. Advances Math. \textbf{216} (2007) 854--878.

\bibitem[Ma]{Ma} A. \textsc{Mathas}, \textit{Iwahori--Hecke algebras and Schur algebras of the symmetric group}. University Lecture Series \textbf{15}, Amer. Math. Soc. (1999).

\bibitem[Ma2]{Ma2} A. \textsc{Mathas}, \textit{Cyclotomic quiver Hecke algebras of type A}, in Modular representation theory of finite and $p$-adic groups, G. W. Teck and K. M. Tan (ed.). National University of Singapore, Lecture Notes Series \textbf{30}, World Scientific (2015) ch. 5, 165--266.

\bibitem[OgPA]{OgPA}O. \textsc{Ogievetsky} and L. \textsc{Poulain d'Andecy}, \textit{Induced representations and traces for chains of affine and cyclotomic Hecke algebras}. J. Geom. Phys. \textbf{87} (2014) 354--372.

\bibitem[PA]{PA} L. \textsc{Poulain d'Andecy}, \textit{Invariants for links from classical and affine Yokonuma--Hecke algebras}. arXiv:1602.05429v1.

\bibitem[Rou]{Rou}R. \textsc{Rouquier}, \textit{2--Kac--Moody algebras}. arXiv:0812.5023v1.

\bibitem[RSVV]{RSVV} R. \textsc{Rouquier}, P. \textsc{Shan}, M. \textsc{Varagnolo} and E. \textsc{Vasserot}, \textit{Categorification and cyclotomic rational double affine Hecke algebras}. arXiv:1305.4456v2.

\bibitem[SVV]{SVV}P. \textsc{Shan}, M. \textsc{Varagnolo} and E. \textsc{Vasserot}, \textit{On the center of quiver-Hecke algebras}. arXiv:1411.4392v3.

\bibitem[Th1]{Th1}N. \textsc{Thiem}, \textit{Unipotent Hecke algebras: the structure, representation theory, and combinatorics.} Ph. D. Thesis, University of Wisconsin (2004).

\bibitem[Th2]{Th2}N. \textsc{Thiem}, \textit{Unipotent Hecke algebras of $\mathrm{GL}_n(\mathbb{F}_q)$.} J. Algebra \textbf{284} (2005) 559--577.

\bibitem[Th3]{Th3}N. \textsc{Thiem}, \textit{A skein-like multiplication algorithm for unipotent Heke algebras}. Trans. Amer. Math. Soc. \textbf{359}(4) (2007) 1685--1724.

\bibitem[Yo]{Yo}T. \textsc{Yokonuma}, \textit{Sur la structure des anneaux de Hecke d'un groupe de Chevalley fini}. C. R. Acad. Sci. Paris Ser. I Math. \textbf{264} (1967) 344--347.
\end{thebibliography}
\end{document}